\documentclass[reqno, 11pt]{amsart}
\usepackage{}
\usepackage{bbm}
\usepackage{url}
\usepackage{stmaryrd}
\usepackage{mathrsfs}
\usepackage{cases}
\usepackage{amsfonts}
\usepackage{amssymb}
\usepackage{amsmath}
\usepackage{enumerate}
\usepackage{tikz}
\usepackage{extarrows}

\usepackage{citeref}

\allowdisplaybreaks[4]
\newtheorem{theorem}{Theorem}[section]
\newtheorem{lemma}[theorem]{Lemma}
\newtheorem{proposition}[theorem]{Proposition}
\newtheorem{corollary}[theorem]{Corollary}
\theoremstyle{definition}
\newtheorem{definition}[theorem]{Definition}

\newtheorem{remark}[theorem]{Remark}
\numberwithin{equation}{section}

\let\al=\alpha
\let\b=\beta
\let\g=\gamma
\let\d=\delta

\let\z=\zeta
\let\la=\lambda
\let\r=\rho
\let\s=\sigma

\let\om=\omega
\let\G= \Gamma

\let\La=\Lambda

\let\Om=\Omega
\let\P=\Phi

\let\th=\theta

\let\va=\varphi

\let\fy=\infty
\def\bbR{\mathbb{R}}

\def\scrP{\mathscr{P}}

\def\calI{\mathcal {I}}
\def\calM{\mathcal {M}}

\def\calS{\mathcal {S}}
\def\calW{\mathcal {W}}

\def\scrF{\mathscr{F}}
\def\bbm{\mathbbm{m}}
\def\bbw{\mathbbm{w}}
\def\bbu{\mathbbm{u}}

\newcommand{\be}{\begin{equation*}}
\newcommand{\ee}{\end{equation*}}
\newcommand{\ben}{\begin{equation}}
\newcommand{\een}{\end{equation}}
\newcommand{\bn}{\begin{enumerate}}
\newcommand{\en}{\end{enumerate}}
\newcommand{\bs}{\backslash}

\def\Omab{\Omega_{a,b}}
\def\Omao{\Omega_{a,0}}

\def\Omai{\Omega_{a,i}}
\def\Omaj{\Omega_{a,j}}

\def\Omba{\Omega_{b,a}}
\def\Ombo{\Omega_{b,0}}
\def\Ombi{\Omega_{b,i}}
\def\Ombj{\Omega_{b,j}}

\def\rr{{\mathbb R}}
\def\rd{{{\rr}^d}}
\def\rdd{{{\rr}^{2d}}}

\def\rmd{{{\mathbb{R}}^{md}}}
\def\rmdd{{{\mathbb{R}}^{(m+1)d}}}
\def\zz{{\mathbb Z}}
\def\zd{{{\mathbb{Z}}^d}}
\def\zdd{{{\mathbb{Z}}^{2d}}}

\def\zmd{{{\mathbb{Z}}^{md}}}
\def\zmdd{{{\mathbb{Z}}^{(m+1)d}}}

\def\-fl1{\scrF^{-1}L^1}
\def\fl1{\scrF L^1}
\def\lan{\langle}
\def\ran{\rangle}

\def\scrL{\mathscr{L}}
\def\orw{\overrightarrow}

\begin{document}
\title[multilinear Rihaczek distributions on Wiener amalgam spaces]
{Characterization of boundedness on Wiener amalgam spaces of
multilinear Rihaczek distributions}
\author{WEICHAO GUO}
\address{School of Science, Jimei University, Xiamen, 361021, P.R.China}
\email{weichaoguomath@gmail.com}
\author{GUOPING ZHAO}
\address{School of Applied Mathematics, Xiamen University of Technology,
Xiamen, 361024, P.R.China} \email{guopingzhaomath@gmail.com}

\begin{abstract}
In this paper, we give several characterizations for the boundedness of multilinear
Rihaczek distributions
acting from Wiener amalgam spaces to modulation and Fourier modulation spaces.
Moreover, we establish the crucial self-improvement property which has its independent significance.
As applications, sharp exponents are established for the boundedness in several typical cases.
Correspondingly, the boundedness of pseudodifferential operators on Wiener amalgam spaces with
symbols in modulation and Fourier modulation spaces are also established.
In some typical cases, we also give the sharp exponents for the boundedness of pseudodifferential operators, including
the recapture and essential extensions of the main results in \cite[IMRN, (10):1860-1893, (2010)]{CorderoNicola2010IMRNI} and \cite[JFAA, 23(4):810-816, (2017)]{Cunanan2017JoFAaA}.
\end{abstract}

\maketitle


\section{INTRODUCTION}

A $m$-linear pseudodifferential operator with a symbol $\s\in \calS'(\rr^{(m+1)d})$ is defined by the formula
\ben\label{PIO}
K_{\s}(f_1,\cdots,f_m)(x)
=\int_{(\rd)^m}\s(x,\xi_1,\cdots,\xi_m)\prod_{j=1}^m\widehat{f_j}(\xi_j)e^{2\pi ix\cdot(\sum_{j=1}^m \xi_j)}d\xi_1\cdots d\xi_m,
\een
for $f_j\in \calS(\rd)$, $j=1,2,\cdots,m$, where $\widehat{f}$ denotes the Fourier transform of $f$.
The 1-linear operator is simply called linear operator, and 2-linear operator is called bilinear.
In particular, the linear version of \eqref{PIO} are well known as the Kohn-Nirenberg operator with symbol $\s$.

For $f_j,g\in \calS(\rd), j=1,\cdots,m$, the action of $K_{\s}$ can be expressed by the formula
\ben\label{rPR}
\langle K_{\sigma}\vec{f}, g\rangle
=
\langle K_{\s}(f_1,\cdots,f_m), g\rangle
=\langle \sigma, R_m(g,f_1,\cdots,f_m)\rangle
=\langle \sigma, R_m(g,\vec{f})\rangle,
\een
where
\be
R_m(g,\vec{f})(x,\vec{\xi})=g(x)\overline{\prod_{j=1}^m\widehat{f_j}(\xi_j)}e^{-2\pi ix\cdot(\sum_{j=1}^m \xi_j)}
\ee
is the multilinear Rihaczek distribution. We recall that for $m=1$, the above operator coincides with the
usual Rihaczek distribution.
From the duality relation \eqref{rPR}, one can see that there are close connections between the boundedness of pseudodifferential operators
and that of Rihaczek distributions, see also Propositions \ref{pp-eqPRM} and \ref{pp-eqPRF}.
In this paper, we first consider the boundedness property of Rihaczek distributions, and then
study the boundedness of pseudodifferential operators by the hands of the corresponding results of Rihaczek distributions.

A strong motivation for the study of pseudodifferential operators is provided by the fact that
pseudodifferential operators lie in the center of many deep results in the field of PDE.
See the pioneering works of Kohn-Nirenberg \cite{KohnNirenberg1965CPAM} and H\"{o}rmander \cite{Hoermander1965CPAM}.
Since then, with the development of PDE, many symbol classes have been studied to ensure the boundedness of the
corresponding pseudodifferential operators on certain function spaces.
Among them, an important symbol class is the famous H\"{o}rmander's class $S_{\r,\d}^m$, in which the symbol functions satisfy
certain smoothness and decay conditions associated with $m$, $\r$ and $\d$.
In particular, the $S_{0,0}^0$ class consists of those $\s$ satisfy the following estimates:
\be
|\partial^{\al}_x\partial^{\b}_{\xi}\s(x,\xi)|\leq C_{\al,\b}
\ee
for all multi-indices $\al,\b$.
The classical Calderon-Vaillancourt theorem \cite{CalderonVaillancourt1971JMSJ}
asserts the $L^2$-boundedness of Kohn-Nirenberg operator $K_{\s}$ with symbol $\s\in S_{0,0}^0$.

Let us mention that in the bilinear (or multi-linear) case  the analogue class of symbols satisfying
\be
|\partial^{\al}_x\partial^{\b}_{\xi}\partial^{\g}_{\eta}\s(x,\xi,\eta)|\leq C_{\al,\b,\g}
\ee
can not yield the expected boundedness from $L^2\times L^2$ into $L^1$, unless additional size conditions are imposed on the symbols, see \cite{BenyiTorres2004MRL}.

Limited by the techniques of so-called ``hard analysis'',
it is very difficult to remove the smoothness and decay conditions in the proof of boundedness of Kohn-Nirenberg operator.
However, the investigation of reducing the smoothness and decay conditions attracts a lot of attention of many researchers,
see \cite{Cordes1975JoFA, Kato1976OJOM, NagaseMichihiro1977CiPDE}.

In 1994, a significant progress was made by Sj\"{o}strand \cite{Sjoestrand1994MRL},
showing that the $L^2$ boundedness of $K_{\s}$ is also valid if the symbol $\s$ belongs to
a new symbol class (Sj\"{o}strand's class)  that contains some non-smooth symbols.
Then, the Sj\"{o}strand's class was recognized to be the modulation space $M^{\fy,1}(\rd)$.
By the inclusion relation $S_{0,0}^0\subsetneq M^{\fy,1}$, Sj\"{o}strand's result essentially extended the Calderon-Vaillancourt theorem.

Using the methods from time-frequency analysis,
Gr\"{o}chenig--Heil \cite{GroechenigHeil1999IEOT} and Gr\"{o}chenig \cite{Groechenig2006RMI} extended Sj\"{o}strand's result to the boundedness on all modulation spaces $M^{p,q}$ with $1\leq p,q\leq \fy$.
Due to the natural definition of $M^{\fy,1}$ by means of STFT, the methods in time-frequency are expected
to behave more naturally when dealing with the boundedness problems of pseudodifferential operators with symbols in $M^{\fy,1}$,
or in more general modulation spaces $M^{p,q}$.
We refer the reader to Toft \cite{Toft2004AGAG}, Cordero-Nicola \cite{CorderoNicola2018IMRNI} and Cordero \cite{Cordero2020a}
for the study of the boundedness on modulation spaces of pseudodifferential operators with symbols in modulation spaces.
For the boundedness on modulation spaces, some useful characterizations can be found in
a recent comprehensive work \cite{GuoChenFanZhao2022IMRN},
where the corresponding boundedness of $\tau$-Wigner distributions are also considered.
For the boundedness on modulation spaces in multi-linear setting, one can see \cite{BenyiGroechenigHeilOkoudjou2005JOT,BenyiOkoudjou2004JFAA,BenyiOkoudjou2006SM}, in which the time-frequency tools also play an important role in the proof.

Modulation space was first introduced by H. Feichtinger \cite{Feichtinger1983TRUoV} in 1983 and
has been studied extensively.
Now, the modulation spaces have turned out to be an important class of function spaces in the field of time-frequency analysis.
More precisely, modulation spaces
are defined by measuring the decay and integrability of the STFT as follows:
\be
M^{p,q}_m(\rd)=\{f\in \calS'(\rd): V_gf\in L^{p,q}_m(\rdd) \},
\ee
endowed with the obvious (quasi-)norm, where $L^{p,q}_m(\rdd)$ are weighted mixed-norm Lebesgue spaces with the weight $m$,
more details can be found in Section 2. By $\calM^{p,q}_m(\rd)$ we denote the $\calS(\rd)$ closure in $M^{p,q}_m(\rd)$.

Compared with the natural advantage of time-frequency tools in the boundedness problem on modulation spaces,
the boundedness on Lebesgue spaces $L^p(\rd)$ or more general Wiener amalgam spaces $W(L^p,L^q)(\rd)$
can not rely entirely on the time-frequency analysis.
An enlightening viewpoint is that $L^p(\rd)$ cannot be characterized by the decay of Gabor coefficients, unless the case $p=2$,
in which $L^2(\rd)$ is equivalent to the modulation space $M^{2,2}(\rd)$.
Therefore, in some sense, it is more challenging to study the boundedness on $L^p(\rd)$ or $W(L^p,L^q)(\rd)$.

For simplicity, we use BPWM to denote
the problem for
\emph{the boundedness of pseudodifferential operators on Wiener amalgam spaces with symbols belonging to modulation spaces}.
Correspondingly, by BRWM we denote
the problem for \emph{the boundedness of Rihaczek distribution acting from Wiener amalgam spaces to
modulation spaces}.

Denote by $Q_0=[-1/2,1/2]^d$ the unit cube centered at the origin.
$T_k$ stands for the translation operator.
We recall that the Wiener amalgam space $W(L^p,L^q)(\rd)$ consists of all measurable functions for which
 the following norm are finite:
\be
\|f\|_{W(L^p,L^q)(\rd)}:= \bigg(\sum_{k\in \zd}\|T_{k}\chi_{Q_0}f\|^q_{L^p(\rd)}\bigg)^{1/q},
\ee
with usual modification when $q=\fy$.
Denote by $\calW(L^p,L^q)(\rd)$ the $\calS(\rd)$ closure in $W(L^p,L^q)(\rd)$.

In \cite{CorderoNicola2010IMRNI}, the full range of exponents has been completely characterized for the following problem:
\ben\label{int_1}
\forall\sigma\in M^{p,q}(\rdd)
  \Longrightarrow
K_{\s}: W(L^{p_0},L^{q_0})(\rd)\longrightarrow W(L^{p_0},L^{q_0})(\rd),
\een
where $1\leq p_0,q_0,p,q\leq \fy$.
More precisely, in \cite{CorderoNicola2010IMRNI} the authors found that the sharp range of exponents
tha makes the boundedness \eqref{int_1} hold is
\be
\frac{1}{p}\geq \left|\frac{1}{p_0}-\frac{1}{2}\right|+\frac{1}{q'},\ \ \ q\leq p_0,p_0',q_0,q_0'.
\ee

In the present paper,
we consider the
BPWM problem on a more general framework.
As in \cite{CorderoNicola2010IMRNI}, to avoid the fact that $\calS(\rd)$ is not dense in some endpoint spaces, such as $M^{p,q}$ with $p=\fy$ or $q=\fy$,
we only consider the action of Rihaczek distribution on Schwartz function spaces.
For the sake of simplicity, we use the statement ``$R_m: M_0\times M_1\times \cdots\times M_m \rightarrow X$'' to express the meaning that
the $m$-linear Rihaczek distribution $R_m$, which first defined on $\calS(\rd)\times \cdots \times \calS(\rd)$, can be extended to a bounded operator from $\calM_0\times \calM_1\times \cdots\times \calM_m$ into $X$,
where $\calM_i$, $i=0,1,\cdots,m$ are the Schwartz closure function spaces considered in this paper, $X$ serves as the target function space.
A similar statement is also used for the boundedness of pseudodifferential operator.

For a suitable weight function $\Om$ on $\rr^{2(m+1)d}$, and weight functions $\mu_j$ on $\rd$, $j=0,1,2,\cdots,m$,  we consider the BPWM of the following type:
\be
\forall\sigma\in M^{p,q}_{\Om}(\rmdd)
  \Longrightarrow
  K_{\s}: W(L^{p_1},L^{q_1}_{\mu_1})(\rd)\times\cdots \times W(L^{p_m},L^{q_m}_{\mu_m})(\rd)\longrightarrow W(L^{p_0},L^{q_0}_{\mu_0})(\rd),
\ee
where $1\leq p,q,p_j,q_j\leq \fy$, $j=0,1,\cdots,m$.
Correspondingly, we consider the BRWM for the Rihaczek distribution:
\be
  R_m: W(L^{p_0},L^{q_0}_{\mu_0})(\rd)\times\cdots \times W(L^{p_m},L^{q_m}_{\mu_m})(\rd)\longrightarrow M^{p,q}_{\Om}(\rmdd)
\ee
with $p, q, p_j, q_j \in (0,\fy]$, $j=0,1,\cdots,m$.

Our first motivation is to give a ``natural'' characterization of BPWM and BRWM,
by using the common structure among differences, between the modulation spaces and Wiener amalgam spaces.
Before giving this characterization, we refer to Subsection 2.2 for the definition of weight class $\mathscr{P}(\rd)$, and Subsection 3.2 for the definition of the coordinate transform $\tau_m$.
We use $B_{\d}=B(0,\d)$ to denote the ball in $\rd$ centered at the origin with radius $\d$. 
To present a more concise and complete part of our results, we give the following theorem for the characterization of BRWM,
one can see Theorems \ref{thm-M1} and \ref{thm-msi} for the results with more general weights.  

\begin{theorem}[Characterization of BRWM]\label{thm-M1-sp}
	Assume $p_i, q_i, p, q \in (0,\fy]$,
	and that $\bbm(z_0,\vec{z})=\bbm_0(z_0)\bbm_1(z_1)\cdots \bbm_m(z_m)$,
	where $\bbm_i \in \mathscr{P}(\rd)$, $i=0,1,\cdots,m$.
	Denote $\bbm_{a,b}(z_0,\vec{\z})=\bbm_0(z_0)$ and $\bbm_{b,a}(z_0,\vec{z})=\bbm_1(z_1)\cdots \bbm_m(z_m)$.
	The following two statements are equivalent:
	\bn
	\item The Rihaczek distribution is bounded:
    \ben\label{thm-M1-sp-cd1}
    R_m: W(L^{p_0},L^{q_0}_{\mu_0})(\rd)\times\cdots \times W(L^{p_m},L^{q_m}_{\mu_m})(\rd)\longrightarrow M^{p,q}_{\bbm\otimes 1}(\rmdd).
    \een
	\item The following local and global boundedness is valid: 
		\ben\label{thm-M1-sp-cd2}
	R_m: L^{p_0}(B_{\d})\times\cdots \times L^{p_m}(B_{\d})\longrightarrow M^{p,q}_{\bbm_{b,a}\otimes 1}(\rmdd),
	\een
	for some $\d>0$,
	and
	\ben\label{thm-M1-sp-cd3}
	\tau_m\big(\otimes_{j=0}^m l^{q_j}_{\mu_j}(\zd)\big)\subset l^{p,q}_{\bbm_{a,b}}(\zd\times\zmd).
	\een
	\en
	More over, if $p_i,q_i\in (0,\fy)$, the above statements are also equivalent to the following two stronger statements:
	\bn[(i)]
	\item The Rihaczek distribution has stronger boundedness:
	   \ben\label{thm-M1-sp-cd4}
	R_m: W(L^{p_0\wedge 2},L^{q_0}_{\mu_0})(\rd)\times\cdots \times W(L^{p_m\wedge 2},L^{q_m}_{\mu_m})(\rd)\longrightarrow M^{p\wedge q,q}_{\bbm\otimes 1}(\rmdd).
	\een
	\item The following strong local and global boundedness is valid: 
	\ben\label{thm-M1-sp-cd5}
	R_m: L^{p_0\wedge 2}(B_{\d})\times\cdots \times L^{p_m\wedge 2}(B_{\d})\longrightarrow M^{p\wedge q,q}_{\bbm_{b,a}\otimes 1}(\rmdd),
	\een
	for some $\d>0$,
	and
	\ben\label{thm-M1-sp-cd6}
	\tau_m\big(\otimes_{j=0}^m l^{q_j}_{\mu_j}(\zd)\big)\subset l^{p\wedge q,q}_{\bbm_{a,b}}(\zd\times\zmd).
	\een
	\en
\end{theorem}

The proof of this theorem will be given directly by the more general conclusions in Theorems \ref{thm-M1} and \ref{thm-msi}.
See Remark \ref{rmk-M1}.
As an important application of Theorems \ref{thm-M1-sp}, we give the full range of exponents for the BRWM boundedness of unweighted version.
\begin{theorem}[Sharp exponents of BRWM]\label{thm-M2}
  Assume $p_i, q_i, p, q \in (0,\fy]$, $i=0,1,2,\cdots,m$.
  Denote by
  \be
  \La:=\bigg\{j: j=0,1,\cdots,m,\ \ \frac{1}{p}\geq 1-\frac{1}{p_j\wedge 2}\bigg\}.
  \ee
  We have
  \ben\label{thm-M2-cd0}
  R_m: W(L^{p_0},L^{q_0})(\rd)\times\cdots \times W(L^{p_m},L^{q_m})(\rd)\longrightarrow M^{p,q}(\rmdd)
  \een
if and only if
\ben\label{thm-M2-cd1}
  \frac{1}{q}\leq 1-\frac{1}{p_i\wedge 2},\ \ \ \ i=0,1,\cdots,m,
  \een
  \ben\label{thm-M2-cd2}
  \frac{|\La|-1}{p}+\frac{1}{q}\leq |\La|-\sum_{j\in \La} \frac{1}{p_j\wedge 2}\ \text{for}\ |\La|\geq 1,
  \een
  and
\ben\label{thm-M2-cd3}
1/q\leq 1/q_i,\ \ \ \ i=0,1\cdots m,
\een
\ben\label{thm-M2-cd4}
\frac{1}{p}+\frac{m}{q}\leq \sum_{j=0}^m \frac{1}{q_j}.
\een
\end{theorem}

In this paper, we also consider the Fourier modulation space, that is, 
the image of modulation space under the Fourier transform, see the next section for
its precise definition.
We use BPWF to denote
the problems for
\emph{the boundedness of pseudodifferential operators on Wiener amalgam spaces with symbols belonging to Fourier modulation spaces}.
Correspondingly, by BRWF we denote
the problem for \emph{the boundedness of Rihaczek distribution acting from weighted Wiener amalgam spaces to
	Fourier modulation spaces}.
We give the following theorem for the characterization of BRWF.
See Theorems \ref{thm-F1} and \ref{thm-fsi} for the corresponding results with more general weights.  

\begin{theorem}[Characterization of BRWF]\label{thm-F1-sp}
	Assume $p_i, q_i, p, q \in (0,\fy]$,
	and that $\bbw(z_0,\vec{z})=\bbw_0(z_0)\bbm_1(z_1)\cdots \bbw_m(z_m)$,
	where $\bbw_i \in \mathscr{P}(\rd)$, $i=0,1,\cdots,m$.
	The following statements are equivalent:
	\bn
	\item The Rihaczek distribution is bounded:
	\ben\label{thm-F1-sp-cd1}
	R_m: W(L^{p_0},L^{q_0}_{\mu_0})(\rd)\times\cdots \times W(L^{p_m},L^{q_m}_{\mu_m})(\rd)\longrightarrow \scrF M^{p,q}_{1\otimes \bbw}(\rmdd).
	\een
	\item The following embedding relations are valid: 
	\ben\label{thm-F1-sp-cd2}
	W(L^{p_0},L^{q_0}_{\mu_0})(\rd)\subset \scrF M^{p,q}_{1\otimes \bbw_0}(\rd),
	\een
	and
	\ben\label{thm-F1-sp-cd3}
	W(L^{p_i},L^{q_i}_{\mu_i})(\rd)\subset M^{p,q}_{1\otimes \bbw_i}(\rd),\ \ \ \ i=1,\cdots,m.
	\een
	\en
	More over, if $p_i,q_i\in (0,\fy)$, the above statements are also equivalent to the following two stronger statements:
		\bn[(i)]
	\item The Rihaczek distribution has stronger boundedness:
	\ben\label{thm-F1-sp-cd4}
	R_m: W(L^{p_0\wedge 2},L^{q_0}_{\mu_0})(\rd)\times\cdots \times W(L^{p_m\wedge 2},L^{q_m}_{\mu_m})(\rd)\longrightarrow \scrF M^{p,q}_{1\otimes \bbw}(\rmdd).
	\een
	\item The following stronger embedding relations are valid: 
	\ben\label{thm-F1-sp-cd5}
	W(L^{p_0\wedge 2},L^{q_0}_{\mu_0})(\rd)\subset \scrF M^{p,q}_{1\otimes \bbw_0}(\rd),
	\een
	and
	\ben\label{thm-F1-sp-cd6}
	W(L^{p_i\wedge 2},L^{q_i}_{\mu_i})(\rd)\subset M^{p,q}_{1\otimes \bbw_i}(\rd),\ \ \ \ i=1,\cdots,m.
	\een
	\en
\end{theorem}

The proof of this theorem will be given directly by the more general conclusions in Theorems \ref{thm-F1} and \ref{thm-fsi}.
See Remark \ref{rmk-F1}.
As an important application of Theorems \ref{thm-F1-sp}, we give the full range of exponents for the BRWF boundedness of unweighted version.

\begin{theorem}[Sharp exponents of BRWF]\label{thm-F2}
Assume $p_i, q_i, p, q \in (0,\fy]$, $i=0,1,\cdots,m$.
 We have
  \ben\label{thm-F2-cd0}
  R_m: W(L^{p_0},L^{q_0})(\rd)\times\cdots \times W(L^{p_m},L^{q_m})(\rd)\longrightarrow \scrF M^{p,q}(\rmdd)
  \een
if and only if
\ben\label{thm-F2-cd1}
\frac{1}{p}\leq 1-\frac{1}{p_0\wedge 2},\ \ \ \ \frac{1}{q}\leq \frac{1}{q_0}
\een
and
\ben\label{thm-F2-cd2}
\frac{1}{q}\leq 1-\frac{1}{p_i\wedge 2},\ \ \ \ \frac{1}{p},\frac{1}{q}\leq \frac{1}{q_i},\ \ \ \ i=1,2,\cdots,m.
\een
\end{theorem}

The rest of this paper is organized as follows. In Section 2, we recall some definitions
of function spaces used throughout this paper. We also list some basic time-frequency representations
associated with Rihaczek distribution, and recall the Gabor expansion of modulation spaces, which
will be frequently used in our proofs.

Section 3 is devoted to the first characterization of BRWM, namely, Theorem \ref{thm-M1},
in which BRWM under a general weight condition is characterized by 
the corresponding local and global boundedness.
To achieve our goal, we first deal with the local and global components in Subsections 3.1 and 3.2, respectively.
Subsection 3.3 is prepared for the discretization of BRWM in the time plane.
We give the proof of Theorem \ref{thm-M1} in Subsection 3.4.

Section 4 is devoted to the first characterization of BRWF, namely, Theorem \ref{thm-F1},
in which BRWF under a general weight condition is characterized by some 
corresponding embedding relations.
We deal with the embedding relations in Subsection 4.1, and give the
proof of Theorem \ref{thm-F1} in Subsection 4.2.

In Section 5, we focus on the self-improvement of BRWM and BRWF, namely, Theorems \ref{thm-msi} and \ref{thm-fsi}.
By establishing some relevant convolution inequalities,
as well as using the idea of discretization by means of the Gabor frame, and with the help of the Khinchin's inequality, 
we give Propositions \ref{pp-si} and \ref{pp-sif}, in which BRWM and BRWF can be improved step by step.
Then, the proofs of Theorems \ref{thm-msi} and \ref{thm-fsi} can be derived from Propositions \ref{pp-si} and \ref{pp-sif} respectively.
We also give the self-improvement of some embedding relations in Subsection 5.4.

Section 6 is used to deal with the unweighted case of BRWM and BRWF.
The sharp exponents of the local and global components of BRWM will be handled in Subsections 6.1 and 6.2.
Then by using Theorem \ref{thm-M1-sp}, we give the proof of Theorem  \ref{thm-M2} in Subsection 6.3.
The proof of Theorem \ref{thm-F2} will be given in Subsection 6.4.

In Section 7, we return to the boundedness of pseudodifferential operators. 
Using a dual argument, the boundedness of pseudodifferential operator follows directly by the corresponding results of Rihaczek distribution.
As an important application, when the symbol belongs to the Sj\"{o}strand's class,
we give the sharp exponents of the boundedness from Bessel potential Wiener amalgam space 
into another Wiener amalgam space.

\textbf{Notations:}
Throughout this paper, we will adopt the following notations. Let $C$ be a
positive constant that may depend on $m, d, p, q, p_{i},\,q_{i},\,\mu_{i},\,\Om$. The notation $X\lesssim Y$ denotes the
statement that $X\leq CY$, and the notation $X\sim Y$ means the statement $%
X\lesssim Y\lesssim X$.
The Schwartz function space is denoted by $\calS(\rd)$, and the space of tempered distributions by $\calS'(\rd)$.
We use the brackets $\langle f,g\rangle$ to denote
the extension to $\calS'(\rd)\times \calS(\rd)$ of the inner product $\langle f,g\rangle=\int_{\rd} f(x)\overline{g(x)}dx$ for $f,g\in L^2(\rd)$.
For $p\in (0,\fy]$, we write $\dot{p}=\min\{1,p\}$.

\section{PRELIMINARIES}
\subsection{Time-frequency representations}
We consider the point $(x,\xi)$ in the time-frequency plane $\rdd$,
where $x, \xi \in \rd$ denote the time  and frequency variables, respectively.
For any fixed $x, \xi$, the translation operator $T_x$ and modulation operator $M_{\xi}$ are defined, respectively, by
\be
T_xf(t)=f(t-x),\ \ \ \ M_{\xi}f(t)=e^{2\pi it\xi}f(t).
\ee
The short-time Fourier transform (STFT) of a function $f$ with respect to a window $g$ is defined by
\be
V_gf(x,\xi):=\int_{\rd}f(t)\overline{g(t-x)}e^{-2\pi it\cdot \xi}dt,\ \ \  f,g\in L^2(\rd).
\ee
Its extension to $\calS'\times \calS$ can be denoted by
\be
V_gf(x,\xi)=\langle f, M_{\xi}T_xg\rangle,
\ee
in which the STFT $V_gf$ is a bilinear map from $\calS'(\rd)\times \calS(\rd)$ into $\calS'(\rdd)$.
If $f\in \calS'(\rd)$ and $g\in \calS(\rd)$, $V_gf$ is a uniformly continuous function on $\rdd$ with polynomial growth,
see \cite[Theorem 11.2.3]{GrochenigBook2013}.
Following are some direct conclusions of the definition of STFT.
\begin{lemma}[Support property of STFT]\label{lm-spSTFT}
Suppose that both $f$ and $g$ have compact supports, we have
  \be
\{x\in \rd: \exists \xi\in \rd\ \text{such that}\ V_gf(x,\xi)\neq 0\}\subset \text{supp}f+\text{supp}g.
\ee
\end{lemma}
\begin{lemma}[Translation property of STFT]\label{lm-tlSTFT}
For any fixed $x_0$ we have
  \be
V_g(T_{x_0}f)(x,\xi)=e^{-2\pi ix_0\cdot\xi} V_g f(x-x_0,\xi).
\ee
\end{lemma}

\begin{lemma}[Fundamental identity of time-frequency analysis]
	The following identity is valid:
\be
V_gf(x,\xi)=e^{-2\pi ix\cdot \xi}V_{\hat{g}}\hat{f}(\xi,-x),\ \ \ (x,\xi)\in \rdd.
\ee	
\end{lemma}

In order to estimate the modulation norm of Rihaczek distribution, we
need the following calculation for the STFT.
One can also see \cite{BenyiGroechenigHeilOkoudjou2005JOT}.

\begin{lemma}[STFT of multilinear Rihaczek distribution]\label{lm-STFT-mRd}
Let $\Phi=R_m(\phi_0,\vec{\phi})$ for nonzero functions $\phi_j\in \calS(\rd)$, $j=0,1,\cdots,m$, $\vec{\phi}=(\phi_1,\cdots,\phi_m)$.
Then the STFT of $R_m(g,\vec{f})$ with respect to
the window $\Phi$ is given by
\be
\begin{split}
V_{\P}(R_m(g,\vec{f}))((z_0,\vec{z}),(\z_0,\vec{\z}))
=e^{-2\pi i\vec{z}\cdot \vec{\z}}
V_{\phi_0}g(z_0,\z_0+\sum_{j=1}^mz_j)
\overline{\prod_{j=1}^mV_{\phi_j}f_j(z_0+\z_j,z_j)}.
\end{split}
\ee
\end{lemma}

\subsection{Function spaces}
Firstly, we introduce the definitions of weights that will be used throughout this paper.
Recall that a weight is a positive and locally integral function on $\rd$.
The weights we consider in this paper are the moderate weights, which are suitable for the time-frequency estimates.
More precisely, a weight function $m$ is called $v$-moderate if there exists another weight function $v$ such that
\be
m(z_1+z_2)\leq Cv(z_1)m(z_2),\ \ \ \ z_1,z_2\in \rd,
\ee
where $v$ belongs to the class of submultiplicative weight, that is, $v$ satisfies
\be
v(z_1+z_2)\leq v(z_1)v(z_2),\ \ \ \ z_1,z_2\in \rd.
\ee
Moreover, in this paper,
we assume that $v$ has at most polynomial growth.
If the associated weight $v$ is implicit, we call that $m$ is moderate,
and use the notation $\scrP(\rd)$ to denote the cone of all non-negative functions which are moderate.
Similarly, $\scrP(\rmdd)$ denotes the same meaning in $\rmdd$.
Without loss of generality, we also assume that a $v$-moderate weight is continuous. We refer to \cite[Lemma 11.2.3]{Heil2008} for more details.

The following mixed-norm spaces are important for the estimates of STFT on the time-frequency plane.
\begin{definition}[Weighted mixed-norm spaces]
Let $m\in \scrP(\rdd)$, $p,q\in (0,\fy]$. Then the weighted mixed-norm space $L^{p,q}_m(\rdd)$
consists of all Lebesgue measurable functions on $\rdd$ such that the (quasi-)norm
\be
\|F\|_{L^{p,q}_m(\rdd)}=\left(\int_{\rd}\left(\int_{\rd}|F(x,\xi)|^pm(x,\xi)^pdx\right)^{q/p}d\xi\right)^{1/q}
\ee
is finite, with the usual modification when $p=\fy$ or $q=\fy$.
\end{definition}
Now, we introduce the definition of modulation space, which is served as our symbol class in the BPWM problem.
\begin{definition}\label{df-M}
Let $0<p,q\leq \infty$, $m\in \mathscr{P}(\rdd)$.
Given a non-zero window function $\phi\in \calS(\rd)$, the (weighted) modulation space $M^{p,q}_m(\rd)$ consists
of all $f\in \calS'(\rd)$ such that the norm
\be
\begin{split}
\|f\|_{M^{p,q}_m(\rd)}&:=\|V_{\phi}f(x,\xi)\|_{L^{p,q}_m(\rdd)}
=\left(\int_{\rd}\left(\int_{\rd}|V_{\phi}f(x,\xi)m(x,\xi)|^{p} dx\right)^{{q}/{p}}d\xi\right)^{{1}/{q}}
\end{split}
\ee
is finite.
\end{definition}
Note that the above definition of $M^{p,q}_m$ is independent of the choice of window function $\phi$
in the sense of equivalent norms.
We refer to \cite{GrochenigBook2013} for the case $(p,q)\in\lbrack 1,\infty ]^{2}$,
and \cite{GalperinSamarah2004ACHA} for full range $(p,q)\in (0,\infty ]^{2}$.
In particular, in order to deal with the case $p<1$ or $q<1$,
a suitable window class was found in \cite{GalperinSamarah2004ACHA}, denoted by $\mathfrak{M}^{p,q}_v$,
which depends on $p, q$.

If $m\equiv 1$, we will simply write $M^{p,q}(\rd)$ for the modulation space $M^{p,q}_m(\rd)$.
Denote by $\calM^{p,q}_m(\rd)$ the $\calS(\rd)$ closure in $M^{p,q}_m(\rd)$.
Recall that $\calM^{p,q}_m(\rd)=M^{p,q}_m(\rd)$ for $p,q\neq \fy$.

Next, we turn to the definition of Fourier modulation space $\scrF M^{p,q}_m(\rd)$.
Observe that
\be
\begin{split}
\|f\|_{\scrF M^{p,q}_m(\rd)}
= &
\|\scrF^{-1}f\|_{M^{p,q}_m(\rd)}
\\
= &
\|V_{\check{\phi}}\check{f}(x,\xi)\|_{L^{p,q}_m(\rdd)}
=
\|V_{\phi}f(\xi,-x)\|_{L^{p,q}_m(\rdd)}.
\end{split}
\ee
The Fourier modulation space can be also defined by the weighted mixed-norm of STFT.
\begin{definition}\label{df-W}
Let $0<p,q\leq \infty$, $m\in \mathscr{P}(\rdd)$.
Given a non-zero window function $\phi\in \calS(\rd)$, the (weighted) Fourier modulation space $\scrF M^{p,q}_m$ consists
of all $f\in \calS'(\rd)$ such that the norm
\be
\begin{split}
\|f\|_{\scrF M^{p,q}_m(\rd)}
=
\|V_{\phi}f(\xi,-x)\|_{L^{p,q}_m(\rdd)}
= \left(\int_{\rd}\left(\int_{\rd}|V_{\phi}f(\xi,-x)m(x,\xi)|^{p} dx\right)^{{q}/{p}}d\xi\right)^{{1}/{q}}
\end{split}
\ee
is finite.
\end{definition}

Next, we introduce the Wiener amalgam space.
In general, the Wiener amalgam space $W(B,C)$ with local component $B$ and global component $C$
consists of all tempered distributions $f$ which are locally in $B$ and globally in $C$.
With a wide variety of $B$ and $C$, the Wiener amalgam spaces cover many important function spaces.
For instance, if we take $B=\scrF L^p_{\bbw}(\rd)$ and $C=L^q_{\bbu}(\rd)$, the modulation space $M^{p,q}_{\bbw\otimes\bbu}(\rd)$
can be written by
$M^{p,q}_{\bbw\otimes\bbu}(\rd)=\scrF^{-1}W(\scrF L^p_{\bbw}, L^q_{\bbu})(\rd)$.
As an extension of Lebesgue spaces, the function spaces on which we consider the boundedness of pseudodifferential operators
are the special case of Wiener amalgam spaces, denoted by $W(L^p, L^q_{\mu})(\rd)$, where
$L^p(\rd)$ and $L^q_{\mu}(\rd)$ serves as the local and global component respectively.
For our convenience, we introduce a discrete norm of $W(L^p, L^q_{\mu})(\rd)$ with smooth cutoff functions.

First, we give a smooth partition of $\rd$.
Denote by $Q_{k}$ the unit cube with the center at $k\in \zd$. Then the family $\{Q_{k}\}_{k\in\zd}$
constitutes a decomposition of $\rd$.
Let $\rho \in \calS(\rd),$
$\rho: \rd \rightarrow [0,1]$ be a smooth function satisfying that $\rho(\xi)=1$ for
$\xi\in Q_0$ and $\rho(\xi)=0$ for $\xi\notin \frac{3}{2}Q_0$.
For any fixed $k\in \zd$, the translation of $\rho$ is defined by
\begin{equation}
\rho_{k}(\xi)=\rho(\xi-k).
\end{equation}
Since $\rho_{k}(\xi)=1$ in $Q_{k}$, we find that $\sum_{k\in\zd}\rho_{k}(\xi)\geq1$
for all $\xi\in\rd$. Define
\begin{equation}
\sigma_{k}(\xi):=\rho_{k}(\xi)\left(\sum_{l\in\zd}\rho_{l}(\xi)\right)^{-1},  ~~~~ k\in\zd.
\end{equation}
Then, $\{\sigma_{k}\}_{k\in\zd}$
constitutes a smooth partition of $\rd$, and $%
\sigma_{k}(\xi)=\sigma_0(\xi-k)$.
With this smooth partition of $\rd$, we give the definition of $W(L^p, L^q_{\mu})(\rd)$.
\begin{definition}
  Let $0<p,q\leq \infty$, $\mu\in \mathscr{P}(\rd)$.
The (weighted) Wiener amalgam space $W(L^p, L^q_{\mu})(\rd)$ consists
of all $f\in \calS'(\rd)$ such that the norm
\be
\begin{split}
\|f\|_{W(L^p, L^q_{\mu})(\rd)}
&:=
\left(\sum_{k\in \zd}\|\s_kf\|_{L^p}^q\mu(k)^q\right)^{1/q}
\end{split}
\ee
is finite, with usual modification when $q=\fy$.
\end{definition}

The discrete norm spaces play important roles not only in the Gabor analysis of modulation spaces,
but also in our characterizations of BPWM and BRWM.

\begin{definition}[Discrete norm spaces]
Let $0<p,q\leq \infty$, $\om\in \scrP(\rd)$.
The space $l^p_{\om}(\zd)$ consists of all $\vec{b}=\{b_{k}\}_{k\in \zd}$ for which the (quasi-)norm
\be
\|\vec{b}\|_{l^{p}_{\om}(\zd)}=\left(\sum_{k\in \zd}|b_{k}|^p\om(k)^p\right)^{1/p}
\ee
is finite, with the usual modification when $p=\fy$.
For $\om=v_s=\lan\xi\ran^s$, we write $l^p_{v_s}:= l^p_s$ for simplicity.
\end{definition}

\begin{definition}[Discrete mixed-norm spaces]
Let $0<p,q\leq \infty$, $m\in \scrP(\rdd)$.
The space $l^{p,q}_m(\zd)$ consists of all sequences $\vec{a}=\{a_{k,n}\}_{k,n\in \zd}$ for which the (quasi-)norm
\be
\|\vec{a}\|_{l^{p,q}_m(\zdd)}=\left(\sum_{n\in \zd}\left(\sum_{k\in \zd}|a_{k,n}|^pm(k,n)^p\right)^{q/p}\right)^{1/q}
\ee
is finite,  with the usual modification when $p=\fy$ or $q=\fy$.
\end{definition}

Finally, we recall an important tool from the probability theory, which is crucial when dealing with the self-improvement properties.

\begin{lemma}[Khinchin's inequality, see \cite{Gut2013}]\label{lm-ki}
  Let $0<p<\infty$, and $\{\omega_k\}_{k=1}^N$ be a sequence of independent random variables taking values $\pm 1$ with equal probability. 
  Denote the expectation (integral over the probability space) by $\mathbb{E}$.
  For any sequence of complex numbers $\{a_k\}_{k=1}^N$, we have
  \begin{equation}
    \mathbb{E}\bigg(|\sum_{k=1}^Na_k\omega_k|^p\bigg)
    \sim 
    \bigg(\sum_{k=1}^N|a_k|^2\bigg)^{\frac{p}{2}},
  \end{equation}
  where the implicit constants depend on $p$ only.
\end{lemma}

\subsection{Gabor analysis of modulation spaces}
Comparing with the classical definition of modulation space in Definition \ref{df-M},
or the semi-discrete definition such as in \cite[Proposition 2.1]{GuoChenFanZhao2019MMJ} similar to the style of Besov space,
the modulation space can be also characterized by the summability and decay properties of their Gabor coefficients,
this is an important reason why the modulation spaces play the central role in the field of time-frequency analysis.

We recall some important operators which are the key tools for the discretization of modulation spaces.

\begin{definition}
  Assume that $g,\g\in L^2(\rd)$ and $\al,\b>0$. The coefficient operator or analysis operator $C_g^{\al,\b}$
  is defined by
  \be
  C_g^{\al,\b}f=(\langle f, T_{\al k}M_{\b n}g\rangle)_{k,n\in \zd}.
  \ee
  The synthesis operator or reconstruction operator $D_{g}^{\al,\b}$ is defined by
  \be
  D_{\g}^{\al,\b}\vec{c}=\sum_{k\in \zd}\sum_{n\in \zd}c_{k,n}T_{\al k}M_{\b n}\g.
  \ee
  The Gabor frame operator $S_{g,\g}^{\al,\b}$ is defined by
  \be
  S_{g,\g}^{\al,\b}f=D_{\g}^{\al,\b}C_g^{\al,\b}f=\sum_{k\in \zd}\sum_{n\in \zd}\langle f, T_{\al k}M_{\b n}g\rangle T_{\al k}M_{\b n}\g.
  \ee
\end{definition}

In order to extend the boundedness result of analysis operator and synthesis operator to the modulation spaces of full range,
the following admissible window class was introduced in \cite{GalperinSamarah2004ACHA}.

\begin{definition}[The space of admissible windows]\label{df-space-windows}
Assume $0<p,q\leq \fy$ and that $m$ is $v$-moderate. Let $r=\min\{1,p\}$ and $s=\min\{1,p,q\}$.
For $r_1,s_1>0$, denote
\be
\om_{r_1,s_1}(x,\om)=v(x,\om)\cdot (1+|x|)^{r_1}\cdot(1+|\om|)^{s_1}.
\ee
Define the space of admissible windows $\mathfrak{M}^{p,q}_v$ for the modulation space $M^{p,q}_m$ to be
\be
\mathfrak{M}^{p,q}_v=\bigcup_{\substack {r_1>d/r \\ s_1>d/s\\1\leq p_1<\fy}}M^{p_1}_{\om_{r_1,s_1}}.
\ee
\end{definition}

Based on the window class mentioned above, we recall the boundedness of $C_g^{\al,\b}$ and $D_{g}^{\al,\b}$,
which works on the full range $p,q\in (0,\fy]$. See \cite{GalperinSamarah2004ACHA} for more details.

\begin{lemma}\label{lm, bdCD}
  Assume that $m$ is $v$-moderate, $p,q\in (0,\fy]$, and $g$ belongs to the subclass $M^{p_1}_{\om_{r_1,s_1}}$ of $\mathfrak{M}^{p,q}_v$.
  Then the analysis operator $C_g^{\al,\b}$ is boundedness from $M^{p,q}_m$ into $l^{p,q}_{\tilde{m}}$,
  and the synthesis operator $D_g^{\al,\b}$ is boundedness from $l^{p,q}_{\tilde{m}}$ into $M^{p,q}_m$ for all $\al,\b>0$,
  where $\tilde{m}(k,n)=m(\al k,\b n)$.
\end{lemma}

Now, we recall the main theorem in \cite{GalperinSamarah2004ACHA}, which extends the Gabor expansion of modulation spaces to the full range $0<p,q\leq \fy$.

\begin{theorem}[see \cite{GalperinSamarah2004ACHA}]\label{thm, frame for Mpq}
  Assume that $m$ is $v$-moderate, $p,q\in (0,\fy]$, $g,\g\in \mathfrak{M}^{p,q}_v$, and that the Gabor frame operator
  $S^{\al,\b}_{g,\g}=D^{\al,\b}_{\g}C^{\al,\b}_g=I$ on $L^2(\rd)$. Then
  \be
  f=\sum_{k\in \zd}\sum_{n\in \zd}\langle f, T_{\al k}M_{\b n}g\rangle T_{\al k}M_{\b n}\g
  =\sum_{k\in \zd}\sum_{n\in \zd}\langle f, T_{\al k}M_{\b n}\g\rangle T_{\al k}M_{\b n}g
  \ee
  with unconditional convergence in $M^{p,q}_m$ if $p,q<\fy$, and weak-star convergence in $M^{\fy}_{1/v}$ otherwise.
  Furthermore there are constants $A,B>0$ such that for all $f\in M^{p,q}_m$
  \be
  A\|f\|_{M^{p,q}_m}
  \leq
  \left(\sum_{n\in \zd}\left(\sum_{k\in \zd}|\langle f, T_{\al k}M_{\b n}g\rangle|^pm(\al k,\b n)^p \right)^{q/p}\right)^{1/q}
  \leq
  B\|f\|_{M^{p,q}_m}
  \ee
  with obvious modification for $p=\fy$ or $q=\fy$.
  Likewise, the quasi-norm equivalence
    \be
  A'\|f\|_{M^{p,q}_m}
  \leq
  \left(\sum_{n\in \zd}\left(\sum_{k\in \zd}|\langle f, T_{\al k}M_{\b n}\g\rangle|^pm(\al k,\b n)^p \right)^{q/p}\right)^{1/q}
  \leq
  B'\|f\|_{M^{p,q}_m}
  \ee
  holds on $M^{p,q}_m$.
  \end{theorem}

  The following well known theorem provides a way to find the Gabor frame of $L^2(\rd)$.
  Recall that $\|g\|_{W(L^{\fy},L^{1})(\rd)}=\sum_{n\in \zd}\|g\chi_{\mathcal{Q}_n}\|_{L^{\fy}}$ with $\mathcal{Q}_n=n+[0,1]^d$.
\begin{theorem}[Walnut \cite{Walnut1992JMAA}]\label{thm-frame-L^2}
  Suppose that $g\in W(L^{\fy},L^{1})(\rd)$ satisfies
  \be
  A\leq \sum_{k\in \zd}|g(x-\al k)|^2\leq B\ \ \ \ a.e.
  \ee
  for constants $A,B\in (0,\fy)$.
  Then there exists a constant $\b_0$ depending on $\al$ such that
  $\mathcal {G}(g,\al,\b):= \{T_{\al k}M_{\b n}g\}_{k,n\in \zd}$ is a Gabor frame of $L^2(\rd)$ for all $\b\leq \b_0$.
\end{theorem}

In order to find the dual window in a suitable function space, the following result is important.
\begin{theorem}[see \cite{GrochenigBook2013}]\label{thm-frame-invertible}
  Assume $g\in M^1_v(\rd)$ and that $\{T_{\al k}M_{\b n}g\}_{k,n\in \zd}$ is a Gabor frame for $L^2(\rd)$.
  Then the Gabor frame operator $S_{g,g}^{\al,\b}$ is invertible on $M^1_v(\rd)$. As a consequence, $S_{g,g}^{\al,\b}$
  is invertible on all modulation spaces $M^{p,q}_m(\rd)$ for $1\leq p,q\leq \fy$ and $m\in \scrP(\rdd)$.
\end{theorem}

In the applications of Gabor characterization of modulation space, we prefer choosing more specific $\al$ and $\b$ for convenience.
\begin{corollary}\label{cy-eqn}
  Suppose that $0<p,q\leq \fy$, $\om \in \scrP(\rdd)$.
  Let $\phi\in \calS(\rd)\bs \{0\}$, there exists a sufficiently large constant $N\in \zz^+$ such that
  \be
  \|f\|_{M^{p,q}_{\om}(\rd)}\sim \bigg\|V_{\phi}\Big(\frac{k_0}{N},\frac{n_0}{N}\Big)\om\Big(\frac{k_0}{N},\frac{n_0}{N}\Big)\bigg\|_{l^{p,q}(\zd\times \zd)}.
  \ee
\end{corollary}
\begin{proof}
There exists a sufficiently large integer $N_1$ such that for suitable positive constants $A, B$ we have
\be
A\leq \sum_{k\in \zd}|\phi(x-k/N_1)|^2\leq B.
\ee
Denote $\al=\frac{1}{N_1}$.
Using Theorem \ref{thm-frame-L^2}, there exists a constant $\b=\al/N_2=\frac{1}{N_1N_2}=\frac{1}{N}$ with sufficiently large integer $N_2$ such that
$\mathcal {G}(\phi,\al,\b)$ is a Gabor frame of $L^2(\rd)$.
By the definition of $L^2$ frame, we obtain that $\mathcal {G}(\phi,\b,\b)=\mathcal {G}(\phi,1/N,1/N)$ is also a Gabor frame of $L^2(\rd)$.
Let $\psi=(S_{\phi,\phi}^{\b,\b})^{-1}\phi$ be the canonical dual widow of $\phi$.
Note that $\phi\in \calS\subset \mathfrak{M}^{p,q}_v$, then Definition \ref{df-space-windows} and Theorem \ref{thm-frame-invertible} imply that
$\psi\in \mathfrak{M}^{p,q}_v$.
By the definitions of $\phi$ and $\psi$, we have
$S_{\phi,\psi}^{\b,\b}=D_{\psi}^{\b,\b}C_{\phi}^{\b,\b}=I$ on $L^2(\rd)$.
Then, the desired conclusion follows by Theorem \ref{thm, frame for Mpq}.
\end{proof}

\section{First characterizations of BRWM: decomposition in the time plane}
The content of this section is to characterize BRWM by the corresponding local and global boundedness. 
For the completeness and generality of the conclusions, 
we handle the problem under more general conditions, although it will bring higher complexity.
Let $\Om\in \mathscr{P}(\rr^{2(m+1)d})$, we give some notations and conventions as follows.
\bn
\item
$\Omab(z_0,\vec{\z})=\Om((z_0,\vec{0}),(0,\vec{\z}))$.
\item
$\Omao(\xi)=\Omab(\xi,(-\xi,\cdots,-\xi))$.
\item
$\Omai(\xi)=\Omab(0,(\underbrace{0,\cdots,\xi,0,\cdots,0}_{\xi\ \text{is the}\ ith\  vector}))$.
\ \ \ \ $i=1,2,\cdots,m$,
\item
$\Omba((z_0,\vec{z}),(\z_0,\vec{\z}))= \Om((0,\vec{z}),(\z_0,\vec{0}))$.
\item
$\Ombo(\xi)=\Omba((0,\vec{0}),(\xi,\vec{0}))$.
\item
$\Ombi(\xi)=\Omba((0,(\underbrace{0,\cdots,\xi,0,\cdots,0}_{\xi\ \text{is the}\ ith\  vector})),(-\xi,\vec{0}))$,\ \ \ \ $i=1,2,\cdots,m$.
\en

\bn
\item[M0.]
$\Om((z_0,\vec{z}),(\z_0,\vec{\z}))\lesssim \Om((z_0,\vec{0}),(0,\vec{\z}))\Om((0,\vec{z}),(\z_0,\vec{0}))$.
\item[M1.]
$\Om((z_0,\vec{0}),(0,\vec{\z}))
\lesssim
\Om((z_0,\vec{0}),(0,(-z_0,\cdots,-z_0)))\prod_{j=1}^m \Om((0,\vec{0}),(0,(\underbrace{0,\cdots,\z_j+z_0,0,\cdots,0}_{\z_j+z_0 \text{ is the}\ jth\  vector})))$.
\item[M2.]
$\Om((0,\vec{z}),(\z_0,\vec{0}))
\lesssim
\Om((0,\vec{0}),(\z_0+\sum_{j=1}^mz_j,\vec{0}))
\prod_{j=1}^m\Om((0,(\underbrace{0,\cdots,z_j,0,\cdots,0}_{z_j\ \text{is the}\ jth\  vector})),(-z_j,\vec{0}))$.
\en

\begin{theorem}[First characterization of BRWM]\label{thm-M1}
	Assume $p_i, q_i, p, q \in (0,\fy]$,
	and that $\Om\in \mathscr{P}(\rr^{2(m+1)d})$, $\mu_i \in \mathscr{P}(\rd)$, $i=0,1,\cdots,m$.
	We have
	\ben\label{thm-M1-cd0}
	R_m: W(L^{p_0},L^{q_0}_{\mu_0})(\rd)\times\cdots \times W(L^{p_m},L^{q_m}_{\mu_m})(\rd)\longrightarrow M^{p,q}_{\Om}(\rmdd)
	\een
	implies
	\ben\label{thm-M1-cd1}
	R_m: L^{p_0}(B_{\d})\times\cdots \times L^{p_m}(B_{\d})\longrightarrow M^{p,q}_{\Omba}(\rmdd),
	\een
	for some $\d>0$,
	and
	\ben\label{thm-M1-cd3}
	\tau_m\big(\otimes_{j=0}^m l^{q_j}_{\mu_j}(\zd)\big)\subset l^{p,q}_{\Omab}(\zd\times\zmd).
	\een
	
	For $p\leq q$, if $\Om$ satisfies condition M0, the converse direction is valid.
	In this case, we have the equivalent relation $\eqref{thm-M1-cd1}, \eqref{thm-M1-cd3}\Longleftrightarrow \eqref{thm-M1-cd0}$.
	
	For $p>q$, if
	$\Om$ satisfies conditions M0 and M1,
	we also have the equivalent relation $\eqref{thm-M1-cd1}, \eqref{thm-M1-cd3}\Longleftrightarrow \eqref{thm-M1-cd0}$.

	Moreover, the local boundedness \eqref{thm-M1-cd1} implies the following embedding relations:
	\ben\label{thm-M1-cd2}
	L^{p_i}(B_{\d})\subset \scrF^{-1}L^q_{\Ombi}(\rd),\ \ \ i=0,1,\cdots,m,
	\een
	which further implies
	\ben\label{thm-M1-cdp}
	p_i\geq 1,\ \ \ \Ombi\lesssim 1, \ \ \ \ \text{for all}\ \   0\leq i\leq m.
	\een
	The embedding relation \eqref{thm-M1-cd3} implies the following embedding relations:
	\ben\label{thm-M1-cd4}
	l^{q_i}_{\mu_i}(\zd)\subset  l^{q}_{\Omai}(\zd),\ i=0,1,\cdots,m.
	\een
	For $p\geq q$,
	the equivalent relation
	$\eqref{thm-M1-cd1}\Longleftrightarrow \eqref{thm-M1-cd2}$ is valid if $\Om$ satisfies condition M2,
	and
	the equivalent relation  $\eqref{thm-M1-cd3}\Longleftrightarrow \eqref{thm-M1-cd4}$ is valid if $\Om$ satisfies condition M1.
\end{theorem}

\begin{remark}\label{rmk-M1}
	Let $\bbm\in \mathscr{P}(\rr^{(m+1)d})$ be a variables separated weight.
	We point out that the weight function $\Om=\bbm\otimes 1\in \mathscr{P}(\rr^{2(m+1)d})$ in Theorem \ref{thm-M1-sp}
	 satisfies all the conditions $Mi$, $i=0,1,2$, mentioned in Theorem \ref{thm-M1}. Using this fact and Theorem \ref{thm-msi}, Theorem \ref{thm-M1-sp} can be directly proved.
\end{remark}

\subsection{Local boundedness of BRWM}
We first recall a local property of modulation space.
\begin{lemma}[Local property of modulation space I]\label{lm-lpm}
Let $0<p,q\leq \fy$, $\Om\in \scrP(\rdd)$.
For any $f$ supported on $B(0,R)$ with $R>0$, we have
\be
\|f\|_{M^{p,q}_{\Om}(\rd)}\sim_R \|f\|_{\scrF^{-1}L^q_{\Om_{0}}(\rd)},
\ee
where $\Om_{0}(\xi)=\Om(0,\xi)$ for $\xi\in \rd$.
\end{lemma}
\begin{proof}
  Let $\phi$ be a smooth real-valued function supported on $B(0,2R)$ with $\phi=1$ on $B(0,R)$. There exists a sufficiently small $\al$ such that
  \be
  \begin{split}
    \|f\|_{M^{p,q}_{\Om}(\rd)}
    \sim &
    \|V_{\phi}f(\al k,\al n)\Om(\al k,\al n)\|_{l^{p,q}(\zd\times \zd)}
    \\
    = &
    \bigg(\sum_{n\in \zd}\big(\sum_{k\in \zd}|V_{\phi}f(\al k,\al n)\Om(\al k,\al n)|^p\big)^{q/p}\bigg)^{1/q}
    \\
    \sim_R &
    \sum_{k\in \zd}\big(\sum_{n\in \zd}|V_{\phi}f(\al k,\al n)\Om_{0}(\al n)|^q\big)^{1/q},
    \end{split}
  \ee
  where in the last term we use the facts that
 only a finite number of $k$  make the term $V_{\phi}f(\al k,\al n)$ nonzero,
 and that for these $k$ we have $\Om(\al k,\al n)\sim \Om_{0}(\al n)$.
 By the definition of STFT,
 \be
 \begin{split}
 \bigg(\sum_{n\in \zd}|V_{\phi}f(\al k,\al n)\Om_{0}(\al n)|^q\bigg)^{1/q}
 \sim
 \bigg(\sum_{n\in \zd}|\scrF(fT_{\al k}\phi)(\al n)\Om_{0}(\al n)|^q\bigg)^{1/q}.
 \end{split}
 \ee
Note that $\text{supp}(fT_{\al k}\phi)\subset B(0,R)$. For sufficiently small $\al$ we have
\be
\bigg(\sum_{n\in \zd}|\scrF(fT_{\al k}\phi)(\al n)\Om_{0}(\al n)|^q\bigg)^{1/q}
\sim
\bigg(\int_{\rd}|\scrF(fT_{\al k}\phi)(\xi)\Om_{0}(\xi)|^qd\xi\bigg)^{1/q}
\sim
\|fT_{\al k}\phi\|_{\scrF^{-1}L^q_{\Om_{0}}},
\ee
where we use the sampling property of $\scrF^{-1}L^q_{\Om_{0}}$ for the functions with compact support on $B(0,R)$,
one can see \cite[Proposition 3.1]{GuoChenFanZhao2019MMJ} for more details.

For above estimates, we conclude that
\be
  \begin{split}
    \|f\|_{M^{p,q}_{\Om}(\rd)}
    \sim_R &
    \sum_{k\in \zd}\|fT_{\al k}\phi\|_{\scrF^{-1}L^q_{\Om_{0}}}
    \geq
    \|f\phi\|_{\scrF^{-1}L^q_{\Om_{0}}}=\|f\|_{\scrF^{-1}L^q_{\Om_{0}}}.
   \end{split}
\ee
On the other hand, by a convolution inequality (see \cite[Lemma 2.2]{GuoChenFanZhao2019MMJ}) with $\dot{q}=\min\{q,1\}$, we deduce that
\be
\begin{split}
  \|fT_{\al k}\phi\|_{\scrF^{-1}L^q_{\Om_{0}}}
  \lesssim &
  \|f\|_{\scrF^{-1}L^q_{\Om_{0}}}
  \|T_{\al k}\phi\|_{\scrF^{-1}L^{\dot{q}}_{v}}
  \\
  = &
  \|f\|_{\scrF^{-1}L^q_{\Om_{0}}}
  \|\phi\|_{\scrF^{-1}L^{\dot{q}}_{v}}
  \lesssim
  \|f\|_{\scrF^{-1}L^q_{\Om_{0}}}.
\end{split}
\ee
From this, we conclude that
\be
  \begin{split}
    \|f\|_{M^{p,q}_{\Om}(\rd)}
    \sim_R &
    \sum_{k\in \zd}\|fT_{\al k}\phi\|_{\scrF^{-1}L^q_{\Om_{0}}}
    \lesssim
    \|f\|_{\scrF^{-1}L^q_{\Om_{0}}},
   \end{split}
\ee
where in the last inequality we use the facts that
 only a finite number of $k$  make the term $\|fT_{\al k}\phi\|_{\scrF^{-1}L^q_{\Om_{0}}}$ nonzero.

\end{proof}

\begin{lemma}\label{lm-lbeq}
  Let $0<p,q,p_j\leq \fy$ for $j=0,1,\cdots,m$, $\Om\in \scrP(\rr^{2(m+1)d})$.
 Then the local boundedness
  \ben\label{lm-lbeq-cd1}
  R_m: L^{p_0}(B_{\d})\times\cdots \times L^{p_m}(B_{\d})\longrightarrow M^{p,q}_{\Om}(\rmdd)
  \een
  is equivalent to
  \ben\label{lm-lbeq-cd2}
  R_m: L^{p_0}(B_{\d})\times\cdots \times L^{p_m}(B_{\d})\longrightarrow M^{p,q}_{\Omba}(\rmdd),
  \een
  which implies the following embedding relations
  \ben\label{lm-lbeq-cd3}
  L^{p_i}(B_{\d})\subset M^{q,q}_{1\otimes\Om_{b,i}}(\rd),\ \ \ i=0,1,2,\cdots,m,
  \een
  where \eqref{lm-lbeq-cd3} is equivalent to
  \ben\label{lm-lbeq-cd4}
  L^{p_i}(B_{\d})\subset \scrF^{-1}L^q_{\Ombi}(\rd),\ \ \ i=0,1,2,\cdots,m.
  \een
  Moreover, for $p\geq q$, if $\Om$ satisfies condition $M2$,
  the opposite direction is also valid. In this case, we have
  the equivalent relation $\eqref{lm-lbeq-cd1}\Longleftrightarrow \eqref{lm-lbeq-cd2}\Longleftrightarrow \eqref{lm-lbeq-cd3}\Longleftrightarrow \eqref{lm-lbeq-cd4}$.
\end{lemma}
\begin{proof}
  Without loss of generality, we only consider the case $\d<1/2$.
  First, let us verify $\eqref{lm-lbeq-cd1}\Longleftrightarrow \eqref{lm-lbeq-cd2}$.
  Take $\Phi=R_m(\phi,\cdots,\phi)$,
where $\phi$ is a smooth function supported in $B_{2\d}$, satisfying $\phi=1$ on $B_{\d}$.
Using Lemma \ref{lm-STFT-mRd} and Lemma \ref{lm-spSTFT},
for smooth functions $f_j$ supported on $B(0,\d)$, $j=0,1,\cdots,m$, 
the STFT of $R_m(f_0,\vec{f})$ associated with window $\Phi$ can be written as
\be
\begin{split}
  &\left|V_{\P}(R_m(f_0,\vec{f}))((z_0,\vec{z}),(\z_0,\vec{\z}))\right|
\\
= &
\bigg|
V_{\phi}f_0(z_0,\z_0+\sum_{j=1}^mz_j)\prod_{j=1}^mV_{\phi}f_j(z_0+\z_j,z_j)\bigg|
\\
= &
\bigg|V_{\phi}f_0(z_0,\z_0+\sum_{j=1}^mz_j)\chi_{B(0,3\d)}(z_0)\prod_{j=1}^mV_{\phi}f_j(z_0+\z_j,z_j)\chi_{B(0,6\d)}(\z_j)\bigg|
\\
= &
\bigg|V_{\P}(R_m(f_0,\vec{f}))((z_0,\vec{z}),(\z_0,\vec{\z}))\chi_{B(0,3\d)}(z_0)\prod_{j=1}^m\chi_{B(0,6\d)}(\z_j)\bigg|.
\end{split}
\ee
Observe that
$\Omba((z_0,\vec{z}),(\z_0,\vec{\z}))= \Om((0,\vec{z}),(\z_0,\vec{0}))
\sim \Om((z_0,\vec{z}),(\z_0,\vec{\z}))$
for $z_0\in B_{3\d}$, $\z_j\in B_{6\d}$, $j=1,2,\cdots,m$.
Then,
\be
\begin{split}
  \|R_m(f_0,\vec{f})\|_{M^{p,q}_{\Om}(\rmdd)}
  = &
  \left\|V_{\P}(R_m(f_0,\vec{f}))((z_0,\vec{z}),(\z_0,\vec{\z}))\right\|_{L^{p,q}_{\Om}(\rmdd\times \rmdd)}
  \\
  \sim &
  \left\|V_{\P}(R_m(f_0,\vec{f}))((z_0,\vec{z}),(\z_0,\vec{\z}))\right\|_{L^{p,q}_{\Omba}(\rmdd\times \rmdd)}
  \\
  = &
  \|R_m(f_0,\vec{f})\|_{M^{p,q}_{\Omba}(\rmdd)}.
\end{split}
\ee
The above relation implies that \eqref{lm-lbeq-cd1} is equivalent to \eqref{lm-lbeq-cd2}.

We turn to verify $\eqref{lm-lbeq-cd2}\Longrightarrow \eqref{lm-lbeq-cd3}$.
By the definition of modulation space and the sampling property of STFT (see Lemma \ref{lm, bdCD}), we obtain that
\ben\label{lm-lbeq-0}
\begin{split}
  &\|R_m(f_0,\vec{f})\|_{M^{p,q}_{\Omba}(\rmdd)}
  \sim
  \|V_{\Phi}R_m(f_0,\vec{f})\|_{L^{p,q}_{\Omba}(\rmdd\times \rmdd)}
  \\
  \gtrsim &
  \|V_{\Phi}R_m(f_0,\vec{f})((z_0,\vec{z}),(\z_0,\vec{\z}))\Omba((z_0,\vec{z}),(\z_0,\vec{\z}))|_{\al\zmdd\times \al\zmdd}\|_{l^{p,q}}
  \\
  \sim &
  \|V_{\phi}f_0(\al k_0,\al(n_0+\sum_{j=1}^mk_j))
  \prod_{j=1}^mV_{\phi}f_j(\al(k_0+n_j),\al k_j)\Omba((0,\al \vec{k}),(\al n_0,\vec{0}))\|_{l^{p,q}}.
\end{split}
\een
Take $f_j=h$ for $j=1,2,\cdots,m$, where
$h$ is a nonnegative smooth function supported on $B(0,\d)$ with $\|h\|_{L^1}=1$.
We have
\be
V_{\phi}f_j(0,0)=
V_{\phi}h(0,0)=
\int_{\rd}h(y)\phi(y)dy
=\int_{\rd}h(y)dy=1,\ \ \ j=1,2,\cdots,m.
\ee
Then, the last term in \eqref{lm-lbeq-0} can be dominated from below by
\be
\begin{split}
&\bigg\|\bigg(\bigg(\sum_{k_0\in \zd}\bigg|V_{\phi}f_0(\al k_0,\al n_0)\prod_{j=1}^mV_{\phi}f_j(0,0)\bigg|^q\Omba((0,\vec{0}),(\al n_0,\vec{0}))^q\bigg)^{1/q}\bigg)_{n_0}\bigg\|_{l^q}
\\
= &
\bigg\|\bigg(\bigg(\sum_{k_0\in \zd}\bigg|V_{\phi}f_0(\al k_0,\al n_0)\bigg|^q\Om_{b,0}(\al n_0)^p\bigg)^{1/q}\bigg)_{n_0}\bigg\|_{l^q},
\ \ \ \ \Ombo(\xi)=\Omba((0,\vec{0}),(\xi,\vec{0})).
\end{split}
\ee
Using Corollary \ref{cy-eqn}, there exits a sufficiently small  $\al$ such that
\be
\bigg\|\bigg(\bigg(\sum_{k_0\in \zd}\bigg|V_{\phi}f_0(\al k_0,\al n_0)\bigg|^q\Om_{b,0}(\al n_0)^p\bigg)^{1/q}\bigg)_{n_0}\bigg\|_{l^q}
\sim
\|f_0\|_{M^{q,q}_{1\otimes\Om_{b,0}}(\rd)}.
\ee
Combining the above estimates with \eqref{lm-lbeq-cd2}, we deduce that
\ben\label{lm-lbeq-1}
\|f_0\|_{M^{q,q}_{1\otimes\Om_{b,0}}(\rd)}
\lesssim
\prod_{j=0}^m\|f_j\|_{L^{p_j}(\rd)}\sim \|f_0\|_{L^{p_0}(\rd)},
\een
for any smooth function $f_0$ supported on $B_{\d}$,
which is just the embedding relation $L^{p_0}(B_{\d})\subset M^{q,q}_{1\otimes\Om_{b,0}}$.

For $i=1,2,\cdots,m$,
take $f_j=h$ for all $0\leq j\leq m$ and $j\neq i$, where $h$ is the function mentioned above.
We have the lower estimate for the last term in \eqref{lm-lbeq-0}:
\be
\begin{split}
  &
  \|V_{\phi}f_0(\al k_0,\al(n_0+\sum_{j=1}^mk_j))
  \prod_{j=1}^mV_{\phi}f_j(\al(k_0+n_j),\al k_j)\Omba((0,\al \vec{k}),(\al n_0,\vec{0}))\|_{l^{p,q}}
  \\
  \gtrsim &
  \bigg(\sum_{n_0,n_i}
  \big|V_{\phi}f_0(0,0)V_{\phi}f_i(\al n_i,-\al n_0)
  \prod_{j\neq i}V_{\phi}f_j(0,0)\Omba((0,(\underbrace{0,\cdots,-\al n_0,0,\cdots,0}_{-\al n_0\ \text{is the}\ ith\  vector})),(\al n_0,\vec{0}))\big|^q\bigg)^{1/q}
  \\
  \sim &
  \bigg(\sum_{n_0,n_i}
  \big|V_{\phi}f_i(\al n_i,\al n_0)
  \Ombi(\al n_0)\big|^q\bigg)^{1/q}\sim \|f_i\|_{M^{q,q}_{1\otimes \Ombi}},
\end{split}
\ee
where
\be
\Ombi(\xi)=\Omba((0,(\underbrace{0,\cdots,\xi,0,\cdots,0}_{\xi\ \text{is the}\ ith\  vector})),(-\xi,\vec{0})).
\ee

Using this and \eqref{lm-lbeq-cd2}, we obtain
\be
\|f_i\|_{M^{q,q}_{1\otimes \Ombi}(\rd)}\lesssim \|f_i\|_{L^{p_i}(\rd)},
\ee
for any smooth function $f_i$ supported on $B_{\d}$, which is just the embedding relation
$L^{p_i}(B_{\d})\subset M^{q,q}_{1\otimes\Om_{b,i}}(\rd)$.
We have now completed the proof for $\eqref{lm-lbeq-cd2}\Longrightarrow \eqref{lm-lbeq-cd3}$.
The equivalent relation between $\eqref{lm-lbeq-cd3}$ and $\eqref{lm-lbeq-cd4}$ follows by Lemma \ref{lm-lpm}.

Next, we verify the opposite direction for $p\geq q$.
In this case,  $\Om$ satisfies condition $M2$, we have
\be
\begin{split}
\Omba((z_0,\vec{z}),(\z_0,\vec{\z}))
\sim &
\Omba((0,\vec{z}),(\z_0,\vec{0}))
\\
\lesssim &
\Omba((0,\vec{0}),(\z_0+\sum_{j=1}^mz_j,\vec{0}))
\prod_{j=1}^m\Omba((0,(\underbrace{0,\cdots,z_j,0,\cdots,0}_{z_j\ \text{is the}\ jth\  vector})),(-z_j,\vec{0}))
\\
= &
\Om_{b,0}(\z_0+\sum_{j=1}^mz_j)\prod_{j=1}^m\Ombj(z_j).
\end{split}
\ee
Using this and the embedding property of modulation space, we have
\be
\begin{split}
  &\|R_m(f_0,\vec{f})\|_{M^{p,q}_{\Omba}(\rmdd)}
  \lesssim
  \|R_m(f_0,\vec{f})\|_{M^{q,q}_{\Omba}(\rmdd)}
  \\
  &\quad\quad\quad=
  \|V_{\phi}f_0(z_0,\z_0+\sum_{j=1}^mz_j)
  \prod_{j=1}^mV_{\phi}f_j(z_0+\z_j,z_j)\|_{L^{q,q}_{\Omba}}
  \\
  &\quad\quad\quad\lesssim
  \|V_{\phi}f_0(z_0,\z_0+\sum_{j=1}^mz_j)\Ombo(\z_0+\sum_{j=1}^mz_j)
  \prod_{j=1}^mV_{\phi}f_j(z_0+\z_j,z_j)\Ombj(z_j)\|_{L^{q,q}}
  \\
  &\quad\quad\quad=
  \|V_{\phi}f_0(z_0,\z_0)\Ombo(\z_0)
  \prod_{j=1}^mV_{\phi}f_j(\z_j,z_j)\Ombj(z_j)\|_{L^{q,q}}
  \\
  &\quad\quad\quad=
  \|f_0\|_{M^{q,q}_{1\otimes \Ombo}(\rd)}\prod_{j=1}^m\|f_j\|_{M^{q,q}_{1\otimes \Ombj}(\rd)}
  \lesssim
  \|f_0\|_{L^{p_0}(B_{\d})}\prod_{j=1}^m\|f_j\|_{L^{p_j}(B_{\d})}.
\end{split}
\ee
We have now completed this proof.
\end{proof}

\begin{lemma}\label{lm-p}
  Suppose that $p,q\in (0,\fy]$, $\mu\in \mathscr{P}(\rd)$. We have
  \be
  L^p(B_{\d})\subset \scrF^{-1}L^q_{\mu}\ \ \text{for some}\ \d\ \Longrightarrow\ p\geq 1\ \ \text{and}\ \ \mu\lesssim 1.
  \ee
\end{lemma}
\begin{proof}
  Let $f$ be a nonzero smooth function supported on $B_{\d}$ satisfying $\widehat{f}(0)=2$.
  Then, there exists a constant $\d_0$ such that $\widehat{f}(\xi)\geq 1$ for $\xi\in B_{\d_0}$.
  Denote
  $f_{\la}(x):=\frac{1}{\la^d}f(\frac{x}{\la})$ for $\la\in (0,1)$.
  We have
  \be
  \begin{split}
    \|f_{\la}\|_{\scrF^{-1}L^q_{\mu}}
    =
    \bigg(\int_{\rd}|\widehat{f}(\la\xi)|^q\mu(\xi)^qd\xi\bigg)^{1/q}
    \geq
    \bigg(\int_{B_{\d_0}}\mu(\xi)^qd\xi\bigg)^{1/q}\gtrsim 1.
  \end{split}
  \ee
  If the embedding $L^p(B_{\d})\subset \scrF^{-1}L^q_{\mu}$ holds, we have
  \be
  1\lesssim \|f_{\la}\|_{\scrF^{-1}L^q_{\mu}}\lesssim \|f_{\la}\|_{L^p}\sim \la^{d(1/p-1)},\ \ \ \ \la\in (0,1),
  \ee
  which implies $p\geq 1$ by letting $\la\rightarrow 0$.
  
  On the other hand, let $f_{\xi_0}=M_{\xi_0}f$, $\xi_0\in \rd$.   We have
  \be
  \begin{split}
  \|f\|_{L^p}=
  \|f_{\xi_0}\|_{L^p}\gtrsim \|f_{\xi_0}\|_{\scrF^{-1}L^q_{\mu}}
  =
  \|\scrF f(\cdot-\xi_0)\|_{L^q_{\mu}}\gtrsim \mu(\xi_0).
  \end{split}
  \ee
\end{proof}

\subsection{A mixed-norm embedding}
Let $\vec{b_0}=\{b_0(k)\}_{k\in \zd}$ and $\vec{B}=\{B(\vec{k})\}_{\vec{k}\in \zmd}$ be
two sequences defined on $\zd$ and $\zmd$ respectively, where $\vec{k}=(k_1,\cdots,k_m)$ be a vector on $\rmd$ with $k_j\in \zd$, $j=1,\cdots,m$.
Let $\tau_m$ be the coordinate transform defined as
\be
\tau_m(\vec{b_0}\otimes \vec{B})(k_0,\vec{k})=b_0(k_0)B(k_1+k_0,k_2+k_0,\cdots,k_m+k_0).
\ee
For the case that $\vec{B}=\otimes_{j=1}^m\vec{b_j}$, we have
\be
\tau_m(\otimes_{j=0}^m\vec{b_j})(k_0,\vec{k})
=
\tau_m(\vec{b_0}\otimes (\otimes_{j=1}^mb_j))(k_0,\vec{k})=b_0(k_0)\prod_{j=1}^mb_j(k_j+k_0).
\ee
Moreover,
we use $\tau_m(\otimes_{j=0}^ml^{q_j}_{\mu_j})\subset l^{p,q}_{W}(\zd\times\zmd)$
to denote the following inequality
\be
\|\tau_m(\otimes_{j=0}^m\vec{b_j})(k_0,\vec{k})\|_{l^{p,q}_{W}(\zd\times\zmd)}\leq C \|\vec{b_0}\|_{l^{q_0}_{\mu_0}(\zd)}\prod_{j=1}^m \|\vec{b_j}\|_{l^{q_j}_{\mu_j}(\zd)},
\ee
where $p,q,q_j\in (0,\fy]$, $W$ is a weight function on $\rmdd$, and $\{\mu_j\}_{j=0}^m$ is a sequence of weight functions on $\rd$.

\begin{lemma}\label{lm-meeq}
  Let $0<p,q,q_j\leq \fy$, $W$ be a weight function on $\rmdd$, and $\mu_j$ be weight functions on $\rd$,
  $j=1,\cdots,m$.
  Denote
  $w_{0}(k_0)=W(k_0,(-k_0,\cdots,-k_0))$,
  $w_i(k_i)=W(0,(\underbrace{0,\cdots,k_i,0,\cdots,0}_{k_i\ \text{is the}\ ith\  vector})) $
for $k_0 \in \zd$ and $\vec{k}=(k_1,\cdots,k_m)\in \zmd$.
  Then the following embedding inequality
  \ben\label{lm-meeq-cd1}
  \tau_m(\otimes_{j=0}^ml^{q_j}_{\mu_j}(\zd))\subset l^{p,q}_{W}(\zd\times\zmd)
  \een
  implies the following embedding relations
  \ben\label{lm-meeq-cd2}
   l^{q_i}_{\mu_i}(\zd)\subset l^q_{w_{i}}(\zd),
  \ \ \ i=0,1,\cdots,m.
  \een
  Moreover, for $p\geq q$, if $W$ satisfies the condition $W(k_0,\vec{k})\lesssim w_0(k_0)\prod_{j=1}^m w_j(k_j+k_0)$, 
  the opposite direction is also valid. In this case, we have
  the equivalent relation $\eqref{lm-meeq-cd1}\Longleftrightarrow \eqref{lm-meeq-cd2}$.
\end{lemma}
\begin{proof}
  Write $\tau_m(\otimes_{j=0}^ml^{q_j}_{\mu_j})\subset l^{p,q}_{W}(\zd\times\zmd)$ by
  \ben\label{lm-meeq-1}
  \left\|\left(\big(\sum_{k_0\in \zd}\big|b_0(k_0)\prod_{j=1}^mb_j(k_j+k_0)W(k_0,\vec{k})\big|^p\big)^{1/p}\right)_{\vec{k}\in \zmd}\right\|_{l^q(\zmd)}
  \lesssim
  \prod_{j=0}^m \|\vec{b_j}\|_{l^{q_j}_{\mu_j}(\zd)}.
  \een
In this inequality, we take $b_{j}(0)=1$ and $b_{j}(k)=0$ for all $k\in \zd\bs\{0\}$, $j=1,\cdots,m$. Then, we obtain that
\be
\|b_0(k_0)W(k_0,-k_0,\cdots,-k_0)\|_{l^q(\zd)}
\lesssim \|\vec{b_0}\|_{l^{q_0}_{\mu_0}(\zd)},
\ee
which is just the embedding relation $l^{q_0}_{\mu_0}(\zd)\subset l^{q}_{w_0}(\zd)$.

For a fixed $i=1,2,\cdots,m$, and any $0\leq j\leq m$ with $j\neq i$,
we take
$b_j(0)=1$ and $b_j(k)=0$, for all $k\in \zd\bs \{0\}$.
Then \eqref{lm-meeq-1} tells us that
\be
\|b_i(k_i)W(0,(\underbrace{0,\cdots,k_i,0,\cdots,0}_{k_i\ \text{is the}\ ith\  vector}))\|_{l^q(\zd)}
\lesssim
\|\vec{b_i}\|_{l^{q_i}_{\mu_i}(\zd)},
\ee
which is just the embedding relation $l^{q_i}_{\mu_i}(\zd)\subset l^q_{w_i}(\zd)$.

Next, we verify the opposite direction for $p\geq q$. In this case, we have the embedding relation $l^q(\rd)\subset l^p(\rd)$.
Using this and the assumption
\be
W(k_0,\vec{k})\lesssim w_0(k_0)\prod_{j=1}^m w_j(k_j+k_0),
\ee
we get
\be
\begin{split}
  &\|\tau_m(\otimes_{j=0}^m\vec{b_j})\|_{l^{p,q}_{W}(\zd\times\zmd)}
  \lesssim
  \|\tau_m(\otimes_{j=0}^m\vec{b_j})\|_{l^{q,q}_{W}(\zd\times\zmd)}
  \\
  & \quad\quad\quad \lesssim
  \bigg(\sum_{\vec{k}\in \zmd}\bigg(\sum_{k_0\in \zd}|b_0(k_0)|^qw_0(k_0)^q\prod_{j=1}^m|b_j(k_j+k_0)|^qw_j(k_j+k_0)^q\bigg)\bigg)^{1/q}
  \\
  & \quad\quad\quad =
  \|\vec{b_0}\|_{l^{q}_{w_0}(\zd)}\prod_{j=1}^m \|\vec{b_j}\|_{l^{q}_{w_j}(\zd)}
  \lesssim
  \|\vec{b_0}\|_{l^{q_0}_{\mu_0}(\zd)}\prod_{j=1}^m \|\vec{b_j}\|_{l^{q_j}_{\mu_j}(\zd)},
\end{split}
\ee
where in the last inequality we use the embedding relations $l^{q_i}_{\mu_i}(\zd)\subset l^q_{w_i}(\zd)$ for $i=0,1,\cdots,m$.
\end{proof}

\subsection{Separation in time plane for BRWM}
\begin{theorem}\label{thm-M0}
  Assume $p_i, q_i, p, q \in (0,\fy]$,
  and that $\Om\in \mathscr{P}(\rr^{2(m+1)d})$, $\mu_i \in \mathscr{P}(\rd)$, $i=0,1,\cdots,m$.
  For any $\d>0$, we have
  \ben\label{thm-M0-cd0}
  \begin{split}
  &\|R_m(g,f_1,\cdots,f_m)\|_{M^{p,q}_{\Om}(\rmdd)}
  \\
  \sim &
  \bigg(\sum_{\vec{k}\in \zmd}\bigg\|\bigg( \sum_{k_0\in \zd}
\bigg\|V_{\P}(R_m(g_{k_0},f_{1,k_0+k_1},\cdots,f_{m,k_0+k_m})\Om\bigg\|_{L^p}^p\bigg)^{1/p}\bigg\|^q_{L^q}\bigg)^{1/q}
  \end{split}
  \een
  for $g=\sum_{k_0\in \zd}g_{k_0}$, $f_j=\sum_{k_j\in \zd}f_{j,k_j}$ with $\text{supp}g_{k_0}\subset B(k_0,\d)$
  and $\text{supp}f_{j,k_j}\subset B(k_j,\d)$,
  and $\Phi=R_m(\phi,\cdots,\phi)$ where $\phi$ is a smooth function supported in $B(0,\d)$.
Moreover, for any $\d>0$, the following two statements are equivalent:
  \bn
  \item The following boundedness is valid:
  \ben\label{thm-M0-cd1}
  R_m: W(L^{p_0},L^{q_0}_{\mu_0})(\rd)\times\cdots \times W(L^{p_m},L^{q_m}_{\mu_m})(\rd)\longrightarrow M^{p,q}_{\Om}(\rmdd).
  \een
  \item
  Let $g=\sum_{k_0\in \zd}g_{k_0}\in W(L^{p_0},L^{q_0}_{\mu_0})$, $f_j=\sum_{k_j\in \zd}f_{j,k_j}\in W(L^{p_j},L^{q_j}_{\mu_j})$ with $\text{supp}g_{k_0}\subset B(k_0,\d)$
  and $\text{supp}f_{j,k_j}\subset B(k_j,\d)$.
  Let $\Phi=R_m(\phi,\cdots,\phi)$, where $\phi$ is a smooth function supported in $B(0,\d)$.
  We have the following boundedness result:
  \ben\label{thm-M0-cd2}
  \begin{split}
  &\bigg(\sum_{\vec{k}\in \zmd}\bigg\|\bigg( \sum_{k_0\in \zd}
\bigg\|V_{\P}(R_m(g_{k_0},f_{1,k_0+k_1},\cdots,f_{m,k_0+k_m})\Om\bigg\|_{L^p}^p\bigg)^{1/p}\bigg\|^q_{L^q}\bigg)^{1/q}
  \\
  \lesssim &
  \|(\|g_{k_0}\|_{L^{p_0}})_{k_0}\|_{l^{q_0}_{\mu_0}}\prod_{j=1}^m\|(\|f_{j,k_j}\|_{L^{p_j}})_{k_j}\|_{l^{q_j}_{\mu_j}}.
  \end{split}
  \een
  \en
\end{theorem}
\begin{proof}
  We first verify \eqref{thm-M0-cd0}.
  Using Lemma \ref{lm-spSTFT} and Lemma \ref{lm-STFT-mRd}, for any fixed $k_0\in \zd$ and $\vec{k}:=(k_1,\cdots,k_m)\in \zmd$,
the STFT of $R_m(g_{k_0},f_{1,k_0+k_1},\cdots,f_{m,k_0+k_m})$ associated with window $\Phi$ can be written by
\be
\begin{split}
  &\left|V_{\P}(R_m(g_{k_0},f_{1,k_0+k_1},\cdots,f_{m,k_0+k_m}))((z_0,\vec{z}),(\z_0,\vec{\z}))\right|
\\
= &
\bigg|
V_{\phi}g_{k_0}(z_0,\z_0+\sum_{j=1}^mz_j)\prod_{j=1}^mV_{\phi}f_{j,k_0+k_j}(z_0+\z_j,z_j)\bigg|
\\
= &
\bigg|V_{\phi}g_{k_0}(z_0,\z_0+\sum_{j=1}^mz_j)\chi_{B(k_0,2\d)}(z_0)\prod_{j=1}^mV_{\phi}f_{j,k_0+k_j}(z_0+\z_j,z_j)\chi_{B(k_j,4\d)}(\z_j)\bigg|
\\
= &
\left|V_{\phi}(R_m(g_{k_0},f_{1,k_0+k_1},\cdots,f_{m,k_0+k_m}))((z_0,\vec{z}),(\z_0,\vec{\z}))\chi_{B(k_0,2\d)}(z_0)\prod_{j=1}^m\chi_{B(k_j,4\d)}(\z_j)\right|.
\end{split}
\ee
Write
\begin{eqnarray*}
  &&V_{\P}(R_m(g,\vec{f}))((z_0,\vec{z}),(\z_0,\vec{\z}))
  =
  \sum_{k_0\in \zd}\sum_{\vec{k}\in \zmd}
V_{\P}(R_m(g_{k_0},f_{1,k_1},\cdots,f_{m,k_m}))((z_0,\vec{z}),(\z_0,\vec{\z}))
\\
&&\quad=
  \sum_{k_0\in \zd}\sum_{\vec{k}\in \zmd}
V_{\P}(R_m(g_{k_0},f_{1,k_0+k_1},\cdots,f_{m,k_0+k_m}))((z_0,\vec{z}),(\z_0,\vec{\z}))
  \\
&&\quad=
  \sum_{\vec{k}\in \zmd}\sum_{k_0\in \zd}
V_{\P}(R_m(g_{k_0},f_{1,k_0+k_1},\cdots,f_{m,k_0+k_m}))((z_0,\vec{z}),(\z_0,\vec{\z})).
\end{eqnarray*}

Using the above two estimates and observing that
the supports of the above functions are almost separated from each other,
we obtain the following decomposition of $V_{\P}(R_m(g,\vec{f}))$ for $p<\fy$ (with usual modification for $p=\fy$):
\be
\begin{split}
  &\left|V_{\P}(R_m(g,\vec{f}))((z_0,\vec{z}),(\z_0,\vec{\z}))\right|^p
\\
\sim &
\sum_{\vec{k}\in \zmd}\sum_{k_0\in \zd}
\bigg|V_{\P}(R_m(g_{k_0},f_{1,k_0+k_1},\cdots,f_{m,k_0+k_m})((z_0,\vec{z}),(\z_0,\vec{\z}))\bigg|^p
\\
= &
\sum_{\vec{k}\in \zmd}\sum_{k_0\in \zd}
\bigg|V_{\P}(R_m(g_{k_0},f_{1,k_0+k_1},\cdots,f_{m,k_0+k_m})((z_0,\vec{z}),(\z_0,\vec{\z}))\bigg|^p
\chi_{B(k_0,2\d)}(z_0)\prod_{j=1}^m\chi_{B(k_j,4\d)}(\z_j).
\end{split}
\ee
Then, the modulation norm of $R_m(g,\vec{f})$ can be written by
\be
\begin{split}
 &\|R_m(g,\vec{f})\|_{M^{p,q}_{\Om}(\rmdd)}
 =
 \|V_{\P}(R_m(g,\vec{f}))\|_{L^{p,q}_{\Om}(\rmdd\times \rmdd)}
 \\
 \sim &
 \bigg\|\bigg(\sum_{\vec{k}\in \zmd}\prod_{j=1}^m\chi_{B(k_j,4\d)}(\z_j) \sum_{k_0\in \zd}
\bigg\|V_{\P}(R_m(g_{k_0},f_{1,k_0+k_1},\cdots,f_{m,k_0+k_m})\Om\bigg\|_{L^p}^p\bigg)^{1/p}\bigg\|_{L^q}
 \\
 \sim &
 \bigg\|\sum_{\vec{k}\in \zmd}\prod_{j=1}^m\chi_{B(k_j,4\d)}(\z_j)\bigg( \sum_{k_0\in \zd}
\bigg\|V_{\P}(R_m(g_{k_0},f_{1,k_0+k_1},\cdots,f_{m,k_0+k_m})\Om\bigg\|_{L^p}^p\bigg)^{1/p}\bigg\|_{L^q}
\\
= &
\bigg(\sum_{\vec{k}\in \zmd}\bigg\|\bigg( \sum_{k_0\in \zd}
\bigg\|V_{\P}(R_m(g_{k_0},f_{1,k_0+k_1},\cdots,f_{m,k_0+k_m})\Om\bigg\|_{L^p}^p\bigg)^{1/p}\bigg\|^q_{L^q}\bigg)^{1/q}.
\end{split}
\ee
We have now completed the proof of \eqref{thm-M0-cd0}.
From this and the fact that
\be
 \|g\|_{W(L^{p_0},L^{q_0}_{\mu_0})}\sim\|(\|g_{k_0}\|_{L^{p_0}})_{k_0}\|_{l^{q_0}_{\mu_0}},\ \ \ \ \
 \|f_j\|_{W(L^{p_j},L^{q_j}_{\mu_j})}\sim \|(\|f_{j,k_j}\|_{L^{p_j}})_{k_j}\|_{l^{q_j}_{\mu_j}},
\ee
we complete the proof of $(1)\Longrightarrow(2)$.
Next, we turn to the proof of $(2)\Longrightarrow(1)$. Without loss of generality, we assume that $\d<1/2$.

Take $M$ to be a sufficiently large constant such that $[-1/M,1/M]^d\subset B(0,\d)$.
There exists a smooth function $\s$ supported in $[-1/M,1/M]^d$ such that
\be
1=\sum_{k\in \zd}\s(x-k/M)=\sum_{i\in \G}\sum_{k\in \zd}\s(x-k-i/M)= :\sum_{i\in \G}T_{i/M}(\sum_{k\in \zd}\s(x-k)),
\ee
where $\G=[0,M)^d\cap \zd$.
Then, any fixed function $h$ can be divided by
\be
h(x)
=\sum_{i\in \G}\bigg( T_{i/M}\sum_{k\in \zd}\s(x-k)\bigg)h(x)
=\sum_{i\in \G}T_{i/M}h_i(x),
\ee
where $h_i(x)=(T_{-i/M}h)(x)\sum_{k\in \zd}\s(x-k)$ is supported in $\bigcup_{k\in \zd}B(k,\d)$.
Similarly, we write
\be
g=\sum_{i\in \G}T_{i/M}g_i,\ \ \ f_j=\sum_{i\in \G}T_{i/M}f_{j,i},
\ee
where $\text{supp}g_i\subset \bigcup_{k\in \zd}B(k,\d)$ and $\text{supp}f_{j,i}\subset \bigcup_{k\in \zd}B(k,\d)$.
Note that for $x_j\in [0,1)$, $j=0,\cdots,m$,
\be
\Om(z_0-x_0,\vec{z},\z_0,\z_1+x_0-x_1,\cdots,\z_m+x_0-x_m)\sim \Om(z_0,\vec{z},\z_0,\vec{\z}).
\ee
From this and Lemma \ref{lm-tlSTFT},
we obtain that
\be
\begin{split}
&\|R_m(T_{x_0}g_i, T_{x_1}f_{1,i_1},\cdots, T_{x_m}f_{m,i_m})\|_{M^{p,q}_{\Om}}
\\
= &
\bigg\|V_{\phi}g_i(z_0-x_0,\z_0+\sum_{j=1}^mz_j)\prod_{j=1}^mV_{\phi}f_{j,i_j}(z_0+\z_j-x_j,z_j)\bigg\|_{L^{p,q}_{\Om}}
\\
\sim &
\bigg\|V_{\phi}g_i(z_0,\z_0+\sum_{j=1}^mz_j)\prod_{j=1}^mV_{\phi}f_{j,i_j}(z_0+\z_j+x_0-x_j,z_j)\bigg\|_{L^{p,q}_{\Om}}
\\
\sim &
\bigg\|V_{\phi}g_i(z_0,\z_0+\sum_{j=1}^mz_j)\prod_{j=1}^mV_{\phi}f_{j,i_j}(z_0+\z_j,z_j)\bigg\|_{L^{p,q}_{\Om}}
=
\|R_m(g_i, f_{1,i_1},\cdots, f_{m,i_m})\|_{M^{p,q}_{\Om}}.
\end{split}
\ee
From this, the modulation norm of $R_m(g,\vec{f})$ can be estimated by
\be
\begin{split}
  &\|R_m(g,\vec{f})\|_{M^{p,q}_{\Om}(\rmdd)}
  \\
  = &
  \bigg\|\sum_{(i_j)_{j=1}^m\in (\G)^m}\sum_{i\in \G}R_m(T_{i/M}g_i,T_{i_1/M}f_{1,i_1},\cdots,T_{i_m/M}f_{m,i_m})\bigg\|_{M^{p,q}_{\Om}(\rmdd)}
  \\
  \lesssim &
  \sum_{(i_j)_{j=1}^m\in (\G)^m}\sum_{i\in \G}\bigg\|R_m(T_{i/M}g_i,T_{i_1/M}f_{1,i_1},\cdots,T_{i_m/M}f_{m,i_m})\bigg\|_{M^{p,q}_{\Om}(\rmdd)}
  \\
  \lesssim &
  \sum_{(i_j)_{j=1}^m\in (\G)^m}\sum_{i\in \G}\bigg\|R_m(g_i,f_{1,i_1},\cdots,f_{m,i_m})\bigg\|_{M^{p,q}_{\Om}(\rmdd)}.
\end{split}
\ee
Recall that all the functions $g_i$ and $f_{j,i} (j=1,2,\cdots,m)$ are supported in $\bigcup_{k\in \zd}B(k,\d)$,
and observe that
\be
\|g_i\|_{W(L^{p_0},L^{q_0}_{\mu_0})}\lesssim \|g\|_{W(L^{p_0},L^{q_0}_{\mu_0})},\ \ \
\|f_{j,i_j}\|_{W(L^{p_j},L^{q_j}_{\mu_j})}\lesssim  \|f_{j}\|_{W(L^{p_j},L^{q_j}_{\mu_0})}.
\ee
We only need to verify \eqref{thm-M0-cd1} by the following inequality
\be
\begin{split}
  \|R_m(G,\vec{F})\|_{M^{p,q}_{\Om}(\rmdd)}
  \lesssim &
  \|G\|_{W(L^{p_0},L^{q_0}_{\mu_0})}\prod_{j=1}^m\|F_j\|_{W(L^{p_j},L^{q_j}_{\mu_j})},
\end{split}
\ee
for all the functions $G$ and $F_j$
are supported in $\bigcup_{k\in \zd} B(k,\d)$.
In this case, write
\be
G=\sum_{k\in \zd}G_k ,\ \ \ \ F_j=\sum_{k\in \zd}F_{j,k},
\ee
where $\text{supp} G_k\subset B(k,\d)$ and $\text{supp} F_{j,k}\subset B(k,\d)$, $j=1,\cdots,m$.
Using \eqref{thm-M0-cd0}, we have
\be
\begin{split}
&\|R_m(G,\vec{F})\|_{M^{p,q}_{\Om}(\rmdd)}
\\
\sim &
\bigg(\sum_{\vec{k}\in \zmd}\bigg\|\bigg( \sum_{k_0\in \zd}
\bigg\|V_{\P}(R_m(G_{k_0},F_{1,k_0+k_1},\cdots,F_{m,k_0+k_m})\Om\bigg\|_{L^p}^p\bigg)^{1/p}\bigg\|^q_{L^q}\bigg)^{1/q}.
\end{split}
\ee
From this and \eqref{thm-M0-cd2}, we have the desired conclusion:
\be
\begin{split}
&\|R_m(G,\vec{F})\|_{M^{p,q}_{\Om}(\rmdd)}
\\
\lesssim &
\|(\|G_{k_0}\|_{L^{p_0}})_{k_0}\|_{l^{q_0}_{\mu_0}}\prod_{j=1}^m\|(\|F_{j,k_j}\|_{L^{p_j}})_{k_j}\|_{l^{q_j}_{\mu_j}}
\sim
 \|G\|_{W(L^{p_0},L^{q_0}_{\mu_0})}\prod_{j=1}^m\|F_j\|_{W(L^{p_j},L^{q_j}_{\mu_j})}.
\end{split}
\ee
\end{proof}

\subsection{First characterization for BRWM}
\begin{proof}[Proof of Theorem \ref{thm-M1}]
By Lemmas \ref{lm-lbeq} and \ref{lm-meeq},
we have
$\eqref{thm-M1-cd1}\Longrightarrow \eqref{thm-M1-cd2}$ and
$\eqref{thm-M1-cd3}\Longrightarrow\eqref{thm-M1-cd4}$,
where the opposite direction is valid for $p\geq q$, if $\Om$ satisfies condition $M2$ and $M1$, respectively.
Using Lemma \ref{lm-p}, we conclude that
$\eqref{thm-M1-cd2}\Longrightarrow \eqref{thm-M1-cdp}$.
Thus, we only need to verify the relations mentioned in Theorem \ref{thm-M1},
between $\eqref{thm-M1-cd0}$ and $\eqref{thm-M1-cd1},\eqref{thm-M1-cd3}$.

We divide the proof into two parts.

\textbf{``Only if'' part.}
First, $\eqref{thm-M1-cd0}\Longrightarrow \eqref{thm-M1-cd1}$ follows by Theorem \ref{thm-M0} and Lemma \ref{lm-lbeq}.
Next, we turn to verify that
$\eqref{thm-M1-cd0} \Longrightarrow \eqref{thm-M1-cd3}$.
For any nonnegative truncated (only finite nonzero items) sequence
  $\vec{a}=(a_{k})_{k\in \zd}$ and $\vec{b_j}=(b_{j,k})_{k\in \zd}$, we set
\be
g=\sum_{k\in \zd}a_kT_k\va=: \sum_{k\in \zd}g_k ,\ \ \ \ f_j=\sum_{k\in \zd}b_{j,k}T_k\va=: \sum_{k\in \zd}f_{j,k},
\ee
where $\va$ is chosen to be a smooth function supported in $B(0,\d)$ with some small constant $\d>0$.
Let $\Phi=R_m(\phi,\cdots,\phi)$, where $\phi$ is a smooth function supported in $B(0,\d)$.
Using the same method in the proof of Theorem \ref{thm-M0}, we have
\be
\begin{split}
  &\left|V_{\P}(R_m(g_{k_0},f_{1,k_0+k_1},\cdots,f_{m,k_0+k_m}))((z_0,\vec{z}),(\z_0,\vec{\z}))\right|
\\
= &
\left|V_{\P}(R_m(g_{k_0},f_{1,k_0+k_1},\cdots,f_{m,k_0+k_m}))((z_0,\vec{z}),(\z_0,\vec{\z}))\chi_{B(k_0,2\d)}(z_0)\prod_{j=1}^m\chi_{B(k_j,4\d)}(\z_j)\right|,
\end{split}
\ee
and
\ben\label{pp-M1-1}
\begin{split}
 &\|R_m(g,\vec{f})\|_{M^{p,q}_{\Om}(\rmdd)}
\\
= &
\bigg(\sum_{\vec{k}\in \zmd}\bigg\|\bigg( \sum_{k_0\in \zd}
\bigg\|V_{\P}(R_m(g_{k_0},f_{1,k_0+k_1},\cdots,f_{m,k_0+k_m})\Om\bigg\|_{L^p}^p\bigg)^{1/p}\bigg\|^q_{L^q}\bigg)^{1/q}.
\end{split}
\een
From this and the fact that
\be
\Om(z_0,\vec{z},\z_0,\vec{\z})
\sim
\Om(k_0,\vec{0},0,\vec{k})=\Omab(k_0,\vec{k})
\ee
for $z_0\in B(k_0,2\d)$,
$z_j, \z_0\in B_{\d}$, $\z_j\in B(k_j,4\d)$, $j=1,\cdots,m$, where $\vec{k}=(k_1,\cdots,k_m)$, we have
\be
\begin{split}
  &\bigg\|V_{\P}(R_m(g_{k_0},f_{1,k_0+k_1},\cdots,f_{m,k_0+k_m})\Om\bigg\|_{L^p}
  \\
  = &
  \bigg\|V_{\P}(R_m(g_{k_0},f_{1,k_0+k_1},\cdots,f_{m,k_0+k_m})\Om\cdot \chi_{B(k_0,2\d)}(z_0)\bigg\|_{L^p}
  \prod_{j=1}^m\chi_{B(k_j,4\d)}(\z_j)
  \\
  \gtrsim &
  \|V_{\P}(R_m(g_{k_0},f_{1,k_0+k_1},\cdots,f_{m,k_0+k_m})\Om\cdot\chi_{B(k_0,2\d)}(z_0)\prod_{j=1}^m\chi_{B_{\d}}(z_j)\bigg\|_{L^p}
  \chi_{B_{\d}}(\z_0)\prod_{j=1}^m\chi_{B(k_j,4\d)}(\z_j)
  \\
  \sim &
  \bigg\|V_{\P}(R_m(g_{k_0},f_{1,k_0+k_1},\cdots,f_{m,k_0+k_m})\prod_{j=1}^m\chi_{B_{\d}}(z_j)\bigg\|_{L^p}\Omab(k_0,\vec{k})\chi_{B_{\d}}(\z_0).
\end{split}
\ee
Hence,

\be
\begin{split}
  &\bigg\|\bigg( \sum_{k_0\in \zd}
\bigg\|V_{\P}(R_m(g_{k_0},f_{1,k_0+k_1},\cdots,f_{m,k_0+k_m})\Om\bigg\|_{L^p}^p\bigg)^{1/p}\bigg\|_{L^q}
  \\
  \gtrsim&
  \bigg\|\bigg( \sum_{k_0\in \zd}
\bigg\|V_{\P}(R_m(g_{k_0},f_{1,k_0+k_1},\cdots,f_{m,k_0+k_m})\prod_{j=1}^m\chi_{B_{\d}}(z_j)\bigg\|_{L^p}^p\Omab(k_0,\vec{k})^p\bigg)^{1/p}\chi_{B_{\d}}(\z_0)\bigg\|_{L^q}.
\end{split}
\ee

Using Lemmas \ref{lm-tlSTFT} and \ref{lm-STFT-mRd}, we obtain
\be
\begin{split}
  &|V_{\P}(R_m(g_{k_0},f_{1,k_0+k_1},\cdots,f_{m,k_0+k_m})((z_0,\vec{z}),(\z_0,\vec{\z}))|
  \\
  = &
  \big|a_{k_0}\prod_{j=1}^m b_{j,k_0+k_j}\big|\cdot |V_{\P}(R_m(T_{k_0}\va ,T_{k_0+k_1}\va,\cdots,T_{k_0+k_m}\va)|
    \\
  = &
  \big|a_{k_0}\prod_{j=1}^m b_{j,k_0+k_j}\big|\cdot \bigg|V_{\phi}T_{k_0}\va(z_0,\z_0+\sum_{j=1}^mz_j)\prod_{j=1}^mV_{\phi}T_{k_0+k_j}\va(z_0+\z_j,z_j)\bigg|
   \\
  = &
  \big|a_{k_0}\prod_{j=1}^m b_{j,k_0+k_j}\big|\cdot \bigg|V_{\phi}\va(z_0-k_0,\z_0+\sum_{j=1}^mz_j)\prod_{j=1}^mV_{\phi}\va(z_0+\z_j-k_0-k_j,z_j)\bigg|.
\end{split}
\ee
The above two estimates yield that
\be
\begin{split}
  &\bigg\|\bigg( \sum_{k_0\in \zd}
\bigg\|V_{\P}(R_m(g_{k_0},f_{1,k_0+k_1},\cdots,f_{m,k_0+k_m})\Om\bigg\|_{L^p}^p\bigg)^{1/p}\bigg\|_{L^q}
\\
\gtrsim &
\bigg\|\bigg( \sum_{k_0\in \zd}\big|a_{k_0}\prod_{j=1}^m b_{j,k_0+k_j}\big|^p\cdot\bigg\|V_{\phi}\va(z_0-k_0,\z_0+\sum_{j=1}^mz_j)
\\
& \prod_{j=1}^mV_{\phi}\va(z_0+\z_j-k_0-k_j,z_j)\chi_{B_{\d}}(z_j)\bigg\|_{L^p}^p\Omab(k_0,\vec{k})^p\bigg)^{1/p}\chi_{B_{\d}}(\z_0)\bigg\|_{L^q}
\\
= &
\bigg\|\bigg( \sum_{k_0\in \zd}\big|a_{k_0}\prod_{j=1}^m b_{j,k_0+k_j}\big|^p\cdot\bigg\|V_{\phi}\va(z_0,\z_0+\sum_{j=1}^mz_j)
\\
& \prod_{j=1}^mV_{\phi}\va(z_0+\z_j,z_j)\chi_{B_{\d}}(z_j)\bigg\|_{L^p}^p\Omab(k_0,\vec{k})^p\bigg)^{1/p}\chi_{B_{\d}}(\z_0)\bigg\|_{L^q}
\\
= &
\bigg\|\bigg( \sum_{k_0\in \zd}\big|a_{k_0}\prod_{j=1}^m b_{j,k_0+k_j}\big|^p\cdot
\bigg\|V_{\P}(R_m(\va,\va,\cdots,\va)\prod_{j=1}^m\chi_{B_{\d}}(z_j)\bigg\|_{L^p}^p\Omab(k_0,\vec{k})^p\bigg)^{1/p}\chi_{B_{\d}}(\z_0)\bigg\|_{L^q}
\\
= &
\bigg( \sum_{k_0\in \zd}\big|a_{k_0}\prod_{j=1}^m b_{j,k_0+k_j}\big|^p\Omab(k_0,\vec{k})^p\bigg)^{1/p}
\bigg\|
\bigg\|V_{\P}(R_m(\va,\va,\cdots,\va)\prod_{j=1}^m\chi_{B_{\d}}(z_j)\bigg\|_{L^p}\chi_{B_{\d}}(\z_0)\bigg\|_{L^q}
\\
\sim &
\bigg( \sum_{k_0\in \zd}\big|a_{k_0}\prod_{j=1}^m b_{j,k_0+k_j}\big|^p\Omab(k_0,\vec{k})^p\bigg)^{1/p}.
\end{split}
\ee
From this and \eqref{pp-M1-1}, we have the estimate
\ben\label{pp-M1-2}
\begin{split}
   &\|R_m(g,\vec{f})\|_{M^{p,q}_{1\otimes \bbu}(\rmdd)}
   \\
  \gtrsim &
  \bigg(\sum_{\vec{k}\in \zmd}\bigg( \sum_{k_0\in \zd}\big|a_{k_0}\prod_{j=1}^m b_{j,k_0+k_j}\big|^p\Omab(k_0,\vec{k})^p\bigg)^{q/p}\bigg)^{1/q}
  =
\|\tau_m\big(\vec{a}\otimes (\otimes_{j=1}^m\vec{b_j})\big)\|_{l^{p,q}_{\Omab}(\zd\times\zmd)}.
\end{split}
\een
On the other hand, we have the following direct estimates from the definition of Wiener amalgam spaces:
\ben\label{pp-M1-3}
\|g\|_{W(L^{p_0},L^{q_0})}=\|\vec{a}\|_{l^{q_0}_{\mu_0}},\ \ \ \|f_j\|_{W(L^{p_0},L^{q_0})}=\|\vec{b_j}\|_{l^{q_j}_{\mu_j}},\ \ \ j=1,\cdots,m.
\een

If \eqref{thm-M1-cd0} is valid, we use \eqref{pp-M1-2} and \eqref{pp-M1-3} to deduce that
\ben\label{pp-M1-8}
\|\tau_m\big(\vec{a}\otimes (\otimes_{j=1}^m\vec{b_j})\big)\|_{l^{p,q}_{\Omab}}
\lesssim
\|\vec{a}\|_{l^{q_0}_{\mu_0}}\prod_{j=1}^m\|\vec{b_j}\|_{l^{q_j}_{\mu_j}},
\een
which is just the relation \eqref{thm-M1-cd3}.

We have now completed the proof of $\eqref{thm-M1-cd0} \Longrightarrow \eqref{thm-M1-cd1},\eqref{thm-M1-cd3}$.

\textbf{``If'' part.}
In this part, we recall that $\Om$ satisfies $M0$.
Using Theorem \ref{thm-M0}, we only need to verify that
  \ben\label{thm-M1-1}
  \begin{split}
  &\bigg(\sum_{\vec{k}\in \zmd}\bigg\|\bigg( \sum_{k_0\in \zd}
\bigg\|V_{\P}(R_m(G_{k_0},F_{1,k_0+k_1},\cdots,F_{m,k_0+k_m})\Om\bigg\|_{L^p}^p\bigg)^{1/p}\bigg\|^q_{L^q}\bigg)^{1/q}
  \\
  \lesssim &
  \|(\|G_{k_0}\|_{L^{p_0}})_{k_0}\|_{l^{q_0}_{\mu_0}}\prod_{j=1}^m\|(\|F_{j,k_j}\|_{L^{p_j}})_{k_j}\|_{l^{q_j}_{\mu_j}}
  \end{split}
  \een
  with
\be
G=\sum_{k\in \zd}G_k ,\ \ \ \ F_j=\sum_{k\in \zd}F_{j,k},
\ee
where $\text{supp} G_k\subset B(k,\d)$ and $\text{supp} F_{j,k}\subset B(k,\d)$, $j=1,\cdots,m$ for sufficiently small $\d$.
By the fact that
\be
\begin{split}
  &|V_{\P}(R_m(G_{k_0},F_{1,k_0+k_1},\cdots,F_{m,k_0+k_m})((z_0,\vec{z}),(\z_0,\vec{\z}))|
   \\
  = &
  \bigg|V_{\P}(R_m(G_{k_0},F_{1,k_0+k_1},\cdots,F_{m,k_0+k_m})((z_0,\vec{z}),(\z_0,\vec{\z}))\chi_{B(k_0,2\d)}(z_0)\prod_{j=1}^m\chi_{B(k_j,4\d)}(\z_j)\bigg|,
\end{split}
\ee
and condition $M_0$ as follows:
\be
\Om(z_0,\vec{z},\z_0,\vec{\z})
\lesssim
\Omab(k_0,\vec{k})\Omba(z_0,\vec{z},\z_0,\vec{\z}),\ \ \ \ \ z_0\in B(k_0,2\d),\ \z_i\in B(k_i, 4\d),
\ee
we have the following estimate for the first term in \eqref{thm-M1-1}:
\ben\label{thm-M1-2}
\begin{split}
  &\bigg(\sum_{\vec{k}\in \zmd}\bigg\|\bigg( \sum_{k_0\in \zd}
\bigg\|V_{\P}(R_m(G_{k_0},F_{1,k_0+k_1},\cdots,F_{m,k_0+k_m})\Om\bigg\|_{L^p}^p\bigg)^{1/p}\bigg\|^q_{L^q}\bigg)^{1/q}
  \\
\lesssim &
\bigg(\sum_{\vec{k}\in \zmd}\bigg\|\bigg( \sum_{k_0\in \zd}
\bigg\|V_{\P}(R_m(G_{k_0},F_{1,k_0+k_1},\cdots,F_{m,k_0+k_m}))\Omba\bigg\|_{L^p}^p\Omab(k_0,\vec{k})^p\bigg)^{1/p}\bigg\|^q_{L^q}\bigg)^{1/q}.
\end{split}
\een

If $p\leq q$, by the Minkowski inequality the above term can be dominated from above by
\be
\begin{split}
&\bigg(\sum_{\vec{k}\in \zmd}
\bigg( \sum_{k_0\in \zd}\bigg\|V_{\P}(R_m(G_{k_0},F_{1,k_0+k_1},\cdots,F_{m,k_0+k_m}))\bigg\|^p_{L^{p,q}_{\Omba}}\Omab(k_0,\vec{k})^p\bigg)^{q/p}\bigg)^{1/q}
\\
= &
\bigg(\sum_{\vec{k}\in \zmd}\bigg( \sum_{k_0\in \zd}\bigg\|R_m(G_{k_0},F_{1,k_0+k_1},\cdots,F_{m,k_0+k_m})\bigg\|_{M^{p,q}_{\Omba}}^p\Omab(k_0,\vec{k})^p\bigg)^{q/p}\bigg)^{1/q}.
\end{split}
\ee
Observe that $\Omba$ is translation invariant with $z_0$,
and $\vec{\z}=(\z_1,\z_2,\cdots,\z_m)$.
Using Lemma \ref{lm-tlSTFT}, we deduce that
\ben\label{thm-M1-3}
\begin{split}
& \bigg\|R_m(G_{k_0},F_{1,k_0+k_1},\cdots,F_{m,k_0+k_m})\bigg\|_{M^{p,q}_{\Omba}}
\\
= &
\bigg\|V_{\phi}G_{k_0}(z_0,\z_0+\sum_{j=1}^mz_j)\prod_{j=1}^mV_{\phi}F_{j,k_0+k_j}(z_0+\z_j,z_j)\bigg\|_{L^{p,q}_{\Omba}}
\\
= &
\bigg\|V_{\phi}G_{k_0}(z_0+k_0,\z_0+\sum_{j=1}^mz_j)\prod_{j=1}^mV_{\phi}F_{j,k_0+k_j}(z_0+\z_j+k_0+k_j,z_j)\bigg\|_{L^{p,q}_{\Omba}}
\\
= &
\bigg\|R_m(T_{-k_0}G_{k_0},T_{-k_0-k_1}F_{1,k_0+k_1},\cdots,T_{-k_0-k_m}F_{m,k_0+k_m})\bigg\|_{M^{p,q}_{\Omba}}
\\
\lesssim &
\|T_{-k_0}G_{k_0}\|_{L^{p_0}}\prod_{j=1}^m\|T_{-k_0-k_j}F_{j,k_0+k_j}\|_{L^{p_j}}
=
\|G_{k_0}\|_{L^{p_0}}\prod_{j=1}^m\|F_{j,k_0+k_j}\|_{L^{p_j}},
\end{split}
\een
where we use \eqref{thm-M1-cd1} in the last inequality with the fact that all the functions $T_{-k_0}G_{k_0}$ and $T_{-k_0-k_j}F_{k_0+k_j}$
are supported in $B(0,\d)$.
The above three estimates yield that for $p\leq q$
\be
\begin{split}
  \|R_m(G,\vec{F})\|_{M^{p,q}_{\Om}(\rmdd)}
  \lesssim &
  \bigg(\sum_{\vec{k}\in \zmd}\bigg( \sum_{k_0\in \zd}\|G_{k_0}\|^p_{L^{p_0}}\prod_{j=1}^m\|F_{j,k_0+k_j}\|_{L^{p_j}}^p\Omab(k_0,\vec{k})^p\bigg)^{q/p}\bigg)^{1/q}
  \\
  \lesssim &
  \|(\|G_{k_0}\|_{L^{p_0}})_{k_0}\|_{l^{q_0}_{\mu_0}}\prod_{j=1}^m\|(\|F_{j,k_j}\|_{L^{p_j}})_{k_j}\|_{l^{q_j}_{\mu_j}}
  \\
  \sim &
  \|G\|_{W(L^{p_0},L^{q_0}_{\mu_0})}\prod_{j=1}^m\|F_j\|_{W(L^{p_j},L^{q_j}_{\mu_j})},
\end{split}
\ee
where we use $\eqref{thm-M1-cd3}$  in the last inequality.

If $p>q$, $\Om$ also satisfies $M1$.
In this case,
we only need to verify that $\eqref{thm-M1-cd1},\eqref{thm-M1-cd4}\Longrightarrow \eqref{thm-M1-cd0}$,
then $\eqref{thm-M1-cd1},\eqref{thm-M1-cd3}\Longrightarrow\eqref{thm-M1-cd0}$ follows by the fact
that
$\eqref{thm-M1-cd3}\Longleftrightarrow \eqref{thm-M1-cd4}$.
Using \eqref{thm-M1-1},\eqref{thm-M1-2} and the embedding relation $l^q\subset l^p$, we have
\be
\begin{split}
  &\|R_m(G,\vec{F})\|_{M^{p,q}_{\Om}(\rmdd)}
  \\
\lesssim &
\bigg(\sum_{\vec{k}\in \zmd}\bigg\|\bigg( \sum_{k_0\in \zd}
\bigg\|V_{\P}(R_m(G_{k_0},F_{1,k_0+k_1},\cdots,F_{m,k_0+k_m}))\Omba\bigg\|_{L^p}^p\Omab(k_0,\vec{k})^p\bigg)^{1/p}\bigg\|^q_{L^q}\bigg)^{1/q}
\\
\lesssim &
\bigg(\sum_{\vec{k}\in \zmd}\bigg\|\bigg( \sum_{k_0\in \zd}
\bigg\|V_{\P}(R_m(G_{k_0},F_{1,k_0+k_1},\cdots,F_{m,k_0+k_m}))\Omba\bigg\|_{L^p}^q\Omab(k_0,\vec{k})^q\bigg)^{1/q}\bigg\|^q_{L^q}\bigg)^{1/q}
\\
= &
\bigg(\sum_{\vec{k}\in \zmd}
\bigg( \sum_{k_0\in \zd}\bigg\|V_{\P}(R_m(G_{k_0},F_{1,k_0+k_1},\cdots,F_{m,k_0+k_m}))\bigg\|^q_{L^{p,q}_{\Omba}}\bigg)\Omab(k_0,\vec{k})^q\bigg)^{1/q}
\\
= &
\bigg(\sum_{\vec{k}\in \zmd}\bigg( \sum_{k_0\in \zd}\bigg\|R_m(G_{k_0},F_{1,k_0+k_1},\cdots,F_{m,k_0+k_m})\bigg\|_{M^{p,q}_{\Omba}}^q\bigg)\Omab(k_0,\vec{k})^q\bigg)^{1/q}.
\end{split}
\ee
From this, \eqref{thm-M1-3} and the condition $M1$ as follows
\be
\Omab(k_0,\vec{k})
\lesssim
\Omao(k_0)\prod_{j=1}^m\Omaj(k_j+k_0),
\ee
we have
\be
\begin{split}
&\|R_m(G,\vec{F})\|_{M^{p,q}_{\Om}(\rmdd)}
\\
\lesssim &
\bigg(\sum_{\vec{k}\in \zmd}\bigg( \sum_{k_0\in \zd}\|G_{k_0}\|^q_{L^{p_0}}\prod_{j=1}^m\|F_{j,k_0+k_j}\|_{L^{p_j}}^q\bigg)\Omao(k_0)^q\prod_{j=1}^m\Omaj(k_j+k_0)^q
\bigg)^{1/q}
\\
= &
\|\big(\|G_{k_0}\|_{L^{p_0}}\big)_{k_0}\|_{l^{q}_{\Omao}}
\prod_{j=1}^m\|\big(\|F_{j,k_j}\|_{L^{p_j}}\big)_{k_j}\|_{l^{q}_{\Omaj}}.
\end{split}
\ee
The desired conclusion follows by the above inequality and $\eqref{thm-M1-cd4}$:
\be
\begin{split}
 &\|R_m(G,\vec{F})\|_{M^{p,q}_{\Om}(\rmdd)}
\lesssim
\|\big(\|G_{k_0}\|_{L^{p_0}}\big)_{k_0}\|_{l^{q}_{\Omao}}
\prod_{j=1}^m\|\big(\|F_{j,k_j}\|_{L^{p_j}}\big)_{k_j}\|_{l^{q}_{\Omaj}}
\\
\lesssim &
\|\big(\|G_{k_0}\|_{L^{p_0}}\big)_{k}\|_{l^{q_0}_{\mu_0}}\prod_{j=1}^m\|\big(\|F_{j,k_j}\|_{L^{p_j}}\big)_{k_j}\|_{l^{q_j}_{\mu_j}}
\sim
\|G\|_{W(L^{p_0},L^{q_0}_{\mu_0})}\prod_{j=1}^m\|F_j\|_{W(L^{p_j},L^{q_j}_{\mu_j})}.
\end{split}
\ee
\end{proof}

\section{First characterizations of BRWF: decomposition in the time plane}
In this section, our goal is to characterize the BRWF boundedness by the corresponding embedding relations. 
As the characterization of BRWM in Theorem \ref{thm-M1}, we would like to deal with more general situations.
Let $\Om\in \mathscr{P}(\rr^{2(m+1)d})$. The following notations and conventions will be used in this article.
\bn
\item
$\Om_0(z_0,\z_0)
=
\Om((z_0,(\z_0,\cdots,\z_0)),(\z_0,\vec{0}))$.
\item
$\Om_{0,1}(x)=\Om_0(x,0)$, $\Om_{0,2}(\xi)=\Om_0(0,\xi)$.
\item
$\Om_i(z_i,\z_i)
=
\Om((\z_i,(\underbrace{0,\cdots,-z_i,0,\cdots,0}_{-z_i \text{ is the ith vector}})),(0,(\underbrace{0,\cdots,\z_i,0,\cdots,0}_{\z_i \text{ is the ith vector}})))$,
\ \ $i=1,2,\cdots,m$.
\item
$\Om_{i,1}(\xi)=\Om_i(0,\xi)$, $\Om_{i,2}(x)=\Om_i(x,0)$, $\Om_{i,0}(x)=\Om_i(x,x)$,
\ \ \ \ $i=1,2,\cdots,m$.
\en

\bn
\item[W0.]
$
\Om((z_0+\sum_{j=1}^m\z_j,(z_1+\z_0,\cdots,z_m+\z_0)),(\z_0,\vec{\z}))
\\
\lesssim
\Om((z_0,(\z_0,\cdots,\z_0)),(\z_0,\vec{0}))
\prod_{j=1}^m\Om((\z_i,(\underbrace{0,\cdots,z_i,0,\cdots,0}_{z_i \text{ is the ith vector}})),(0,(\underbrace{0,\cdots,\z_i,0,\cdots,0}_{\z_i \text{ is the ith vector}}))).
$
\item[W1.]
$\Om_0(x,\xi)\lesssim \Om_0(x,0)\Om_0(0,\xi)$.
\item[W2.]
$\Om_i(x,\xi)\lesssim \Om_i(x,0)\Om_i(0,\xi)$,\ \ \ \ $i=1,2,\cdots,m$.
\en

\begin{theorem}[First characterization of BRWF]\label{thm-F1}
	Assume $p_i, q_i, p, q \in (0,\fy]$,
	and that $\Om\in \mathscr{P}(\rr^{2(m+1)d})$, $\mu_i \in \mathscr{P}(\rd)$, $i=0,1,\cdots,m$.
	We have
	\ben\label{thm-F1-cd0}
	R_m: W(L^{p_0},L^{q_0}_{\mu_0})(\rd)\times\cdots \times W(L^{p_m},L^{q_m}_{\mu_m})(\rd)\longrightarrow \scrF M^{p,q}_{\Om}(\rmdd)
	\een
	implies
	\ben\label{thm-F1-cd1}
	W(L^{p_0},L^{q_0}_{\mu_0})(\rd)\subset \scrF M^{p,q}_{\Om_0}(\rd),
	\een
	and
	\ben\label{thm-F1-cd2}
	W(L^{p_i},L^{q_i}_{\mu_i})(\rd)\subset M^{p,q}_{\Om_i}(\rd),\ \ \ \ i=1,\cdots,m,
	\een
	where the converse direction is valid if $\Om$ satisfies condition W0.
	Moreover,
	the embedding relation \eqref{thm-F1-cd1} implies
	\ben\label{thm-F1-cd3}
	L^{p_0}(B_{\d})\subset \scrF L^p_{\Om_{0,1}}(\rd),\ \ \ \ l^{q_0}_{\mu_0}(\zd)\subset l^{q}_{\Om_{0,2}}(\zd),
	\een
	where the converse direction is valid if $\Om$ satisfies the condition $W1$.
	The embedding relations \eqref{thm-F1-cd2} imply the following embedding relations
	\ben\label{thm-F1-cd4}
	L^{p_i}(B_{\d})\subset \scrF^{-1}L^q_{\Om_{i,1}}(\rd),\ \ \ \ l^{q_i}_{\mu_i}(\zd)\subset l^{p}_{\Om_{i,2}}(\zd),\ \ \ \ l^{q_i}_{\mu_i}(\zd)\subset l^{q}_{\Om_{i,0}}(\zd)
	\een
	where the converse direction holds for 
	$p\leq q$
	if $\Om_i$ satisfies the condition $W2$,
	and holds for $p>q$ if $\Om_i$ satisfies $W2$ and $\Om_i(x,0)\lesssim \Om_i(x,x)$, $i=1,2,\cdots,d$.
\end{theorem}

\begin{remark}\label{rmk-F1}
	Let $\bbw\in \mathscr{P}(\rr^{(m+1)d})$ be a variables separated weight.
	Then the weight function $\Om=1\otimes\bbw\in \mathscr{P}(\rr^{2(m+1)d})$ in Theorem \ref{thm-F1-sp} satisfies all the conditions $Wi$ for $i=0,1,2$.
	Using this fact and Theorem \ref{thm-fsi}, Theorem \ref{thm-F1-sp} can be directly proved.
\end{remark}

\subsection{Some embedding relations}
\begin{lemma}[Local property of modulation space II]\label{lm-lpm2}
Let $0<p,q\leq \fy$, $\Om\in \scrP(\rdd)$.
For any $f$ with $\text{supp}\widehat{f}\subset B(0,R)$, $R>0$, we have
\be
\|f\|_{M^{p,q}_{\Om}}\sim_R \|f\|_{L^p_{\widetilde{\Om_{0}}}},
\ee
where $\widetilde{\Om_{0}}(x)=\Om(x,0)$ for $x\in \rd$.
\end{lemma}

\begin{proof}
  Let $\widehat{\phi}$ be a real-valued Schwartz function with $\text{supp}\widehat{\phi}\subset B(0,2R)$ and $\widehat{\phi}=1$ on $B(0,R)$. For sufficiently small $\al$ we have
  \be
  \begin{split}
    \|f\|_{M^{p,q}_{\Om}(\rd)}
    \sim &
    \|V_{\phi}f(\al k,\al n)\Om(\al k,\al n)\|_{l^{p,q}(\zd\times \zd)}
    \\
    = &
    \bigg(\sum_{n\in \zd}\big(\sum_{k\in \zd}|V_{\phi}f(\al k,\al n)\Om(\al k,\al n)|^p\big)^{q/p}\bigg)^{1/q}
    \\
    \sim_R &
    \sum_{n\in \zd}\big(\sum_{k\in \zd}|V_{\phi}f(\al k,\al n)\widetilde{\Om_{0}}(\al k)|^p\big)^{1/p},
    \end{split}
  \ee
  where in the last term we use the facts that
 only a finite number of $n$  makes the term $V_{\phi}f(\al k,\al n)$ nonzero,
 and that for these $n$ we have $\Om(\al k,\al n)\sim \widetilde{\Om_{0}}(\al k)$.
 By the definition of STFT,
 \be
 \begin{split}
 \sum_{k\in \zd}|V_{\phi}f(\al k,\al n)\widetilde{\Om_{0}}(\al k)|^p\big)^{1/p}
 \sim
 \sum_{k\in \zd}|\scrF^{-1}(\widehat{f} T_{\al n}\widehat{\phi})(\al k)\widetilde{\Om_{0}}(\al k)|^p\big)^{1/p}.
 \end{split}
 \ee
Note that $\text{supp}(\widehat{f}T_{\al n}\widehat{\phi})\subset B(0,R)$. For sufficiently small $\al$ we have
\be
\begin{split}
\sum_{k\in \zd}|\scrF^{-1}(\widehat{f}T_{\al n}\widehat{\phi})(\al k)\widetilde{\Om_{0}}(\al k)|^p\big)^{1/p}
\sim &
\int_{\rd}|\scrF^{-1}(\widehat{f}T_{\al n}\widehat{\phi})(\xi)\widetilde{\Om_{0}}(\xi)|^p\big)^{1/p},
\end{split}
\ee
where we use the sampling property of $\scrF^{-1}L^q_{\Om_{0}}$ for the functions with compact support on $B(0,R)$,
we refer to \cite[Proposition 3.1]{GuoChenFanZhao2019MMJ} for more details.
For above estimates, we conclude that
\be
  \begin{split}
    \|f\|_{M^{p,q}_{\Om}(\rd)}
    \sim_R &
    \sum_{n\in \zd}\int_{\rd}|\scrF^{-1}(\widehat{f}T_{\al n}\widehat{\phi})(\xi)\widetilde{\Om_{0}}(\xi)|^p\big)^{1/p}
    \\
    \geq &
    \int_{\rd}|\scrF^{-1}(\widehat{f}\widehat{\phi})(\xi)\widetilde{\Om_{0}}(\xi)|^p\big)^{1/p}
    =\|f\|_{L^p_{\widetilde{\Om_{0}}}}.
   \end{split}
\ee
On the other hand, by a convolution inequality (see \cite[Lemma 2.2]{GuoChenFanZhao2019MMJ}) with $\dot{q}=\min\{q,1\}$, we have
\be
\begin{split}
  \int_{\rd}|\scrF^{-1}(\widehat{f}T_{\al n}\widehat{\phi})(\xi)\widetilde{\Om_{0}}(\xi)|^p\big)^{1/p}
  \lesssim &
  \|f\|_{L^p_{\widetilde{\Om_{0}}}}
  \|\scrF^{-1}(T_{\al n}\widehat{\phi})\|_{L^{\dot{p}}_{v}}
  \\
  = &
  \|f\|_{L^p_{\widetilde{\Om_{0}}}}
  \|\phi\|_{L^{\dot{p}}_{v}}
  \lesssim
  \|f\|_{L^p_{\widetilde{\Om_{0}}}}.
\end{split}
\ee
From this, we conclude that
\be
  \begin{split}
    \|f\|_{M^{p,q}_{\Om}(\rd)}
    \sim_R &
    \sum_{n\in \zd}\int_{\rd}|\scrF^{-1}(\widehat{f}T_{\al n}\widehat{\phi})(\xi)\widetilde{\Om_{0}}(\xi)|^p\big)^{1/p}
    \lesssim
    \|f\|_{L^p_{\widetilde{\Om_{0}}}},
   \end{split}
\ee
where in the last inequality we use the fact that
 only a finite number of $n$  makes the terms in the summation nonzero.

\end{proof}

\begin{lemma}\label{lm-eb0}
Let $0<p_0,q_0,p,q\leq \fy$, $\Om_0\in \mathscr{P}(\rdd)$ and $\mu_0\in \scrP(\rd)$.
For any $\d>0$, we have
  \ben\label{lm-eb0-cd0}
  \begin{split}
  \|g\|_{\scrF M^{p,q}_{\Om_0}(\rd)}
  \sim
  \bigg(\sum_{k_0\in \zd}\|V_{\phi}g_{k_0}(\xi,-x)\Om_0(x,\xi)\|^q_{L^{p,q}(\rdd)}\bigg)^{1/q},
  \end{split}
  \een
  for $g=\sum_{k_0\in \zd}g_{k_0}$ with $\text{supp}g_{k_0}\subset B(k_0,\d)$,
  and a nonzero smooth function $\phi$ supported in $B(0,\d)$.
Moreover, for any $\d>0$, the following two statements are equivalent:
\bn
  \item The following embedding is valid:
  \ben\label{lm-eb0-cd1}
  W(L^{p_0},L^{q_0}_{\mu_0})(\rd)\subset \scrF M^{p,q}_{\Om_0}(\rd).
  \een
  \item
  Let $g=\sum_{k_0\in \zd}g_{k_0}\in W(L^{p_0},L^{q_0}_{\mu_0})$ with $\text{supp}g_{k_0}\subset B(k_0,\d)$.
  We have the following inequality:
  \ben\label{lm-eb0-cd2}
  \bigg(\sum_{k_0\in \zd}\|V_{\phi}g_{k_0}(\xi,-x)\Om_0(x,\xi)\|^q_{L^{p,q}(\rdd)}\bigg)^{1/q}
  \lesssim \|(\|g_{k_0}\|_{L^{p_0}})_{k_0}\|_{l^{q_0}_{\mu_0}}.
  \een
\en
\end{lemma}
\begin{proof}
  Using Lemma \ref{lm-spSTFT}, we write
  \be
  \begin{split}
    V_{\phi}g(\xi,-x)
    = &
    \sum_{k_0\in \zd}V_{\phi}g_{k_0}(\xi,-x)
    \\
    = &
    \sum_{k_0\in \zd}V_{\phi}g_{k_0}(\xi,-x)\chi_{B(k_0,2\d)}(\xi).
  \end{split}
  \ee
  Then, \eqref{lm-eb0-cd0} follows by the definition of $\scrF M^{p,q}_{\Om_0}(\rd)$ and
  the fact that
  the supports of the above functions are almost separated from each other.
  Using \eqref{lm-eb0-cd0} and the fact that
  $\|g\|_{W(L^{p_0},L^{q_0}_{\mu_0})}\sim \|(\|g_{k_0}\|_{L^{p_0}})_{k_0}\|_{l^{q_0}_{\mu_0}}$,
  we complete the proof of $\eqref{lm-eb0-cd1}\Longrightarrow \eqref{lm-eb0-cd2}$.
  The converse direction follows by a similar and simpler reduction as in the proof of Theorem \ref{thm-M0}.
\end{proof}

\begin{lemma}\label{lm-eb0s}
Let $0<p_0,q_0,p,q\leq \fy$, $\Om_0\in \mathscr{P}(\rdd)$ and $\mu_0\in \scrP(\rd)$.
Denote $\Om_{0,1}(x)=\Om_0(x,0)$, $\Om_{0,2}(\xi)=\Om_0(0,\xi)$.
Then, for any $\d>0$,
\ben\label{lm-eb0s-cd0}
W(L^{p_0},L^{q_0}_{\mu_0})(\rd)\subset \scrF M^{p,q}_{\Om_0}(\rd)
\een
implies
\ben\label{lm-eb0s-cd1}
L^{p_0}(B_{\d})\subset \scrF M^{p,q}_{\Om_{0,1}\otimes 1}(\rd)
\een
and
\ben\label{lm-eb0s-cd2}
l^{q_0}_{\mu_0}(\zd)\subset l^{q}_{\Om_{0,2}}(\zd).
\een
The converse direction is valid if $\Om_0(x,\xi)\lesssim \Om_0(x,0)\Om_0(0,\xi)$,
and in this case, we have the equivalent relation
$\eqref{lm-eb0s-cd0}\Longleftrightarrow \eqref{lm-eb0s-cd1},\eqref{lm-eb0s-cd2}$.
Moreover, the embedding \eqref{lm-eb0s-cd1} is equivalent to
\ben\label{lm-eb0s-cd3}
L^{p_0}(B_{\d})\subset \scrF L^p_{\Om_{0,1}}.
\een
\end{lemma}
\begin{proof}
  The relation $\eqref{lm-eb0s-cd0}\Longrightarrow \eqref{lm-eb0s-cd1}$ follows by Lemma \ref{lm-eb0} and the fact
  \be
  \begin{split}
  \|V_{\phi}g_{0}(\xi,-x)\Om_0(x,\xi)\|_{L^{p,q}(\rdd)}
  \sim &
  \|V_{\phi}g_{0}(\xi,-x)\Om_0(x,0)\|_{L^{p,q}(\rdd)}
  \\
  = &
  \|V_{\phi}g_{0}(\xi,-x)(\Om_{0,1}\otimes 1)(x,\xi)\|_{L^{p,q}(\rdd)}
  =
  \|g_0\|_{\scrF M^{p,q}_{\Om_{0,1}\otimes 1}},
  \end{split}
  \ee
  for $g_0$ supported in $B(0,\d)$.

  Next, we turn to the proof of $\eqref{lm-eb0s-cd0}\Longrightarrow \eqref{lm-eb0s-cd2}$.
  For any nonnegative truncated (only finite nonzero items) sequence
  $\vec{a}=(a_{k_0})_{k_0\in \zd}$, we set
\be
g=\sum_{k_0\in \zd}a_{k_0}T_{k_0}\va=: \sum_{k_0\in \zd}g_{k_0},
\ee
where $\va$ is chosen to be a nonzero smooth function supported in $B(0,\d)$ with some small constant $\d>0$.
Let $\phi$ be a nonzero smooth function supported in $B(0,\d)$.
We have
  \be
    V_{\phi}g(\xi,-x)
    =
    \sum_{k_0\in \zd}V_{\phi}g_{k_0}(\xi,-x)\chi_{B(k_0,2\d)}(\xi).
  \ee
Using this and Lemma \ref{lm-eb0}, we conclude that
\be
\begin{split}
  \|g\|_{\scrF M^{p,q}_{\Om_{0}}}
  \sim &
  \bigg(\sum_{k_0\in \zd}\|V_{\phi}g_{k_0}(\xi,-x)\Om_0(x,\xi)\|^q_{L^{p,q}(\rdd)}\bigg)^{1/q}
  \\
  = &
  \bigg(\sum_{k_0\in \zd}a_{k_0}^q\|V_{\phi}T_{k_0}\va(\xi,-x)\Om_0(x,\xi)\chi_{B(k_0,2\d)}(\xi)\|^q_{L^{p,q}(\rdd)}\bigg)^{1/q}
  \\
  \gtrsim &
  \bigg(\sum_{k_0\in \zd}a_{k_0}^q\|V_{\phi}\va(\xi-k_0,-x)\Om_0(x,\xi)\chi_{B(k_0,2\d)}(\xi)\chi_{B(0,\d)}(x)\|^q_{L^{p,q}(\rdd)}\bigg)^{1/q}.
\end{split}
\ee
Using the fact that
\be
\Om_0(x,\xi)\chi_{B(k_0,2\d)}(\xi)\chi_{B(0,\d)}(x)\sim \Om_{0,2}(k_0)\chi_{B(k_0,2\d)}(\xi)\chi_{B(0,\d)}(x),
\ee
the last term of the above inequality is equivalent to
\be
\begin{split}
  &\bigg(\sum_{k_0\in \zd}a_{k_0}^q\Om_{0,2}(k_0)^q\|V_{\phi}\va(\xi-k_0,-x)\chi_{B(k_0,2\d)}(\xi)\chi_{B(0,\d)}(x)\|^q_{L^{p,q}(\rdd)}\bigg)^{1/q}
  \\
  \sim &
  \bigg(\sum_{k_0\in \zd}a_{k_0}^q\Om_{0,2}(k_0)^q\|V_{\phi}\va(\xi,-x)\chi_{B(0,\d)}(x)\|^q_{L^{p,q}(\rdd)}\bigg)^{1/q}
  \sim
  \bigg(\sum_{k_0\in \zd}a_{k_0}^q\Om_{0,2}(k_0)^q\bigg)^{1/q}.
\end{split}
\ee
The desired conclusion follows by this and the fact
\be
\|g\|_{W(L^{p_0},L^{q_0}_{\mu_0})}\sim \|(\|g_{k_0}\|_{L^{p_0}})_{k_0}\|_{l^{q_0}_{\mu_0}}\sim \|a_{k_0}\|_{l^{q_0}_{\mu_0}}.
\ee

Conversely, if \eqref{lm-eb0s-cd1} and \eqref{lm-eb0s-cd2} hold, we only need to verify \eqref{lm-eb0s-cd0} by
\ben\label{lm-eb0s-1}
  \bigg(\sum_{k_0\in \zd}\|V_{\phi}G_{k_0}(\xi,-x)\Om_0(x,\xi)\|^q_{L^{p,q}(\rdd)}\bigg)^{1/q}
  \lesssim \|(\|G_{k_0}\|_{L^{p_0}})_{k_0}\|_{l^{q_0}_{\mu_0}}
\een
for $G=\sum_{k_0\in \zd}G_{k_0}$ with $\text{supp}G_{k_0}\subset B(k_0,\d)$.
By the fact that
\be
\begin{split}
  |V_{\phi}G_{k_0}(\xi,-x)\Om_0(x,\xi)|\lesssim |V_{\phi}G_{k_0}(\xi,-x)|\Om_{0,1}(x)\Om_{0,2}(k_0),
\end{split}
\ee
the left term of \eqref{lm-eb0s-1} can be dominated from above by
\be
\begin{split}
  &\bigg(\sum_{k_0\in \zd}\|V_{\phi}G_{k_0}(\xi,-x)\Om_{0,1}(x)\|^q_{L^{p,q}(\rdd)}\Om_{0,2}(k_0)^q\bigg)^{1/q}
  \\
  \sim &
  \bigg(\sum_{k_0\in \zd}\|G_{k_0}\|^q_{\scrF M^{p,q}_{\Om_{0,1}\otimes 1}(\rd)}\Om_{0,2}(k_0)^q\bigg)^{1/q}
  =
  \bigg(\sum_{k_0\in \zd}\|T_{-k_0}G_{k_0}\|^q_{\scrF M^{p,q}_{\Om_{0,1}\otimes 1}(\rd)}\Om_{0,2}(k_0)^q\bigg)^{1/q}.
\end{split}
\ee
Observe that $\text{supp}T_{-k_0}G_{k_0}\subset B(0,\d)$. We use \eqref{lm-eb0s-cd1} and \eqref{lm-eb0s-cd2} to conclude that
\be
\begin{split}
  &\bigg(\sum_{k_0\in \zd}\|T_{-k_0}G_{k_0}\|^q_{\scrF M^{p,q}_{\Om_{0,1}\otimes 1}(\rd)}\Om_{0,2}(k_0)^q\bigg)^{1/q}
  \\
  \lesssim &
  \bigg(\sum_{k_0\in \zd}\|T_{-k_0}G_{k_0}\|^q_{L^{p_0}(\rd)}\Om_{0,2}(k_0)^q\bigg)^{1/q}
  \lesssim
  \|(\|G_{k_0}\|_{L^{p_0}})_{k_0}\|_{l^{q_0}_{\mu_0}}.
\end{split}
\ee

Finally, the equivalent relation $\eqref{lm-eb0s-cd1}\Longleftrightarrow \eqref{lm-eb0s-cd3}$ follows by Lemma \ref{lm-lpm2}.
\end{proof}

\begin{lemma}\label{lm-ebi}
Let $0<p_i,q_i,p,q\leq \fy$, $\Om_i\in \mathscr{P}(\rdd)$, $\mu_i\in \scrP(\rd)$.
For any $\d>0$, we have
  \ben\label{lm-ebi-cd0}
  \begin{split}
  \|f_i\|_{M^{p,q}_{\Om_i}(\rd)}
  \sim
  \bigg\|\bigg(\sum_{k_i\in \zd}\|V_{\phi}f_{i,k_i}\Om_i\|^p_{L^{p}(\rd)}\bigg)^{1/p}\bigg\|_{L^q},
  \end{split}
  \een
  for $f_i=\sum_{k_i\in \zd}f_{i,k_i}$, with $\text{supp}f_{i,k_i}\subset B(k_i,\d)$,
  and a nonzero smooth function $\phi$ supported in $B(0,\d)$.
Moreover, for any $\d>0$, the following two statements are equivalent:
\bn
  \item The following embedding is valid:
  \ben\label{lm-ebi-cd1}
  W(L^{p_i},L^{q_i}_{\mu_i})(\rd)\subset M^{p,q}_{\Om_i}(\rd).
  \een
  \item
  Let $f_i=\sum_{k_i\in \zd}f_{i,k_i}\in W(L^{p_i},L^{q_i}_{\mu_i})$ with $\text{supp}f_{i,k_i}\subset B(k_i,\d)$.
  We have the following inequality:
  \ben\label{lm-ebi-cd2}
  \bigg\|\bigg(\sum_{k_i\in \zd}\|V_{\phi}f_{i,k_i}\Om_i\|^p_{L^{p}(\rd)}\bigg)^{1/p}\bigg\|_{L^q}
  \lesssim \|(\|f_{i,k_i}\|_{L^{p_i}})_{k_i}\|_{l^{q_i}_{\mu_i}}.
  \een
\en
\end{lemma}
This lemma can be proved by the similar method as in the proofs of Theorem \ref{thm-M0} and Lemma \ref{lm-eb0}, we omit the details here.

\begin{lemma}\label{lm-ebis}
Let $0<p_i,q_i,p,q\leq \fy$, $\Om_i\in \mathscr{P}(\rdd)$, $\mu_i\in \scrP(\rd)$.
Denote $\Om_{i,1}(\xi)=\Om_i(0,\xi)$, $\Om_{i,2}(x)=\Om_i(x,0)$ and $\Om_{i,0}(x)=\Om_i(x,x)$.
Then, for any $\d>0$, we have
\ben\label{lm-ebis-cd0}
W(L^{p_i},L^{q_i}_{\mu_i})(\rd)\subset M^{p,q}_{\Om_i}(\rd)
\een
implies
\ben\label{lm-ebis-cd1}
L^{p_i}(B_{\d})\subset M^{p,q}_{1\otimes \Om_{i,1}}(\rd)
\een
and
\ben\label{lm-ebis-cd2}
l^{q_i}_{\mu_i}(\zd)\subset l^{p}_{\Om_{i,2}}(\zd),\ \ \ l^{q_i}_{\mu_i}(\zd)\subset l^{q}_{\Om_{i,0}}(\zd).
\een
The converse direction holds for $p\leq q$ if $\Om_i(x,\xi)\lesssim \Om_{i}(x,0)\Om_{i}(0,\xi)$,
and holds for $p> q$ if $\Om$ satisfies $\Om_i(x,0)\lesssim \Om_{i}(x,x)$ and $\Om_i(x,\xi)\lesssim \Om_{i}(x,0)\Om_{i}(0,\xi)$.

Moreover, the embedding \eqref{lm-ebis-cd1} is equivalent to
\ben\label{lm-ebis-cd3}
L^{p_i}(B_{\d})\subset \scrF^{-1}L^q_{\Om_{i,1}}.
\een
\end{lemma}
\begin{proof}
  The relation $\eqref{lm-ebis-cd0}\Longrightarrow \eqref{lm-ebis-cd1}$ follows by Lemma \ref{lm-ebi} and the fact that
  \be
  \begin{split}
  V_{\phi}g_{0}(x,\xi)\Om_i(x,\xi)
  = &
  V_{\phi}g_{0}(x,\xi)\chi_{B(0,2\d)}(x)\Om_i(x,\xi)
  \\
  \sim &
  V_{\phi}g_{0}(x,\xi)\Om_i(0,\xi)
  =
  V_{\phi}g_{0}(x,\xi)\Om_{i,1}(\xi)
  =
  V_{\phi}g_{0}(x,\xi)(1\otimes\Om_{i,1})(x,\xi),
  \end{split}
  \ee
  for $g_0$ supported in $B(0,\d)$.

  Next, we turn to the proof of $\eqref{lm-ebis-cd0}\Longrightarrow \eqref{lm-ebis-cd2}$.
  For any nonnegative truncated (only finite nonzero items) sequence
  $\vec{a}=(a_{k})_{k\in \zd}$, we set
\be
g=\sum_{k_0\in \zd}a_{k_0}T_{k_0}\va=: \sum_{k_0\in \zd}g_{k_0},
\ee
where $\va$ is chosen to be a nonzero smooth function supported in $B(0,\d)$ with some small positive constant $\d>0$.
Let $\phi$ be a nonzero smooth function supported in $B(0,\d)$.
We have
  \be
    V_{\phi}g(x,\xi)
    =
    \sum_{k_0\in \zd}V_{\phi}g_{k_0}(x,\xi)\chi_{B(k_0,2\d)}(x).
  \ee
Using this and Lemma \ref{lm-ebi}, we conclude that
\be
\begin{split}
  \|g\|_{M^{p,q}_{\Om_{i}}}
  \sim &
  \bigg\|\bigg(\sum_{k_0\in \zd}\|V_{\phi}g_{k_0}(x,\xi)\Om_i(x,\xi)\|^p_{L^{p}(\rd)}\bigg)^{1/p}\bigg\|_{L^q}
  \\
  \gtrsim &
  \bigg\|\bigg(\sum_{k_0\in \zd}a_{k_0}^p\|V_{\phi}\va(x-k_0,\xi)\Om_i(x,\xi)\chi_{B(k_0,2\d)}(x)\chi_{B(0,\d)}(\xi)\|^p_{L^p(\rd)}\bigg)^{1/p}\bigg\|_{L^q}.
\end{split}
\ee
Using the fact that
\be
\Om_i(x,\xi)\chi_{B(k_0,2\d)}(x)\chi_{B(0,\d)}(\xi)
\sim \Om_{i,2}(k_0)\chi_{B(k_0,2\d)}(x)\chi_{B(0,\d)}(\xi),
\ee
the last term of the above inequality is equivalent to
\be
\begin{split}
  &\bigg\|\bigg(\sum_{k_0\in \zd}a_{k_0}^p\Om_{i,2}(k_0)^p\|V_{\phi}\va(x-k_0,\xi)\chi_{B(k_0,2\d)}(x)\chi_{B(0,\d)}(\xi)\|^p_{L^p(\rd)}\bigg)^{1/p}\bigg\|_{L^q}
  \\
  \sim &
  \bigg(\sum_{k_0\in \zd}a_{k_0}^p\Om_{i,2}(k_0)^p\bigg)^{1/p}\|V_{\phi}\va(x,\xi)\chi_{B(0,\d)}(\xi)\|^q_{L^{p,q}(\rdd)}
  \sim
  \bigg(\sum_{k_0\in \zd}a_{k_0}^p\Om_{i,2}(k_0)^p\bigg)^{1/p}.
\end{split}
\ee
The embedding relation $l^{q_i}_{\mu_i}(\zd)\subset l^{p}_{\Om_{i,2}}(\zd)$ follows by this and the fact
\be
\|g\|_{W(L^{p_i},L^{q_i}_{\mu_i})}\sim \|(\|g_{k_0}\|_{L^{p_i}})_{k_0}\|_{l^{q_i}_{\mu_i}}\sim \|a_{k_0}\|_{l^{q_i}_{\mu_i}}.
\ee

On the other hand, for any nonnegative truncated (only finite nonzero items) sequence
$\vec{b}=(b_{k})_{k\in \zd}$, we set
\be
h=\sum_{k_0\in \zd}b_{k_0}T_{k_0}M_{k_0}\psi=: \sum_{k_0\in \zd}h_{k_0},
\ee
where $\psi$ is a nonzero smooth function with $\text{supp}\widehat{\psi}\subset B(0,\d)$ for some small constant $\d>0$.
Let $\phi$ be a smooth function  with $\text{supp}\widehat{\phi}\subset B(0,\d)$.
We have
\be
V_{\phi}h(x,\xi)
=
\sum_{k_0\in \zd}V_{\phi}h_{k_0}(x,\xi)\chi_{B(k_0,2\d)}(\xi).
\ee
From this and the definition of modulation space, we conclude that
\be
\begin{split}
	\|h\|_{M^{p,q}_{\Om_{i}}}
	\sim &
	\bigg(\sum_{k_0\in \zd}\|V_{\phi}h_{k_0}(x,\xi)\Om_i(x,\xi)\|_{L^{p,q}(\rd\times \rd)}^q\bigg)^{1/q}
	\\
	= &
	\bigg(\sum_{k_0\in \zd}b_{k_0}^q\|V_{\phi}\psi(x-k_0,\xi-k_0)\Om_i(x,\xi)\chi_{B(k_0,2\d)}(\xi)\|_{L^{p,q}(\rd\times \rd)}^q\bigg)^{1/q}
	\\
	\gtrsim &
	\bigg(\sum_{k_0\in \zd}b_{k_0}^q\Om_{i,0}(k_0)^q\|V_{\phi}\psi(x-k_0,\xi-k_0)\chi_{B(k_0,2\d)}(x)\chi_{B(k_0,2\d)}(\xi)\|_{L^{p,q}(\rd\times \rd)}^q\bigg)^{1/q}
	\\
	= &
	\bigg(\sum_{k_0\in \zd}b_{k_0}^q\Om_{i,0}(k_0)^q\bigg)^{1/q} \|V_{\phi}\psi(x,\xi) \chi_{B(0,2\d)}(x)\|_{L^{p,q}}
	\sim \|\vec{b}\|_{l^q_{\Om_{i,0}}}.
\end{split}
\ee

By a direct calculation with the repid decay of the Schwartz function, we conclude that
\be
\begin{split}
	&\big\|T_{l}\chi_{Q_0}\sum_{k_0\in \zd}b_{k_0}T_{k_0}M_{k_0}\psi\big\|_{L^p}
	\lesssim 
	\sum_{k_0\in \zd}b_{k_0}\lan l-k_0\ran^{-\scrL}\big\|T_{l}\chi_{Q_0}\big\|_{L^p}
	\lesssim 
	\sum_{k_0\in \zd}b_{k_0}\lan l-k_0\ran^{-\scrL},
\end{split}
\ee
where we use $\scrL$ to denote a sufficiently large constant.

From this and the weighted Young's inequality, we obtain the following estimate
\be
\begin{split}
	\|h\|_{W(L^{p_i},L^{q_i}_{\mu_i})}
	= &
	\big\|\big(\big\|T_{l}\chi_{Q_0}\sum_{k_0\in \zd}b_{k_0}T_{k_0}M_{k_0}\psi\big\|_{L^p}\big)_{l}\big\|_{l^{q_i}_{\mu_i}}
	\\
	\lesssim &
	\big\|\big(\sum_{k_0\in \zd}b_{k_0}\lan l-k_0\ran^{-\scrL}\big)_{l}\big\|_{l^{q_i}_{\mu_i}}
	\lesssim
	\|\vec{b}\|_{l^{q_i}_{\mu_i}} \|\lan l\ran^{-\scrL}\|_{l^{\dot{q_i}}_v}\lesssim \|\vec{b}\|_{l^{q_i}_{\mu_i}}.
\end{split}
\ee

Then, the desired embedding $l^{q_i}_{\mu_i}(\zd)\subset l^{q}_{\Om_{i,0}}(\zd)$  follows by
\be
\|\vec{b}\|_{l^q_{\Om_{i,0}}}\lesssim \|h\|_{M^{p,q}_{\Om_{i}}}
\lesssim \|h\|_{W(L^{p_i},L^{q_i}_{\mu_i})}
\lesssim \|\vec{b}\|_{l^{q_i}_{\mu_i}}.
\ee

Conversely, if \eqref{lm-ebis-cd1} and \eqref{lm-ebis-cd2} hold, we only need to verify \eqref{lm-ebis-cd0} by
\ben\label{lm-ebis-1}
  \bigg\|\bigg(\sum_{k_0\in \zd}\|V_{\phi}G_{k_0}\Om_i\|^p_{L^{p}(\rd)}\bigg)^{1/p}\bigg\|_{L^q}
  \lesssim \|(\|G_{k_0}\|_{L^{p_i}})_{k_0}\|_{l^{q_i}_{\mu_i}},
\een
for $G=\sum_{k_0\in \zd}G_{k_0}$ with $\text{supp}G_{k_0}\subset B(k_0,\d)$.

If $p\leq q$ and $\Om_i(x,\xi)\lesssim \Om_{i}(x,0)\Om_{i}(0,\xi)$,
by the fact that
\be
\begin{split}
  |V_{\phi}G_{k_0}(x,\xi)\Om_i(x,\xi)|\lesssim |V_{\phi}G_{k_0}(x,\xi)|\Om_{i,1}(\xi)\Om_{i,2}(k_0),
\end{split}
\ee
the left term of \eqref{lm-ebis-1} can be dominated from above by
\be
\begin{split}
  &\bigg\|\bigg(\sum_{k_0\in \zd}\|V_{\phi}G_{k_0}(x,\xi)\Om_{i,1}(\xi)\|^p_{L^{p}(\rd)}\Om_{i,2}(k_0)^p\bigg)^{1/p}\bigg\|_{L^q}
  \\
  \lesssim &
  \bigg(\sum_{k_0\in \zd}\|V_{\phi}G_{k_0}(x,\xi)\Om_{i,1}(\xi)\|^p_{L^{p,q}(\rdd)}\Om_{i,2}(k_0)^p\bigg)^{1/p}
  \\
  \sim &
  \bigg(\sum_{k_0\in \zd}\|G_{k_0}\|^p_{M^{p,q}_{1\otimes \Om_{i,1}}(\rd)}\Om_{i,2}(k_0)^p\bigg)^{1/p}
  =
  \bigg(\sum_{k_0\in \zd}\|T_{-k_0}G_{k_0}\|^p_{M^{p,q}_{1\otimes \Om_{i,1}}(\rd)}\Om_{i,2}(k_0)^p\bigg)^{1/p}.
\end{split}
\ee
Observe that $\text{supp}T_{-k_0}G_{k_0}\subset B(0,\d)$. We use \eqref{lm-ebis-cd1} and \eqref{lm-ebis-cd2} to conclude that
\be
\begin{split}
  &\bigg(\sum_{k_0\in \zd}\|T_{-k_0}G_{k_0}\|^p_{M^{p,q}_{1\otimes \Om_{i,1}}(\rd)}\Om_{i,2}(k_0)^p\bigg)^{1/p}
  \\
  \lesssim &
  \bigg(\sum_{k_0\in \zd}\|T_{-k_0}G_{k_0}\|^p_{L^{p_i}(\rd)}\Om_{i,2}(k_0)^p\bigg)^{1/p}
  \lesssim
  \|(\|G_{k_0}\|_{L^{p_i}(\rd)})_{k_0}\|_{l^{q_i}_{\mu_i}}.
\end{split}
\ee

If $p> q$ and $\Om_i(x,0)\lesssim \Om_{i}(x,x)$, the left term of \eqref{lm-ebis-1} can be dominated from above by
\be
\begin{split}
	&\bigg\|\bigg(\sum_{k_0\in \zd}\|V_{\phi}G_{k_0}(x,\xi)\Om_{i,1}(\xi)\|^p_{L^{p}(\rd)}\Om_{i,2}(k_0)^p\bigg)^{1/p}\bigg\|_{L^q}
	\\
	\lesssim &
	\bigg(\sum_{k_0\in \zd}\|V_{\phi}G_{k_0}(x,\xi)\Om_{i,1}(\xi)\|^q_{L^{p,q}(\rdd)}\Om_{i,2}(k_0)^q\bigg)^{1/q}
	\\
	\sim &
	\bigg(\sum_{k_0\in \zd}\|G_{k_0}\|^q_{M^{p,q}_{1\otimes \Om_{i,1}}(\rd)}\Om_{i,2}(k_0)^q\bigg)^{1/q}
	\lesssim
	\bigg(\sum_{k_0\in \zd}\|G_{k_0}\|^q_{M^{p,q}_{1\otimes \Om_{i,1}}(\rd)}\Om_{i,0}(k_0)^q\bigg)^{1/q}.
\end{split}
\ee
Using this, and the fact $\|G_{k_0}\|_{M^{p,q}_{1\otimes \Om_{i,1}}(\rd)}\lesssim  \|G_{k_0}\|_{L^{p_i}(\rd)}$ mentioned above,
 and the embedding $l^{q_i}_{\mu_i}(\zd)\subset l^{q}_{\Om_{i,0}}(\zd)$, we have that 
the left term of \eqref{lm-ebis-1} can be further dominated from above by
\be
\bigg(\sum_{k_0\in \zd}\|G_{k_0}\|^{q_i}_{L^{p_i}(\rd)}\mu_{i}(k_0)^{q_i}\bigg)^{1/q_i}=\|(\|G_{k_0}\|_{L^{p_i}})_{k_0}\|_{l^{q_i}_{\mu_i}}.
\ee

Finally, the equivalent relation $\eqref{lm-ebis-cd1}\Longleftrightarrow \eqref{lm-ebis-cd3}$ follows by Lemma \ref{lm-lpm}.
\end{proof}

\subsection{First characterization for BRWF}
\begin{proof}[Proof of Theorem \ref{thm-F1}]
By Lemma \ref{lm-eb0s},
we have
$\eqref{thm-F1-cd1}\Longrightarrow \eqref{thm-F1-cd3}$,
where the converse direction is valid if $\Om_0(x,\xi)\lesssim \Om_0(x,0)\Om_0(0,\xi)$.
By Lemma \ref{lm-ebis},
we obtain that
$\eqref{thm-F1-cd2}\Longrightarrow\eqref{thm-F1-cd4}$,
where the opposite direction is valid for $p\geq q$
if $\Om_i(x,\xi)\lesssim \Om_i(x,0)\Om_i(0,\xi)$, and for $p>q$ if $\Om_i(x,0)\lesssim \Om_i(x,x)$ and $\Om_i(x,\xi)\lesssim \Om_{i}(x,0)\Om_{i}(0,\xi)$.
Thus, we only need to verify that
$\eqref{thm-F1-cd0} \Longrightarrow \eqref{thm-F1-cd1},\eqref{thm-F1-cd2}$,
where the converse direction is valid if $\Om$ satisfies condition $W_0$.
We divide the proof into two parts.

\textbf{``Only if'' part.}

Let $\Phi=R_m(\phi,\cdots,\phi)$, where $\phi$  is a smooth function which is supported in $B_{2\d}$ and satisfies $\phi(\xi)=1$ on $B_{\d}$.
By the definition of $\scrF M^{p,q}_{\Om}$, we have
\be
\begin{split}
  &\|R_m(f_0,\vec{f})\|_{\scrF M^{p,q}_{\Om}(\rmdd)}
  =
  \|V_{\Phi}(R_m(f_0,\vec{f}))((\z_0,\vec{\z}),(z_0,\vec{z}))\Om((-z_0,-\vec{z}),(\z_0,\vec{\z}))\|_{L^{p,q}(\rmdd\times\rmdd)}
  \\
  = &
  \|V_{\phi}f_0(\z_0,z_0+\sum_{j=1}^m\z_j)\prod_{j=1}^mV_{\phi}f_j(z_j+\z_0,\z_j)
  \Om((-z_0,-\vec{z}),(\z_0,\vec{\z}))\|_{L^{p,q}(\rmdd\times\rmdd)}
  \\
  = &
  \|V_{\phi}f_0(\z_0,z_0)\prod_{j=1}^mV_{\phi}f_j(z_j,\z_j)
  \Om((-z_0+\sum_{j=1}^m\z_j,(-z_1+\z_0,\cdots,-z_m+\z_0)),(\z_0,\vec{\z}))\|_{L^{p,q}(\rmdd\times\rmdd)}.
\end{split}
\ee

Let $f_j=h$ for all $1\leq j\leq m$, where $h\in \calS(\rd)$ with $\int_{\rd}h(x)dx=1$.
Observe that $V_{\phi}f_j(0,0)=\int_{\rd}h(x)dx=1$,
  we use the continuous property of STFT to obtain that,
  for sufficiently small $\d$,
  \be
  V_{\phi}f_j(z_j,\z_j)\gtrsim 1,\ \ \ z_j, \z_j\in B(0,\d),\ \ \ \ \ j=1,\cdots,m.
  \ee
From this, we have
  \ben\label{thm-F1-1}
  \begin{split}
    |V_{\phi}f_0(\z_0,z_0)\prod_{j=1}^mV_{\phi}f_j(z_j,\z_j)|
    \gtrsim
    |V_{\phi}f_0(\z_0,z_0)\prod_{j=1}^m\chi_{B_{\d}}(z_j)\chi_{B_{\d}}(\z_j)|.
  \end{split}
  \een
  Observe that for $z_j, \z_j\in B(0,\d)$,
  \be
  \begin{split}
  &\Om((-z_0+\sum_{j=1}^m\z_j,(-z_1+\z_0,\cdots,-z_m+\z_0)),(\z_0,\vec{\z}))
  \\
  \sim &
  \Om((-z_0,(\z_0,\cdots,\z_0)),(\z_0,\vec{0}))
  =
  \Om_0(-z_0,\z_0).
  \end{split}
  \ee
  From the above two estimates, we get
  \be
  \begin{split}
    &\big|V_{\phi}f_0(\z_0,z_0)\prod_{j=1}^mV_{\phi}f_j(z_j,\z_j)
  \Om((-z_0+\sum_{j=1}^m\z_j,(-z_1+\z_0,\cdots,-z_m+\z_0)),(\z_0,\vec{\z}))\big|
    \\
    \gtrsim &
    V_{\phi}f_0(\z_0,z_0)\prod_{j=1}^m\big(\chi_{B_{\d}}(z_j)\chi_{B_{\d}}(\z_j)\big)\Om_0(-z_0,\z_0).
  \end{split}
  \ee
  Hence, we have the following estimate of $\|R_m(f_0,\vec{f})\|_{\scrF M^{p,q}_{\Om}(\rmdd)}$:
  \be
  \begin{split}
    &\|R_m(f_0,\vec{f})\|_{\scrF M^{p,q}_{\Om}(\rmdd)}
    \\
    \gtrsim &
    \big\|V_{\phi}f_0(\z_0,z_0)\prod_{j=1}^m\big(\chi_{B_{\d}}(z_j)\chi_{B_{\d}}(\z_j)\big)\Om_0(-z_0,\z_0)\big\|_{L^{p,q}(\rmdd\times\rmdd)}
    \\
    = &
    \big\|V_{\phi}f_0(\z_0,z_0)
\Om_0(-z_0,\z_0)\big\|_{L^{p,q}(\rd\times\rd)} \prod_{j=1}^m\big\|\chi_{B_{\d}}(z_j)\chi_{B_{\d}}(\z_j)\big\|_{L^{p,q}(\rd\times\rd)}
    \\
    \gtrsim &
    \big\|V_{\phi}f_0(\z_0,z_0)\Om_0(-z_0,\z_0)\big\|_{L^{p,q}(\rd\times\rd)}=\|f_0\|_{\scrF M^{p,q}_{\Om_0}(\rd)}.
  \end{split}
  \ee
  From this and \eqref{thm-F1-cd0}, we obtain
  \be
  \begin{split}
  \|f_0\|_{\scrF M^{p,q}_{\Om_0}(\rd)}
  \lesssim &
  \|R_m(f_0,\vec{f})\|_{\scrF M^{p,q}_{\Om}(\rmdd)}
  \\
  \lesssim &
  \|f_0\|_{W(L^{p_0}_{\mu_0},L^{q_0})}\prod_{j=1}^m\|f_j\|_{W(L^{p_j}_{\mu_j},L^{q_j})}
  \\
  \lesssim &
  \|f_0\|_{W(L^{p_0}_{\mu_0},L^{q_0})}\prod_{j=1}^m\|h\|_{W(L^{p_j}_{\mu_j},L^{q_j})}
  \lesssim
  \|f_0\|_{W(L^{p_0}_{\mu_0},L^{q_0})}.
  \end{split}
  \ee
  This yields the embedding relation $W(L^{p_0}_{\mu_0},L^{q_0})\subset \scrF M^{p,q}_{\Om_0}(\rd)$.

  For any fixed $1\leq i\leq m$, denote $\G_i=\{0,1,\cdots,m\}\bs\{i\}$.
  Take $f_j=h$ for all $j\in \G_i$, where $h\in \calS(\rd)$ with $\int_{\rd}h(x)dx=1$.
  For $z_j, \z_j\in B_{\d}\  (j\in \G_i)$ with sufficiently small $\d$, we have
  \be
  |V_{\phi}f_j(z_j,\z_j)|\gtrsim 1,\ \ \ j\in \G_i.
  \ee
Thus,
  \ben\label{thm-F1-2}
  \begin{split}
    |V_{\phi}f_0(\z_0,z_0)\prod_{j=1}^mV_{\phi}f_j(z_j,\z_j)|
    \gtrsim 
    |V_{\phi}f_i(z_i,\z_i)|
    \prod_{j\in \G_i}\chi_{B_{\d}}(z_j)\chi_{B_{\d}}(\z_j).
  \end{split}
  \een
  Observe that for $z_j, \z_j\in B_{\d}\  (j\in \G_i)$, we have
  \be
  \begin{split}
&\Om((-z_0+\sum_{j=1}^m\z_j,(-z_1+\z_0,\cdots,-z_m+\z_0)),(\z_0,\vec{\z}))
\\
\sim &
\Om((\z_i,(\underbrace{0,\cdots,-z_i,0,\cdots,0}_{-z_i \text{ is the ith  vector}})),
(0,(\underbrace{0,\cdots,\z_i,0,\cdots,0}_{\z_i \text{ is the ith vector}})))
  \\
  = &\Om_i(z_i,\z_i).
  \end{split}
  \ee
  From the above two estimates, we get
  \be
  \begin{split}
    &\big|V_{\phi}f_0(\z_0,z_0)\prod_{j=1}^mV_{\phi}f_j(z_j,\z_j)
  \Om((-z_0+\sum_{j=1}^m\z_j,(-z_1+\z_0,\cdots,-z_m+\z_0)),(\z_0,\vec{\z}))\big|
    \\
    \gtrsim &
    |V_{\phi}f_i(z_i,\z_i)\Om_i(z_i,\z_i)|
    \prod_{j\in \G_i}\chi_{B_{\d}}(z_j)\chi_{B_{\d}}(\z_j).
  \end{split}
  \ee

   Hence, we have the following estimate of $\|R_m(f_0,\vec{f})\|_{\scrF M^{p,q}_{\Om}(\rmdd)}$:
  \be
  \begin{split}
    &\|R_m(f_0,\vec{f})\|_{\scrF M^{p,q}_{\Om}(\rmdd)}
    \\
    \gtrsim &
    \big\|V_{\phi}f_i(z_i,\z_i)\Om_i(z_i,\z_i)\prod_{j\in \G_i}\chi_{B_{\d}}(z_j)\chi_{B_{\d}}(\z_j)\big\|_{L^{p,q}(\rmdd\times\rmdd)}
    \\
    = &
    \big\|V_{\phi}f_i(z_i,\z_i)\Om_i(z_i,\z_i)\big\|_{L^{p,q}(\rd\times\rd)}
    \prod_{j\in \G_i}\big\|\chi_{B_{\d}}(z_j)\chi_{B_{\d}}(\z_j)\big\|_{L^{p,q}(\rd\times\rd)}
    \\
    \gtrsim &
    \big\|V_{\phi}f_i(z_i,\z_i)\Om_i(z_i,\z_i)\big\|_{L^{p,q}(\rd\times\rd)}=\|f_i\|_{M^{p,q}_{\Om_i}(\rd)}.
  \end{split}
  \ee
  From this and \eqref{thm-F1-cd0}, we obtain
  \be
  \begin{split}
  \|f_i\|_{M^{p,q}_{\Om_i}(\rd)}
  \lesssim &
  \|R_m(f_0,\vec{f})\|_{\scrF M^{p,q}_{\Om}(\rmdd)}
  \\
  \lesssim &
  \|f_0\|_{W(L^{p_0},L^{q_0}_{\mu_0})}\prod_{j=1}^m\|f_j\|_{W(L^{p_j},L^{q_j}_{\mu_j})}
  \\
  \lesssim &
  \|f_i\|_{W(L^{p_i},L^{q_i}_{\mu_i})}
  \prod_{j\in \G_i}\|h\|_{W(L^{p_j},L^{q_j}_{\mu_j})}
  \lesssim
 \|f_i\|_{W(L^{p_i},L^{q_i}_{\mu_i})}.
  \end{split}
  \ee
  This yields the embedding relation $W(L^{p_i},L^{q_i}_{\mu_i})\subset M^{p,q}_{\Om_i}(\rd)$.

  \textbf{``If'' part.}
  In this case, $\Om$ satisfies condition $W_0$, that is,
    \be
  \begin{split}
  \Om((-z_0+\sum_{j=1}^m\z_j,(-z_1+\z_0,\cdots,-z_m+\z_0)),(\z_0,\vec{\z}))
  \lesssim
  \Om_0(-z_0,\z_0)\prod_{j=1}^m\Om_i(z_j,\z_j).
  \end{split}
  \ee
  From this, we conclude that
  \be
  \begin{split}
    &\big|V_{\phi}f_0(\z_0,z_0)\prod_{j=1}^mV_{\phi}f_j(z_j,\z_j)
  \Om((-z_0+\sum_{j=1}^m\z_j,(-z_1+\z_0,\cdots,-z_m+\z_0)),(\z_0,\vec{\z}))\big|
    \\
    \lesssim &
    \big|V_{\phi}f_0(\z_0,z_0)\Om_0(-z_0,\z_0)\big|
    \prod_{j=1}^m\big|V_{\phi}f_j(z_j,\z_j)\Om_j(z_j,\z_j)\big|.
  \end{split}
  \ee
  Taking $L^{p,q}$-norm, we get
  \ben\label{thm-F1-3}
  \begin{split}
  &\|V_{\phi}f_0(\z_0,z_0)\prod_{j=1}^mV_{\phi}f_j(z_j,\z_j)
  \Om((-z_0+\sum_{j=1}^m\z_j,(-z_1+\z_0,\cdots,-z_m+\z_0)),(\z_0,\vec{\z}))\|_{L^{p,q}(\rmdd\times\rmdd)}
  \\
  \lesssim &
  \big\|V_{\phi}f_0(\z_0,z_0)\Om_0(-z_0,\z_0)
    \prod_{j=1}^m\big|V_{\phi}f_j(z_j,\z_j)\Om_j(z_j,\z_j)\big\|_{L^{p,q}(\rmdd\times\rmdd)}.
  \end{split}
  \een
  The last term above is equivalent to
  \be
  \begin{split}
    &
    \big\|V_{\phi}f_0(\z_0,z_0)\Om_0(-z_0,\z_0)\|_{L^{p,q}(\rd\times\rd)}
    \prod_{j=1}^m\big\|\big(V_{\phi}f_j(z_j,\z_j)\Om_j(z_j,\z_j)\big)\big\|_{L^{p,q}(\rmd\times\rmd)}
    \\
    = &
    \|f_0\|_{\scrF M^{p,q}_{\Om_0}}\prod_{j=1}^m\|f_j\|_{M^{p,q}_{\Om_j}}
    \lesssim
    \|f_0\|_{W(L^{p_0},L^{q_0}_{\mu_0})}\prod_{j=1}^m\|f_j\|_{W(L^{p_j},L^{q_j}_{\mu_j})},
  \end{split}
  \ee
  where we use the embedding relations \eqref{thm-F1-cd1} \eqref{thm-F1-cd2} in the last inequality.
  From this and \eqref{thm-F1-3}, we obtain the desired conclusion
  \be
  \|R_m(f_0,\vec{f})\|_{\scrF M^{p,q}_{\Om}(\rmdd)}\lesssim \|f_0\|_{W(L^{p_0},L^{q_0}_{\mu_0})}\prod_{j=1}^m\|f_j\|_{W(L^{p_j},L^{q_j}_{\mu_j})}.
  \ee
\end{proof}

\section{Self-improvement of the boundedness}
By an observation of the different structure between modulation and Wiener amalgam spaces,
and using some ideas from probability and classical harmonic analysis, we discover that 
both BRWM and BRWF boundedness have surprising self-improving properties.
The main theorems are stated as follows.

\begin{theorem}[Self-improvement of BRWM]\label{thm-msi}
	Assume $p_i, q_i\in (0,\fy)$, $p, q \in (0,\fy]$,
	and that $\Om\in \mathscr{P}(\rr^{2(m+1)d})$, $\mu_i \in \mathscr{P}(\rd)$, $i=0,1,\cdots,m$.
	Then, the following boundedness
	\ben\label{thm-msi-cd1}
	R_m: W(L^{p_0},L^{q_0}_{\mu_0})(\rd)\times\cdots \times W(L^{p_m},L^{q_m}_{\mu_m})(\rd)\longrightarrow M^{p,q}_{\Om}(\rmdd)
	\een
	can be self-improved to
	\ben\label{thm-msi-cd2}
	R_m: W(L^{p_0\wedge 2},L^{q_0}_{\mu_0})(\rd)\times\cdots \times W(L^{p_m\wedge 2},L^{q_m}_{\mu_m})(\rd)\longrightarrow M^{p,q}_{\Om}(\rmdd).
	\een
	Moreover, if $\Om$ satisfies M0, M1 and M2, the boundedness can be further improved to
	\ben\label{thm-msi-cd3}
	R_m: W(L^{p_0\wedge 2},L^{q_0}_{\mu_0})(\rd)\times\cdots \times W(L^{p_m\wedge 2},L^{q_m}_{\mu_m})(\rd)\longrightarrow M^{p\wedge q,q}_{\Om}(\rmdd).
	\een
\end{theorem}

\begin{theorem}[Self-improvement of BRWF]\label{thm-fsi}
	Assume $p_i,q_i\in (0,\fy)$, $p, q \in (0,\fy]$,
	and that $\Om\in \mathscr{P}(\rr^{2(m+1)d})$, $\mu_i \in \mathscr{P}(\rd)$, $i=0,1,\cdots,m$.
	Then the following boundedness 
	\ben\label{thm-fsi-cd1}
	R_m: W(L^{p_0},L^{q_0}_{\mu_0})(\rd)\times\cdots \times W(L^{p_m},L^{q_m}_{\mu_m})(\rd)\longrightarrow \scrF M^{p,q}_{\Om}(\rmdd)
	\een
	can be self-improved to
	\ben\label{thm-fsi-cd2}
	R_m: W(L^{p_0\wedge 2},L^{q_0}_{\mu_0})(\rd)\times \cdots \times W(L^{p_m\wedge 2},L^{q_m}_{\mu_m})(\rd)\longrightarrow \scrF M^{p,q}_{\Om}(\rmdd).
	\een
\end{theorem}
\subsection{Estimates for weighted convolution}
Suppose that $\vec{a}=\{a(k_0,n_0)\}_{k_0,n_0\in \zd}$,  $\vec{b_j}=\{b_j(k_0,n_0)\}_{k_0,n_0\in \zd}\ (j=1,2,\cdots,m)$ are some sequences defined on $\zd\times \zd$.
For a fixed $p\in (0,\fy]$ and a weight function $\Om\in \scrP(\rr^{2(m+1)d})$, the m-linear mixed weighted convolution operator associated with $\Om$
is defined as
\be
\begin{split}
	&T_{p,\Om}(\vec{a},\vec{b_1},\cdots,\vec{b_m})(n_0,\vec{n})
	\\
	= &
	\bigg(\sum_{k_0\in \zd, \vec{k}\in \zmd}\Big|a(k_0,n_0+\sum_{j=1}^mk_j)\prod_{j=1}^mb_j(n_j+k_0,k_j)\Big|^p \Om((k_0,\vec{k}),(n_0,\vec{n}))^p\bigg)^{1/p},
\end{split}
\ee
with the usual modification for $p=\fy$, where $n_0\in \zd, \vec{n}=(n_1,n_2,\cdots,n_m)\in \zmd$.

For $\vec{c}=\{c(k_0,n_0)\}_{k_0,n_0\in \zd}$ defined on $\zd\times \zd$ and $\vec{\r}=\{\r(l)\}_{l\in \zd}$ defined on $\zd$, we use the following notation
for the convolution associated with the second variable:
\be
(\r\ast_2 \vec{c})(k_0,n_0):=\sum_{l\in \zd}\r(l)c(k_0,n_0-l).
\ee

In the proof of the self-improvement of BRWM,
we will use the Fourier series to overcome the absence of Gabor frame in Lebesgue space.
The following lemma, providing some boundedness estimates associated with $T_{p,\Om}$,
will be used to retain the information of the Fourier coefficients and filter out redundant information
when estimating the norm of modulation space for Rihaczek distribution.

\begin{lemma}\label{lm-cvm}
  Suppose $p,q\in (0,\fy]$. Let $\Om\in \mathscr{P}(\rr^{2(m+1)d})$ be $v_{\Om}$ moderate.
  Denote $v_i(z_i)=v_{\Om}(\underbrace{0,\cdots,z_i,0,\cdots,0}_{z_i\ \text{is the}\ ith\  vector})$,
  $z_i\in \zd$, $i=1,2,\cdots,2m+2$.
  Let $v(z)\geq\max_{i=1,\cdots,2(m+1)}v_i(z)$, $z\in \zd$ be a radial function with polynomial growth.
  We have the following estimates:
  \ben\label{lm-cvm-cd1}
  \big\|\overrightarrow{T_{p,\Om}}(\vec{\r}\ast_2\vec{a},\vec{b_1},\cdots,\vec{b_m})\big\|_{l^q(\zmdd)}
  \lesssim
  \big\|\overrightarrow{|\r|}\big\|_{l^{\dot{p}\cdot[(q/\dot{p})\wedge 1]}_{v}(\zd)}
  \cdot \big\|\overrightarrow{T_{p,\Om}}(\vec{a},\vec{b_1},\cdots,\vec{b_m})\big\|_{l^q(\zmdd)}
  \een
  and
  \ben\label{lm-cvm-cd2}
  \big\|\overrightarrow{T_{p,\Om}}(\vec{a},\vec{b_1},\cdots,\vec{\r}\ast_2\vec{b_i},\cdots,\vec{b_m})\big\|_{l^q(\zmdd)}
  \lesssim
  \big\|\overrightarrow{|\r|}\big\|_{l^{\dot{p}[(q/\dot{p})\wedge 1]}_{v^2}(\zd)}
  \big\|\overrightarrow{|T_{p,\Om}|}\big\|_{l^{q}(\zmdd)}.
  \een
\end{lemma}
\begin{proof}
For simplicity, we write $T_p$ for the weighted convolution operator $T_{p,\Om}$ in this proof.
First, let us verify \eqref{lm-cvm-cd1}.
  Write
  \ben\label{lm-cvm-1}
  \begin{split}
  &T_p(\r\ast_2\vec{a},\vec{b_1},\cdots,\vec{b_m})(n_0,\vec{n})
  \\
  = &
  \bigg(\sum_{k_0,\vec{k}}|\sum_{l\in \zd}\r(l)a(k_0,n_0+\sum_{j=1}^mk_j-l)\prod_{j=1}^mb_j(n_j+k_0,k_j)|^p \Om((k_0,\vec{k}),(n_0,\vec{n}))^p\bigg)^{1/p}.
  \end{split}
  \een
  If $p\leq 1$, \eqref{lm-cvm-1} can be dominated by
  \be
  \begin{split}
  &\bigg(\sum_{k_0,\vec{k}}\sum_{l\in \zd}|\r(l)|^p|a(k_0,n_0+\sum_{j=1}^mk_j-l)|^p\prod_{j=1}^m|b_j(n_j+k_0,k_j)|^p\Om((k_0,\vec{k}),(n_0,\vec{n}))^p\bigg)^{1/p}
  \\
  = &
  \bigg(\sum_{l\in \zd}|\r(l)|^p\sum_{k_0,\vec{k}}|a(k_0,n_0+\sum_{j=1}^mk_j-l)|^p\prod_{j=1}^m|b_j(n_j+k_0,k_j)|^p\Om((k_0,\vec{k}),(n_0,\vec{n}))^p\bigg)^{1/p}
  \\
  \lesssim &
  \bigg(\sum_{l\in \zd}|\r(l)|^pv(l)^p\sum_{k_0,\vec{k}}|a(k_0,n_0+\sum_{j=1}^mk_j-l)|^p\prod_{j=1}^m|b_j(n_j+k_0,k_j)|^p\Om((k_0,\vec{k}),(n_0-l,\vec{n}))^p\bigg)^{1/p}
  \\
  = &
  \big(\sum_{l\in \zd}|\r(l)v(l)|^p|T_p(\vec{a},\vec{b_1},\cdots,\vec{b_m})(n_0-l,\vec{n})|^p\big)^{1/p}
  =
  (\overrightarrow{|\r v|^p}\ast \overrightarrow{|T_p(\cdot,\vec{n})|^p})(n_0)^{1/p}.
  \end{split}
  \ee
  If $p> 1$, by the Minkowski inequality, \eqref{lm-cvm-1} can be dominated by
  \be
  \begin{split}
  &\sum_{l\in \zd}|\r(l)|\bigg(\sum_{k_0,\vec{k}}|a(k_0,n_0+\sum_{j=1}^mk_j-l)\prod_{j=1}^mb_j(n_j+k_0,k_j)|^p\Om((k_0,\vec{k}),(n_0,\vec{n}))^p\bigg)^{1/p}
  \\
  \lesssim &
  \sum_{l\in \zd}|\r(l)v(l)|\bigg(\sum_{k_0,\vec{k}}|a(k_0,n_0+\sum_{j=1}^mk_j-l)\prod_{j=1}^mb_j(n_j+k_0,k_j)|^p\Om((k_0,\vec{k}),(n_0-l,\vec{n}))^p\bigg)^{1/p}
  \\
  = &
   \sum_{l\in \zd}|\r(l)v(l)||T_p(\vec{a},\vec{b_1},\cdots,\vec{b_m})(n_0-l,\vec{n})|
   =
   (\overrightarrow{|\r v|}\ast\overrightarrow{ |T_p(\cdot,\vec{n})|})(n_0).
   \end{split}
  \ee
  The above two estimates then imply that for $p\in (0,\fy]$,
  \be
  |T_p(\r\ast_2\vec{a},\vec{b_1},\cdots,\vec{b_m})(n_0,\vec{n})|
  \lesssim
  (\overrightarrow{|\r v|^{\dot{p}}}\ast \overrightarrow{|T_p(\cdot,\vec{n})|^{\dot{p}}})(n_0)^{1/\dot{p}}.
  \ee
  From this and the convolution inequality $l^{q/\dot{p}}\ast l^{(q/\dot{p})\wedge 1}\subset l^{q/\dot{p}}$, by taking the $l^q$ norm associated with the variable $n_0$, we obtain
  \be
  \begin{split}
  \|\vec{T_p}(\r\ast_2\vec{a},\vec{b_1},\cdots,\vec{b_m})(\cdot,\vec{n})\|_{l^q}
  \lesssim &
  \big\|\big((\overrightarrow{|\r v|^{\dot{p}}}\ast \overrightarrow{|T_p(\cdot,\vec{n})|^{\dot{p}}})^{1/\dot{p}}(n_0)\big)_{n_0}\big\|_{l^q}
  \\
  = &
  \big\|\overrightarrow{|\r v|^{\dot{p}}}\ast \overrightarrow{|T_p(\cdot,\vec{n})|^{\dot{p}}}\big\|^{1/\dot{p}}_{l^{q/\dot{p}}}
  \lesssim
  \big\|\overrightarrow{|\r v|^{\dot{p}}}\big\|^{1/\dot{p}}_{l^{(q/\dot{p})\wedge 1}}
  \big\|\overrightarrow{|T_p(\cdot,\vec{n})|^{\dot{p}}}\big\|^{1/\dot{p}}_{l^{q/\dot{p}}}
  \\
  = &
  \big\|\overrightarrow{|\r|}\big\|_{l^{\dot{p}\cdot[(q/\dot{p})\wedge 1]}_{v}}
  \|\overrightarrow{T_p(\cdot,\vec{n})}\|_{l^{q}}.
  \end{split}
  \ee
  Finally, by taking the $l^q$ norm associated with the variables $\vec{n}$, we get the desired conclusion
  \be
  \big\|\overrightarrow{T_{p}}(\vec{\r}\ast_2\vec{a},\vec{b_1},\cdots,\vec{b_m})\big\|_{l^q(\zmdd)}
  \lesssim
  \big\|\overrightarrow{|\r|}\big\|_{l^{\dot{p}\cdot[(q/\dot{p})\wedge 1]}_{v}(\zd)}
  \big\|\overrightarrow{T_{p}}(\vec{a},\vec{b_1},\cdots,\vec{b_m})\big\|_{l^q(\zmdd)}.
  \ee

  Next, we turn to the proof of \eqref{lm-cvm-cd2}. Without loss of generality, we only consider the case $i=1$.
  Write
  \ben\label{lm-cvm-2}
  \begin{split}
  &T_p(\vec{a},\r\ast_2\vec{b_1},\vec{b_2}\cdots,\vec{b_m})(n_0,\vec{n})
  \\
  = &
  \bigg(\sum_{k,\vec{k}}|a(k_0,n_0+\sum_{j=1}^mk_j)\sum_{l\in \zd}\r(l)b_1(n_1+k_0,k_1-l)\prod_{j=2}^mb_j(n_j+k_0,k_j)|^p\Om((k_0,\vec{k}),(n_0,\vec{n}))^p\bigg)^{1/p}.
  \end{split}
  \een
  If $p\leq 1$, \eqref{lm-cvm-2} can be dominated by
  \be
  \begin{split}
  &\bigg(\sum_{k,\vec{k}}|a(k_0,n_0+\sum_{j=1}^mk_j)|^p\cdot \sum_{l\in \zd}|\r(l)b_1(n_1+k_0,k_1-l)\prod_{j=2}^mb_j(n_j+k_0,k_j)|^p\Om((k_0,\vec{k}),(n_0,\vec{n}))^p\bigg)^{1/p}
  \\
  = &
  \bigg(\sum_{l\in \zd}|\r(l)|^p\sum_{k,\vec{k}}|a(k_0,n_0+\sum_{j=1}^mk_j)|^p |b_1(n_1+k_0,k_1-l)\prod_{j=2}^mb_j(n_j+k_0,k_j)|^p\Om((k_0,\vec{k}),(n_0,\vec{n}))^p\bigg)^{1/p}
  \\
  = &
  \bigg(\sum_{l\in \zd}|\r(l)|^p\sum_{k,\vec{k}}|a(k_0,n_0+l+\sum_{j=1}^mk_j)|^p |\prod_{j=1}^mb_j(n_j+k_0,k_j)|^p\Om((k_0,(k_1+l,k_2,\cdots,k_m)),(n_0,\vec{n}))^p\bigg)^{1/p}.
  \end{split}
  \ee
  Recall that
  \ben\label{lm-cvm-3}
  \Om((k_0,(k_1+l,k_2,\cdots,k_m)),(n_0,\vec{n}))\lesssim
  v(l)^2\Om((k_0,\vec{k}),(n_0+l,\vec{n})),
  \een
  the last term of the above equality can be dominated by
  \be
  \begin{split}
    &\bigg(\sum_{l\in \zd}|\r(l)v(l)^2|^p\sum_{k,\vec{k}}|a(k_0,n_0+l+\sum_{j=1}^mk_j)|^p |\prod_{j=1}^mb_j(n_j+k_0,k_j)|^p\Om((k_0,\vec{k}),(n_0+l,\vec{n}))^p\bigg)^{1/p}
    \\
    = &
    \big(\sum_{l\in \zd}|\r(l)v(l)^2|^p|T_p(\vec{a},\vec{b_1},\cdots,\vec{b_m})(n_0+l,\vec{n})|^p\big)^{1/p}
    \\
    = &
    \big(\sum_{l\in \zd}|\r(-l)v(-l)^2|^p|T_p(\vec{a},\vec{b_1},\cdots,\vec{b_m})(n_0-l,\vec{n})|^p\big)^{1/p}
    =
  (\overrightarrow{|\calI(\r v^2)|^p}\ast \overrightarrow{|T_p(\cdot,\vec{n})|^p})(n_0)^{1/p}.
  \end{split}
  \ee
  If $p> 1$, by the Minkowski inequality, \eqref{lm-cvm-2} can be dominated by
  \be
  \begin{split}
  &\sum_{l\in \zd}|\r(l)|\bigg(\sum_{k,\vec{k}}|a(k_0,n_0+\sum_{j=1}^mk_j)b_1(n_1+k_0,k_1-l)\prod_{j=2}^mb_j(n_j+k_0,k_j)|^p\Om((k_0,\vec{k}),(n_0,\vec{n}))^p\bigg)^{1/p}
  \\
  = &
  \sum_{l\in \zd}|\r(l)|\bigg(\sum_{k,\vec{k}}|a(k_0,n_0+l+\sum_{j=1}^mk_j)\prod_{j=1}^mb_j(n_j+k_0,k_j)|^p\Om((k_0,(k_1+l,k_2,\cdots,k_m)),(n_0,\vec{n}))^p\bigg)^{1/p}.
   \end{split}
  \ee
  Again, by \eqref{lm-cvm-3}, the last term of the above equality can be dominated by
  \be
  \begin{split}
    &\sum_{l\in \zd}|\r(l)v(l)^2|\bigg(\sum_{k,\vec{k}}|a(k_0,n_0+l+\sum_{j=1}^mk_j)\prod_{j=1}^mb_j(k_j)|^p\Om((k_0,\vec{k}),(n_0+l,\vec{n}))^p\bigg)^{1/p}
    \\
    = &
    \sum_{l\in \zd}|\r(l)v(l)^2||T_p(\vec{a},\vec{b_1},\cdots,\vec{b_m})(n_0+l,\vec{n})|
    \\
    = &
    \sum_{l\in \zd}|\r(-l)v(-l)^2||T_p(\vec{a},\vec{b_1},\cdots,\vec{b_m})(n_0-l,\vec{n})|
    =
  \overrightarrow{|\calI(\r v^2)|}\ast \overrightarrow{|T_p(\cdot,\vec{n})|})(n_0).
  \end{split}
  \ee
  The above two estimates then imply that
  \be
  |T_p(\vec{a},\r\ast_2\vec{b_1},\vec{b_2},\cdots,\vec{b_m})(n_0,\vec{n})|
  \lesssim
  (\overrightarrow{|\calI(\r v^2)|^{\dot{p}}}\ast \overrightarrow{|T_p(\cdot,\vec{n})|^{\dot{p}}})(n_0)^{1/\dot{p}}.
  \ee
  Then, by taking the $l^q$ norm associated with the variable $n_0$, we obtain
  \be
  \begin{split}
  \|\vec{T_p}(\vec{a},\r\ast_2\vec{b_1},\vec{b_2},\cdots,\vec{b_m})(\cdot,\vec{n})\|_{l^q}
  \lesssim &
  \big\|\big((\overrightarrow{|\calI(\r v^2)|^{\dot{p}}}\ast \overrightarrow{|T_p(\cdot,\vec{n})|^{\dot{p}}})(n_0)^{1/\dot{p}}\big)_{n_0}\big\|_{l^q}
  \\
  = &
  \big\|\overrightarrow{|\calI(\r v^2)|^{\dot{p}}}\ast \overrightarrow{|T_p(\cdot,\vec{n})|^{\dot{p}}}\big\|^{1/\dot{p}}_{l^{q/\dot{p}}}
  \\
  \lesssim &
  \big\|\overrightarrow{|\calI(\r v^2)|^{\dot{p}}}\big\|^{1/\dot{p}}_{l^{(q/\dot{p})\wedge 1}}
  \big\|\overrightarrow{|T_p(\cdot,\vec{n})|^{\dot{p}}}\big\|^{1/\dot{p}}_{l^{q/\dot{p}}}
  \\
  = &
  \big\|\overrightarrow{|\r|}\big\|_{l^{\dot{p}[(q/\dot{p})\wedge 1]}_{v^2}}
  \big\|\overrightarrow{|T_p(\cdot,\vec{n})|}\big\|_{l^{q}}.
  \end{split}
  \ee
  The desired conclusion follows by taking $l^q$ norm associated with $\vec{n}$.
\end{proof}

\subsection{Self-improvement of BRWM}
In this subsection, we present the proof of Theorem \ref{thm-msi}.
To prove this theorem,  we give the following key proposition in which the self-improving process can be realized in several steps.

\begin{proposition}\label{pp-si}
Let $p, q \in (0,\fy]$,  $\Om\in \mathscr{P}(\rr^{2(m+1)d})$,
and $\G=\{j: p_j>2, 0\leq j\leq m\}$.
Suppose that $p_i, q_i\in (0,\fy)$ for $i\in \G$.
The following statements are equivalent.
\begin{enumerate}
  \item The following boundedness is valid
  \be
  R_m: W(L^{p_0},L^{q_0}_{\mu_0})(\rd)\times\cdots \times W(L^{p_m},L^{q_m}_{\mu_m})(\rd)\longrightarrow M^{p,q}_{\Om}(\rmdd).
  \ee
  \item Let $\vec{b_j}=\{b_{j}(k_j,n_j)\}_{k_j,n_j\in \zd}$ for $j\in\G$, 
  and $\phi$ be a smooth function supported in $Q_0$, satisfying $\phi=1$ on $\frac{Q_0}{2}$.
  For any Schwartz function sequences $f_j$ for $j\notin \G$,
  denote
  $\{b_j(k_j,n_j)\}_{k_j,n_j\in \zd}:=\{V_{\phi}f_j(k_j,n_j)\}_{k_j,n_j\in \zd}$ for $j\notin \G$. Then,
  \be
\|\overrightarrow{T_{p,\Om}}(\vec{b_0},\vec{b_1},\cdots,\vec{b_m})\|_{l^q}
\lesssim
 \prod_{j\in \G}\big\|(\|b_j(k_j,\cdot)\|_{l^2})_{k_j}\big\|_{l^{q_j}_{\mu_j}}
 \prod_{j\notin \G}\|f_j\|_{W(L^{p_j},L^{q_j}_{\mu_j})}.
\ee
  \item Suppose $\d\in (0,1/4)$ and that
  $\phi$ is a smooth function which is supported in $Q_0$ and satisfies $\phi=1$ on $\frac{Q_0}{2}$.
  Let $f_j$ be a sequence of Schwartz functions for $j\notin \G$,
  and
   $f_j=\sum_{k_j\in \zd}f_{j,k_j}$ with $\text{supp}f_{j,k_j}\subset B(k_j,\d)$, for $j\in \G$.
  Denote
   $\{b_j(k_j,n_j)\}_{k_j,n_j\in \zd}:=\{V_{\phi}f_j(k_j,n_j)\}_{k_j,n_j\in \zd}$ for $j=0,1,\cdots,m$.
   We have
  \be
\|\overrightarrow{T_{p,\Om}}(\vec{b_0},\vec{b_1},\cdots,\vec{b_m})\|_{l^q}
\lesssim
\prod_{j\in \G}\big\|(\|f_{j,k_j}\|_{L^{2}})_{k_j}\big\|_{l^{q_j}_{\mu_j}}
\prod_{j\notin \G}\|f_j\|_{W(L^{p_j},L^{q_j}_{\mu_j})}.
\ee
  \item The following boundedness is valid
  \be
  R_m: W(L^{p_0\wedge 2},L^{q_0}_{\mu_0})(\rd)\times\cdots \times W(L^{p_m\wedge 2},L^{q_m}_{\mu_m})(\rd)\longrightarrow M^{p,q}_{\Om}(\rmdd).
  \ee
\end{enumerate}
\end{proposition}

\begin{proof}
Without loss of generality, we assume $\G=\{0\}$, since the other cases can be proved by repeating the proof process 
from (1) to (4)
similar to the case of $\G=\{0\}$.

\textbf{The proof of $(1)\Longrightarrow (2)$.}
To obtain our desired conclusion, we only need to first verify the sparse version
for sufficiently large $N\in \zz^+$
as follows
 \ben\label{pp-si-1}
 \begin{split}
\|\overrightarrow{T_{p,\Om_N}}(\vec{b_{0,N}},\vec{b_{1,N}},\cdots,\vec{b_{m,N}})\|_{l^q}
\lesssim
\big\|(\|b_{0}(k_0,\cdot)\|_{l^2})_{k_0}\big\|_{l^{q_0}_{\mu_0}}
\prod_{j=1}^m\|f_j\|_{W(L^{p_j},L^{q_j}_{\mu_j})},
\end{split}
\een
where
\be
\begin{split}
  &T_{p,\Om_N}(\vec{b_{0,N}},\vec{b_{1,N}},\cdots,\vec{b_{m,N}})(n_0,\vec{n})
  \\
  = &
  \bigg(\sum_{k_0, \vec{k}}|b_{0,N}(k_0,n_0+\sum_{j=1}^mk_j)\prod_{j=1}^mb_{j,N}(n_j+k_0,k_j)|^p \Om_N((k_0,\vec{k}),(n_0,\vec{n}))^p\bigg)^{1/p}
  \\
  = &
  \bigg(\sum_{k_0, \vec{k}}|b_0(Nk_0,Nn_0+\sum_{j=1}^mNk_j)\prod_{j=1}^mb_j(Nn_j+Nk_0,Nk_j)|^p \Om((Nk_0,N\vec{k}),(Nn_0,N\vec{n}))^p\bigg)^{1/p},
  \end{split}
  \ee
  and
  \be
  b_{j,N}(k_j,n_j)=b_j(Nk_j,Nn_j),\ \ \Om_N((k_0,\vec{k}),(n_0,\vec{n}))=\Om((Nk_0,N\vec{k}),(Nn_0,N\vec{n})).
  \ee

Choose $\va$ to be a smooth function satisfying $\text{supp}\va\subset B_{\d}$ with small $\d<1/4$ and $\widehat{\va}(0)=1$.
For a fixed truncated sequence (only finite nonzero items) $\vec{b_0}=\{b_0(k_0,n_0)\}_{k_0,n_0\in \zd}$,
we set
\be
g(x)=\sum_{k_0\in N\zd}\sum_{n_0\in N\zd}b_0(k_0,n_0)e^{2\pi i n_0\cdot x}T_{k_0}\va(x)=\sum_{k_0\in N\zd}g_{k_0},
\ee
and
\be
a_0(k_0,n_0)=\widehat{g_{k_0}}(n_0),\ \ \ \ k_0,n_0\in N\zd.
\ee
Recall that $\phi$ is a smooth function satisfying $\text{supp}\phi\subset Q_0$ and $\phi=1$ on $\frac{Q_0}{2}$.
By the fact that
$g_lT_{k_0}\phi=\d_{l,k_0}g_{k_0}$ for $l,k_0\in N\zd$,
we have
\ben\label{pp-si-0}
V_{\phi}g(k_0,n_0)
=
V_{\phi}g_{k_0}(k_0,n_0)=\widehat{g_{k_0}}(n_0)=a_0(k_0,n_0),\ \ \ k_0,n_0\in N\zd.
\een
For $j\notin \G$, we choose $f_j\in \calS(\rd)$, and denote
\be
b_{j}(k_j,n_j):=V_{\phi}f(k_j,n_j),\ \ \ \ k_j,n_j\in \zd.
\ee
We claim that
\ben\label{pp-si-c1}
\|\overrightarrow{T_{p,\Om_N}}(\vec{b_{0,N}},\vec{b_{1,N}},\cdots,\vec{b_{m,N}})\|_{l^q}
\lesssim
\|R_m(g,\vec{f})\|_{M^{p,q}_{\Om}}
\een
for sufficiently large $N$.

Let $\Phi=R_m(\phi,\cdots,\phi)$,
where $\phi$ is the smooth function mentioned above.
Using the definition of modulation space and the sampling property of STFT (see Lemma \ref{lm, bdCD}),
we deduce that
\ben\label{pp-si-7}
\begin{split}
  &\|R_m(g,\vec{f})\|_{M^{p,q}_{\Om}}
\sim
  \|V_{\Phi}R_m(g,\vec{f})\|_{L^{p,q}_{\Om}}
  \\
  \gtrsim &
  \|V_{\Phi}R_m(g,\vec{f})((z_0,\vec{z}),(\z_0,\vec{\z}))|_{N\zmdd\times N\zmdd}\|_{l^{p,q}_{\Om_N}}
  \\
  = &
  \|V_{\phi}g(Nk_0,Nn_0+\sum_{j=1}^mNk_j)
  \prod_{j=1}^mV_{\phi}f_j(Nk_0+Nn_j,Nk_j)\|_{l^{p,q}_{\Om_N}}
  \\
= &
\bigg\|\bigg(\bigg(\sum_{k,\vec{k}}\bigg|a_{0,N}(k_0,n_0+\sum_{j=1}^mk_j)\prod_{j=1}^m b_{j,N}(k_0+n_j,k_j)\bigg|^p\Om_N((k_0,\vec{k}),(n_0,\vec{n}))^p\bigg)^{1/p}\bigg)_{n_0,\vec{n}}\bigg\|_{l^q}
\\
= &
\|\overrightarrow{T_{p,\Om_N}}(\vec{a_{0,N}},\vec{b_{1,N}},\cdots,\vec{b_{m,N}})\|_{l^q}.
\end{split}
\een
Here, we denote $a_{0,N}(k_0,n_0)=a_0(Nk_0,Nn_0)$.
In order to prove the claim \eqref{pp-si-c1}, we only need to verify
\be
\|\overrightarrow{T_{p,\Om_N}}(\vec{b_{0,N}},\vec{b_{1,N}},\cdots,\vec{b_{m,N}})\|_{l^q}
\lesssim
\|\overrightarrow{T_{p,\Om_N}}(\vec{a_{0,N}},\vec{b_{1,N}},\cdots,\vec{b_{m,N}})\|_{l^q}
\ee
for sufficiently large $N$. By the definition of $\vec{a_0}$ and $\vec{b_0}$,
for $k_0, l\in \zd$
we have
\be
\begin{split}
a_{0,N}(k_0,l)
= &
a_{0}(Nk_0,Nl)
= 
\widehat{g_{Nk_0}}(Nl)
\\
= &
\scrF\bigg(\sum_{n_0\in \zd}b_{0,N}(k_0,n_0)e^{2\pi i Nn_0\cdot x}T_{Nk_0}\va(x)\bigg)(Nl)
\\
= &
\sum_{n_0\in \zd}b_{0,N}(k_0,n_0)\widehat{\va}(Nl-Nn_0)e^{-2\pi iNk_0(Nl-Nn_0)}
\\
= &
\sum_{n_0\in \zd}b_{0,N}(k_0,n_0)\rho_{N}(k_0,l-n_0)
=(\r_N\ast_2 b_{0,N})(k_0,l).
\end{split}
\ee
Here, we denote
\be
\rho_{N}(k_0,n_0)=\widehat{\va}(Nn_0)e^{-2\pi iNk_0(Nn_0)}.
\ee
Let $\vec{\r_N^0}=\vec{\r_N}-\vec{e_0}$, that is,
\be
\r_N^0(k_0,0)=0,\ \ \ \
\r_N^0(k_0,n_0)=\r_N(k_0,n_0)\  \text{for}\  n_0\neq 0.
\ee
Then
\be
\vec{b_{0,N}}=\vec{a_{0,N}}+\vec{b_{0,N}}-\vec{a_{0,N}}
=\vec{a_{0,N}}
-\vec{\r_N^0}\ast_2 \vec{b_{0,N}},
\ \ \ j=0,1,\cdots,m.
\ee
Write
\be
\begin{split}
\overrightarrow{T_{p,\Om_N}}(\vec{b_{0,N}},\vec{b_{1,N}},\cdots,\vec{b_{m,N}})
=
\overrightarrow{T_{p,\Om_N}}(\vec{a_{0,N}}
-\vec{\r_N^0}\ast_2 \vec{b_{0,N}},\vec{b_{1,N}},\cdots,\vec{b_{m,N}}).
\end{split}
\ee
We obtain
\ben\label{pp-si-5}
\begin{split}
  &\|\overrightarrow{T_{p,\Om_N}}(\vec{b_{0,N}},\vec{b_{1,N}},\cdots,\vec{b_{m,N}})\|_{l^q}
  \\
  \leq &
  C\|\overrightarrow{T_{p,\Om_N}}(\vec{a_{0,N}},\vec{b_{1,N}},\cdots,\vec{b_{m,N}})\|_{l^q}
  +\|\overrightarrow{T_{p,\Om_N}}(\vec{\r_N^0}\ast_2 \vec{b_{0,N}},\vec{b_{1,N}},\cdots,\vec{b_{m,N}})\|_{l^q}.
\end{split}
\een
Denote
\be
h_N(n_0)=|\widehat{\va}(Nn_0)|\ \text{for}\ n_0\neq 0,\ \ \ h_N(0)=0.
\ee
By a direct calculation, we conclude that
\be
\begin{split}
  \overrightarrow{T_{p,\Om_N}}(\vec{\r_N^0}\ast_2 \vec{b_{0,N}},\vec{b_{1,N}},\cdots,\vec{b_{m,N}})
  \leq &
  \overrightarrow{T_{p,\Om_N}}(\orw{|\r_N^0|}\ast_2 \orw{|b_{0,N}|},\orw{|b_{1,N}|},\cdots,\orw{|b_{m,N}|})
  \\
  = &
  \overrightarrow{T_{p,\Om_N}}(\orw{h_N}\ast_2 \orw{|b_{0,N}|},\orw{|b_{1,N}|},\cdots,\orw{|b_{m,N}|}),
\end{split}
\ee
where we use the fact that
\be
|\r_N^0(k_0,n_0)|=h_N(n_0)\ \ \  \text{for}\ n_0\in \zd.
\ee
Using this and Lemma \ref{lm-cvm}, we have the following estimate:
\ben\label{pp-si-4}
\begin{split}
  \|\overrightarrow{T_{p,\Om_N}}(\vec{\r_N^0}\ast_2 \vec{b_{0,N}},\vec{b_{1,N}},\cdots,\vec{b_{m,N}})\|_{l^q}
  \leq &
  \|\overrightarrow{T_{p,\Om_N}}(\orw{h_N}\ast_2 \orw{|b_{0,N}|},\orw{|b_{1,N}|},\cdots,\orw{|b_{m,N}|})\|_{l^q}
  \\
  \lesssim &
  \|h_N\|_{l^{\dot{p}\cdot[(q/\dot{p})\wedge 1]}_{v}}
  \|\overrightarrow{T_{p,\Om_N}}(\orw{|b_{0,N}|},\orw{|b_{1,N}|},\cdots,\orw{|b_{m,N}|})\|_{l^q}
  \\
  = &\|h_N\|_{l^{\dot{p}\cdot[(q/\dot{p})\wedge 1]}_{v}}
  \|\overrightarrow{T_{p,\Om_N}}(\vec{b_{0,N}},\vec{b_{1,N}},\cdots,\vec{b_{m,N}})\|_{l^q}.
\end{split}
\een
The combination of \eqref{pp-si-4} and \eqref{pp-si-5} yields that
\ben\label{pp-si-6}
\begin{split}
  &\|\overrightarrow{T_{p,\Om_N}}(\vec{b_{0,N}},\vec{b_{1,N}},\cdots,\vec{b_{m,N}})\|_{l^q}
  \\
  \leq &
  C\big(\|\overrightarrow{T_{p,\Om_N}}(\vec{a_{0,N}},\vec{b_{1,N}},\cdots,\vec{b_{m,N}})\|_{l^q}
  +\|h_N\|_{l^{\dot{p}\cdot[(q/\dot{p})\wedge 1]}_{v}}
  \|\overrightarrow{T_{p,\Om_N}}(\vec{b_{0,N}},\vec{b_{1,N}},\cdots,\vec{b_{m,N}})\|_{l^q}\big).
\end{split}
\een
Recall $h_N(n_0)=|\widehat{\va}(Nn_0)|$
for $n_0\neq 0$
and $\va$ is a $C_c^{\fy}(\rd)$ function. We have
\be
|h_N(n_0)|=|\widehat{\va}(Nn_0)|\lesssim \langle Nn_0\rangle^{-\scrL}\sim N^{-\scrL}\langle n_0\rangle^{-\scrL}\ \ \ \ (n_0\neq 0),
\ee
where $\scrL$ indicates a sufficiently large number.
Then,
\be
\begin{split}
  \|h_N\|_{l^{\dot{p}\cdot[(q/\dot{p})\wedge 1]}_{v}}
  \lesssim &
  N^{-\scrL}
  \big\|\big(\langle n_0\rangle^{-\scrL}\big)_{n_0}\big\|_{l^{\dot{p}\cdot[(q/\dot{p})\wedge 1]}_{v}}
  \lesssim N^{-\scrL},
\end{split}
\ee
which tends to zero as $N\rightarrow \fy$.
Using this and \eqref{pp-si-6}, for sufficiently large $N$ we have $C\|h_N\|_{l^{\dot{p}\cdot[(q/\dot{p})\wedge 1]}_{v}}\leq 1/2$, and
\be
\begin{split}
  \|\overrightarrow{T_{p,\Om_N}}(\vec{b_{0,N}},\vec{b_{1,N}},\cdots,\vec{b_{m,N}})\|_{l^q}
  \leq
  C\|\overrightarrow{T_{p,\Om_N}}(\vec{a_{0,N}},\vec{b_{1,N}},\cdots,\vec{b_{m,N}})\|_{l^q}
  +\frac{1}{2}\|\overrightarrow{T_{p,\Om_N}}(\vec{b_{0,N}},\vec{b_{1,N}},\cdots,\vec{b_{m,N}})\|_{l^q},
\end{split}
\ee
which implies that
\ben
\|\overrightarrow{T_{p,\Om_N}}(\vec{b_{0,N}},\vec{b_{1,N}},\cdots,\vec{b_{m,N}})\|_{l^q}
  \leq
  2C\|\overrightarrow{T_{p,\Om_N}}(\vec{a_{0,N}},\vec{b_{1,N}},\cdots,\vec{b_{m,N}})\|_{l^q}.
\een
Then, the claim \eqref{pp-si-c1} follows by this and \eqref{pp-si-7}.

Using \eqref{pp-si-c1}, if (1) is valid  we obtain
\ben\label{pp-si-8}
\begin{split}
&\|\overrightarrow{T_{p,\Om_N}}(\vec{b_{0,N}},\vec{b_{1,N}},\cdots,\vec{b_{m,N}})\|_{l^q}
\lesssim
\|R_m(g,\vec{f})\|_{M^{p,q}_{\Om}}
\\
\lesssim &
\big\|\big(\big\|\sum_{n_0\in \zd}b_{0}(Nk_0,Nn_0)e^{2\pi i Nn_0\cdot x}T_{Nk_0}\va(x)\big\|_{L^{p_0}}\big)_{k_0}\big\|_{l^{q_0}_{\mu_0}}
\prod_{j=1}^m\|f_j\|_{W(L^{p_j},L^{q_j}_{\mu_j})}.
\end{split}
\een

Next, we show that the above inequality can be improved by Khinchin's inequality.
To achieve this goal, we replace $\vec{b_0}$ by $\vec{b_0^{\om}}$ defined as
\be
b_0^{\om}(Nk_0,Nn_0)=b_0(Nk_0,Nn_0)\om_{n_0},\ \ \ n_0\in \zd,
\ee
where $\vec{\omega}=\{\omega_l\}_{l\in \mathbb{Z}^n}$ is a sequence of independent random variables taking values $\pm 1$ with equal probability (for instance, one can choose the Rademacher functions).
Using Khinchin's inequality, if $p_0, q_0<\fy$,
we deduce that
\be
\begin{split}
  &\bigg(\mathbb{E}\big(\big\|\big(\big\|\sum_{n_0\in \zd}b_{0}^{\om}(Nk_0,Nn_0)e^{2\pi i Nn_0\cdot x}T_{Nk_0}\va(x)\big\|_{L^{p_0}}\big)_{k_0}\big\|_{l^{q_0}_{\mu_0}}^{p_0\vee q_0}\big)\bigg)^{1/(p_0\vee q_0)}
  \\
  \lesssim &
  \bigg(\mathbb{E}\big(\big\|\big(\big\|\sum_{n_0\in \zd}b_{0}^{\om}(Nk_0,Nn_0)e^{2\pi i Nn_0\cdot x}T_{Nk_0}\va(x)\big\|_{L^{p_0\vee q_0}}\big)_{k_0}\big\|_{l^{q_0}_{\mu_0}}^{p_0\vee q_0}\big)\bigg)^{1/(p_0\vee q_0)}
  \\
  \leq &
  \big\|\big(\mathbb{E}\big(\big\|\sum_{n_0\in \zd}b_{0}^{\om}(Nk_0,Nn_0)e^{2\pi i Nn_0\cdot x}T_{Nk_0}\va(x)\big\|_{L^{p_0\vee q_0}}^{p_0\vee q_0}\big)^{1/(p_0\vee q_0)}\big)_{k_0}\big\|_{l^{q_0}_{\mu_0}}
  \\
  = &
  \big\|\big(\big\|\mathbb{E}\big(\big|\sum_{n_0\in \zd}b_{0}^{\om}(Nk_0,Nn_0)e^{2\pi i Nn_0\cdot x}T_{Nk_0}\va(x)\big|^{p_0\vee q_0}\big)^{(1/p_0\vee q_0)}\big\|_{L^{p_0\vee q_0}}\big)_{k_0}\big\|_{l^{q_0}_{\mu_0}}
  \\
  \sim &
  \big\|\big(\big\|\big(\sum_{n_0\in \zd}\big|b_{0}(Nk_0,Nn_0)\big|^2\big)^{1/2}T_{Nk_0}\va(x)\big\|_{L^{p_0\vee q_0}}\big)_{k_0}\big\|_{l^{q_0}_{\mu_0}}
  \\
  \sim &
  \big\|\big(\sum_{n_0\in \zd}\big|b_{0}(Nk_0,Nn_0)\big|^2\big)^{1/2}\big)_{k_0}\big\|_{l^{q_0}_{\mu_0}}
  =
  \big\|(\|b_{0,N}(k_0,\cdot)\|_{l^2})_{k_0}\big\|_{l^{q_0}_{\mu_0}}.
\end{split}
\ee
Applying the above estimates to the right term in \eqref{pp-si-8}, and observing that the left term
is invariant under taking expectation, we obtain the sparse version of conclusion (2):
\be
\begin{split}
\|\overrightarrow{T_{p,\Om_N}}(\vec{b_{0,N}},\vec{b_{1,N}},\cdots,\vec{b_{m,N}})\|_{l^q}
\lesssim &
\big\|(\|b_{0,N}(k_0,\cdot)\|_{l^2})_{k_0}\big\|_{l^{q_0}_{\mu_0}}
\prod_{j=1}^m\|f_j\|_{W(L^{p_j},L^{q_j}_{\mu_j})}
\\
\leq &
\big\|(\|b_{0}(k_0,\cdot)\|_{l^2})_{k_0}\big\|_{l^{q_0}_{\mu_0}}
\prod_{j=1}^m\|f_j\|_{W(L^{p_j},L^{q_j}_{\mu_j})}.
\end{split}
\ee
For $\vec{i}=(i_0,\cdots,i_m),\vec{l}=(l_0,\cdots,l_m)\in [0,N)^{(m+1)d}\cap\zmdd$, denote
\be
\begin{split}
	&T_{p,\Om_N,(\vec{i},\vec{l})}(\vec{b_{0,N}},\vec{b_{1,N}},\cdots,\vec{b_{m,N}})(n_0,\vec{n})
	\\
	= &
	\bigg(\sum_{k_0, \vec{k}}|b_0(Nk_0+i_0,Nn_0+l_0+\sum_{j=1}^m(Nk_j+i_j))
	\\
	& \prod_{j=1}^mb_j(Nn_j+l_j+Nk_0+i_0,Nk_j+i_j)|^p \Om((Nk_0,N\vec{k}),(Nn_0,N\vec{n}))^p\bigg)^{1/p}.
\end{split}
\ee
Observe that 
$\|\overrightarrow{T_{p,\Om}}(\vec{b_0},\vec{b_1},\cdots,\vec{b_m})\|_{l^q}$
can be dominated from above by the summation of the terms $\|T_{p,\Om_N,(\vec{i},\vec{l})}(\vec{b_{0,N}},\vec{b_{1,N}},\cdots,\vec{b_{m,N}})(n_0,\vec{n})\|_{l^q}$ with respect to
$\vec{i}=(i_0,\cdots,i_m),\vec{l}=(l_0,\cdots,l_m)\in [0,N)^{(m+1)d}\cap\zmdd$.
Using the sparse estimate \eqref{pp-si-1}
and the translation invariant of the norms $l^{q_0}_{\mu_0}(l^2)$ and $W(L^{p_j},L^{q_j}_{\mu_j})$,
we conclude that
\be
\|\overrightarrow{T_{p,\Om_N,(\vec{i},\vec{j})}}(\vec{b_{0,N}},\vec{b_{1,N}},\cdots,\vec{b_{m,N}})\|_{l^q}
\lesssim
\big\|(\|b_{0}(k_0,\cdot)\|_{l^2})_{k_0}\big\|_{l^{q_0}_{\mu_0}}
\prod_{j=1}^m\|f_j\|_{W(L^{p_j},L^{q_j}_{\mu_j})}.
\ee
Then, the full version follows by a summation of above terms
for all  $\vec{i}=(i_0,\cdots,i_m),\vec{l}=(l_0,\cdots,l_m)\in [0,N)^{(m+1)d}\cap\zmdd$.

\textbf{The proof of $(2)\Longrightarrow (3)$.}
Recall that $f_0=\sum_{k_0\in \zd}f_{0,k_0}$ with $\text{supp}f_{0,k_0}\subset B(k_0,\d)$. We have
\be
b_{0}(k_0,n_0)=V_{\phi}f_{0}(k_0,n_0)=\widehat{f_{0,k_0}}(n_0),\ \ \ \ k_0,n_0\in \zd.
\ee
Hence,
\ben\label{pp-si-9}
\begin{split}
\|\overrightarrow{T_{p,\Om}}(\vec{b_{0}},\vec{b_{1}},\cdots,\vec{b_{m}})\|_{l^q}
\lesssim &
\big\|(\|b_{0}(k_0,\cdot)\|_{l^2})_{k_0}\big\|_{l^{q_0}_{\mu_0}}
\prod_{j=1}^m\|f_j\|_{W(L^{p_j},L^{q_j}_{\mu_j})}
\\
\lesssim &
\big\|(\|f_{0,k_0}\|_{L^{2}})_{k_0}\big\|_{l^{q_0}_{\mu_0}}
\prod_{j=1}^m\|f_j\|_{W(L^{p_j},L^{q_j}_{\mu_j})},
\end{split}
\een
where $\{b_j(k_j,n_j)\}_{k_j,n_j\in \zd}:=\{V_{\phi}f_j(k_j,n_j)\}_{k_j,n_j\in \zd}$,
for any Schwartz function sequence $f_j$, $j=0,1,\cdots,m$.
This completes the proof of $(2)\Longrightarrow (3)$.

\textbf{The proof of $(3)\Longrightarrow (4)$.}
By a similar reduction as in the proof of Theorem \ref{thm-M0}, we only need to verify this conclusion for
$f_j=\sum_{k_j\in \zd}f_{j,k_j}$ with $\text{supp}f_{j,k_j}\subset B(k_j,\d)$.
Using conclusion (3) and the fact
\be
\begin{split}
	b_{j}(k_j,n_j)
	=
	V_{\phi}f_{j}(k_j,n_j)
	=
	V_{\phi}f_{j,k_j}(k_j,n_j)
	=
	\widehat{f_{j,k_j}}(n_j),\ \ \ \ \  j=0,1,\cdots,m,
\end{split}
\ee
we obtain
\ben\label{pp-si-10}
\|\overrightarrow{T_{p,\Om}}(\vec{b_{0}},\vec{b_{1}},\cdots,\vec{b_{m}})\|_{l^q}
\lesssim
\big\|(\|f_{0,k_0}\|_{L^{2}})_{k_0}\big\|_{l^{q_0}_{\mu_0}}
\prod_{j=1}^m\big\|(\|f_{j,k_j}\|_{L^{p_j}})_{k_j}\big\|_{l^{q_j}_{\mu_j}}
\een
with
\be
\begin{split}
\|\overrightarrow{T_{p,\Om}}(\vec{b_{0}},\vec{b_{1}},\cdots,\vec{b_{m}})\|_{l^q}
=
\|\widehat{f_{0,k_0}}(n_0+\sum_{j=1}^mk_j)
  \prod_{j=1}^m \widehat{f_{j,k_0+n_j}}(k_j)\|_{l^{p,q}_{\Om}}.
\end{split}
\ee

Take $\va$ to be a smooth function supported on $B(0,\d)$.
Using Corollary \ref{cy-eqn}, there exists a constant $N\in \zz^+$ such that
\be
\begin{split}
  &\|R_m(f_0,\vec{f})\|_{M^{p,q}_{\Om}}
  \\
  \sim &
  \bigg\|V_{\va}f_0\Big(\frac{k_0}{N},\frac{n_0+\sum_{j=1}^mk_j}{N}\Big)
  \prod_{j=1}^mV_{\va}f_j\Big(\frac{k_0+n_j}{N},\frac{k_j}{N}\Big)\Om\bigg(\Big(\frac{k_0}{N}\frac{\vec{k}}{N}\Big),\Big(\frac{n_0}{N},\frac{\vec{n}}{N}\Big)\bigg)\bigg\|_{l^{p,q}(\zmdd\times\zmdd)}.
\end{split}
\ee
Denote
\be
\La=[0,N)^{2(m+1)d}\cap \mathbb{Z}^{2(m+1)d},\ \  \G_{\vec{i},\vec{l}}=(\vec{i},\vec{l})+N\mathbb{Z}^{2(m+1)d},\ \  (\vec{i},\vec{l})\in \La,
\ee
where $\vec{i}=(i_0,\cdots,i_m)$, $\vec{l}=(l_0,\cdots,l_m)$ with $i_j, l_j\in \zd$, $j=0,1,\cdots,m$.
We obtain the finite partition of $\mathbb{Z}^{2(m+1)d}$:
\be
\mathbb{Z}^{2(m+1)d}=\bigcup_{(\vec{i},\vec{l})\in \La}\G_{\vec{i},\vec{l}}
\ee
and the following estimate:
\be
\begin{split}
&\bigg\|V_{\va}f_0\Big(\frac{k_0}{N},\frac{n_0+\sum_{j=1}^mk_j}{N}\Big)
  \prod_{j=1}^mV_{\va}f_j\Big(\frac{k_0+n_j}{N},\frac{k_j}{N}\Big)\Om\bigg(\Big(\frac{k_0}{N}\frac{\vec{k}}{N}\Big),\Big(\frac{n_0}{N},\frac{\vec{n}}{N}\Big)\bigg)\bigg\|_{l^{p,q}(\zmdd\times\zmdd)}
  \\
\sim&
\sum_{(\vec{i},\vec{l})\in \La}\bigg\|V_{\va}f_0\Big(k_0+\frac{i_0}{N},n_0+\sum_{j=1}^mk_j+\frac{l_0+\sum_{j=1}^mi_j}{N}\Big)
\\
&\ \ \ \ \ \ \ \ \cdot\prod_{j=1}^mV_{\va}f_j\Big(k_0+n_j+\frac{i_0+l_j}{N},k_j+\frac{i_j}{N}\Big)\Om((k_0,\vec{k}),(n_0,\vec{n}))\bigg\|_{l^{p,q}(\zmdd\times\zmdd)}
\\
= &
\sum_{(\vec{i},\vec{l})\in \La}
\bigg\|\scrF\Big(f_0\overline{M_{\frac{l_0+\sum_{j=1}^mi_j}{N}}T_{k_0+\frac{i_0}{N}}\va}\Big)\Big(n_0+\sum_{j=1}^mk_j\Big)
\prod_{j=1}^m\scrF\Big(f_j\overline{M_{\frac{i_j}{N}}T_{k_0+n_j+\frac{i_0+l_j}{N}}\va}\Big)(k_j)\bigg\|_{l^{p,q}_{\Om}(\zmdd\times\zmdd)}.
\end{split}
\ee
For every $(\vec{i},\vec{l})\in \La$, in \eqref{pp-si-10}, replacing $f_0$ by
\be
F_0=T_{-\frac{i_0}{N}}f_0\sum_{k_0\in \zd}\overline{M_{\frac{l_0+\sum_{j=1}^mi_j}{N}}T_{k_0}\va}
=
\sum_{k_0\in \zd}\big(T_{-\frac{i_0}{N}}f_0\big)\overline{M_{\frac{l_0+\sum_{j=1}^mi_j}{N}}T_{k_0}\va}
=\sum_{k_0\in \zd}F_{0,k_0,}
\ee
and replacing $f_j$ by
\be
F_j=T_{-\frac{i_0+l_j}{N}}f_j\sum_{k_j\in \zd}\overline{M_{\frac{i_j}{N}}T_{k_j}\va}
=
\sum_{k_j\in \zd}\big(T_{-\frac{i_0+l_j}{N}}f_j\big)\overline{M_{\frac{i_j}{N}}T_{k_j}\va}
=\sum_{k_j\in \zd}F_{j,k_j},
\ee
using the fact that $\text{supp}F_{j,k_j}\subset B(k_j,\d)$ for $j=1,2,\cdots,m$,
we conclude that
\be
\begin{split}
  &\bigg\|\scrF(f_0\overline{M_{\frac{l_0+\sum_{j=1}^mi_j}{N}}T_{k_0+\frac{i_0}{N}}\va})(n_0+\sum_{j=1}^mk_j)
\prod_{j=1}^m\scrF(f_j\overline{M_{\frac{i_j}{N}}T_{k_0+n_j+\frac{i_0+l_j}{N}}\va})(k_j)\bigg\|_{l^{p,q}_{\Om}(\zmdd\times\zmdd)}
\\
= &
\Big\|\widehat{F_{0,k_0}}(n_0+\sum_{j=1}^mk_j)
  \prod_{j=1}^m \widehat{F_{j,k_0+n_j}}(k_j)\Big\|_{l^{p,q}_{\Om}}
  \\
  \lesssim &
\big\|(\|F_{0,k_0}\|_{L^{2}})_{k_0}\big\|_{l^{q_0}_{\mu_0}}
\prod_{j=1}^m\big\|(\|F_{j,k_j}\|_{L^{p_j}})_{k_j}\big\|_{l^{q_j}_{\mu_j}}
\lesssim
\big\|(\|f_{0,k_0}\|_{L^{2}})_{k_0}\big\|_{l^{q_0}_{\mu_0}}
\prod_{j=1}^m\big\|(\|f_{j,k_j}\|_{L^{p_j}})_{k_j}\big\|_{l^{q_j}_{\mu_j}}.
\end{split}
\ee
Recall that $\La$ is a finite subset of $\mathbb{Z}^{2(m+1)d}$.
By a summation of the above terms with respect to all $(\vec{i},\vec{l})\in \La$, we conclude the desired conclusion
\be
\begin{split}
  \|R_m(f_0,\vec{f})\|_{M^{p,q}_{\Om}}
  \lesssim &
  \sum_{(\vec{i},\vec{l})\in \La}
  \big\|(\|f_{0,k_0}\|_{L^{2}})_{k_0}\big\|_{l^{q_0}_{\mu_0}}
\prod_{j=1}^m\big\|(\|f_{j,k_j}\|_{L^{p_j}})_{k_j}\big\|_{l^{q_j}_{\mu_j}}
\\
  \lesssim &
  \big\|(\|f_{0,k_0}\|_{L^{2}})_{k_0}\big\|_{l^{q_0}_{\mu_0}}
\prod_{j=1}^m\big\|(\|f_{j,k_j}\|_{L^{p_j}})_{k_j}\big\|_{l^{q_j}_{\mu_j}}.
\end{split}
\ee

\textbf{The proof of $(4)\Longrightarrow (1)$.}
It follows directly by the known embedding relation
$L^{p_0}(B_{\d})\subset L^{p_0\wedge 2}(B_{\d})$.
We have now completed the proof.
\end{proof}

Finally, with the help of Theorem \ref{thm-M1} and Proposition \ref{pp-si},
we give the proof of Theorem \ref{thm-msi}.

\begin{proof}[Proof of Theorem \ref{thm-msi}]
	The equivalent relation $\eqref{thm-msi-cd1}\Longleftrightarrow \eqref{thm-msi-cd2}$ follows directly by 
	Proposition \ref{pp-si}. In other words, the boundedness \eqref{thm-msi-cd1} can be self-improved to \eqref{thm-msi-cd2}.
	
	Next, we consider the further improvement when  $\Om$ satisfies M0, M1 and M2.
	In this case, by Theorem \ref{thm-M1} and the relation $\eqref{thm-msi-cd1}\Longleftrightarrow \eqref{thm-msi-cd2}$ proved above, we conclude that the boundedness \eqref{thm-msi-cd1} is equivalent to 
	 \ben\label{thm-msi-1}
	R_m: L^{p_0\wedge 2}(B_{\d})\times\cdots \times L^{p_m\wedge 2}(B_{\d})\longrightarrow M^{p,q}_{\Omba}(\rmdd),
	\een
	for some $\d>0$,
	and
	\ben\label{thm-msi-2}
	\tau_m\big(\otimes_{j=0}^m l^{q_j}_{\mu_j}(\zd)\big)\subset l^{p,q}_{\Omab}(\zd\times\zmd).
	\een
	Moreover, when $p\geq q$, 
	the condition \eqref{thm-msi-1} is equivalent to 
	
	\ben\label{thm-msi-3}
	L^{p_i}(B_{\d})\subset \scrF^{-1}L^q_{\Ombi}(\rd),\ \ \ i=0,1,\cdots,m,
	\een
    and the condition \eqref{thm-msi-2} is equivalent to 
	\ben\label{thm-msi-4}
	l^{q_i}_{\mu_i}(\zd)\subset  l^{q}_{\Omai}(\zd),\ i=0,1,\cdots,m.
	\een
	Observing that the exponent $p$ is missing in \eqref{thm-msi-3} and  \eqref{thm-msi-4}, 
	so the exponent $p$ can be replaced by $p\wedge q$ in both \eqref{thm-msi-1} and \eqref{thm-msi-2}.
	Using this fact, and applying Theorem \ref{thm-M1} again, with $p$ replacing by $p\wedge q$,
	the conditions \eqref{thm-msi-1} and \eqref{thm-msi-2} further imply the boundedenss 
	  \be
	R_m: W(L^{p_0\wedge 2},L^{q_0}_{\mu_0})(\rd)\times\cdots \times W(L^{p_m\wedge 2},L^{q_m}_{\mu_m})(\rd)\longrightarrow M^{p\wedge q,q}_{\Om}(\rmdd).
	\ee
	This is our desired conclusion.
	\end{proof}

\subsection{Self-improvement of BRWF}
In this subsection, we consider the self-improvement of BRWF and give the proof of Theorem \ref{thm-fsi}.
 Since the method here is similar to that in the proof of Theorem \ref{thm-msi},
we will omit most of the details in this case.

Let $\vec{a}=\{a(k_0,n_0)\}_{k_0,n_0\in \zd}$,  $\vec{b_j}=\{b_j(k_0,n_0)\}_{k_0,n_0\in \zd}$ be sequences defined on $\zd\times \zd$.
Let $\Om$ be a weight function belonging to $\scrP(\rr^{2(m+1)d})$.
For the sake of convenience, we denote

\be
\begin{split}
  &S_{p,\Om}(\vec{a},\vec{b_1},\cdots,\vec{b_m})(n_0,\vec{n})
  \\
  = &
  \bigg(\sum_{k_0\in \zd, \vec{k}\in \zmd}|a(n_0,-k_0+\sum_{j=1}^mn_j)\prod_{j=1}^mb_j(-k_j+n_0,n_j) \Om((k_0,\vec{k}),(n_0,\vec{n}))|^p\bigg)^{1/p},
  \end{split}
  \ee
  with the usual modification for $p=\fy$, where $n_0\in \zd, \vec{n}\in \zmd$.
We first establish the following convolution inequalities for $S_{p,\Om}$.

\begin{lemma}\label{lm-cvf}
  Suppose $p,q\in (0,\fy]$. Let $\Om\in \mathscr{P}(\rr^{2(m+1)d})$ be $v_{\Om}$ moderate.
  Denote $v_i(z_i)=v_{\Om}(\underbrace{0,\cdots,z_i,0,\cdots,0}_{z_i\ \text{is the}\ ith\  vector})$,
  $z_i\in \zd$, $i=1,2,\cdots,2m+2$.
  Let $v(z)\geq\max_{i=1,\cdots,2(m+1)}v_i(z)$, $z\in \zd$ be a radial function with polynomial growth.
  We have the following estimates:
   \ben\label{lm-cvf-cd1}
  \big\|\overrightarrow{S_{p,\Om}}(\vec{\r}\ast_2\vec{a},\vec{b_1},\cdots,\vec{b_m})\big\|_{l^q(\zmdd)}
  \lesssim
  \big\|\overrightarrow{|\r|}\big\|_{l^{\dot{p}}_{v}(\zd)}
  \cdot \big\|\overrightarrow{S_{p,\Om}}(\vec{a},\vec{b_1},\cdots,\vec{b_m})\big\|_{l^q(\zmdd)}
  \een
  and
  \ben\label{lm-cvf-cd2}
  \big\|\overrightarrow{S_{p,\Om}}(\vec{a},\vec{b_1},\cdots,\vec{\r}\ast_2\vec{b_i},\cdots,\vec{b_m})\big\|_{l^q(\zmdd)}
  \lesssim
  \big\|\overrightarrow{|\r|}\big\|_{l^{\dot{p}[(q/\dot{p})\wedge 1]}_{v^2}(\zd)}
  \big\|\overrightarrow{|S_{p,\Om}|}\big\|_{l^{q}(\zmdd)}.
  \een
\end{lemma}
\begin{proof}
Write
  \be
  \begin{split}
  &S_{p,\Om}(\vec{\r}\ast_2\vec{a},\vec{b_1},\cdots,\vec{b_m})(n_0,\vec{n})
  \\
  = &
  \bigg(\sum_{k_0,\vec{k}}|\sum_{l\in \zd}\r(l)a(n_0,-k_0+\sum_{j=1}^mn_j-l)\prod_{j=1}^m b_j(-k_j+n_0,n_j) \Om((k_0,\vec{k}),(n_0,\vec{n}))|^p\bigg)^{1/p}
  \\
  \leq &
    \bigg(\sum_{k_0,\vec{k}}\big(\sum_{l\in \zd}|\r(l)v(l)||a(n_0,-k_0+\sum_{j=1}^mn_j-l)|\prod_{j=1}^m |b_j(-k_j+n_0,n_j)| \Om((k_0+l,\vec{k}),(n_0,\vec{n}))\big)^p\bigg)^{1/p}.
  \end{split}
  \ee
  By using the Young inequality $l^p\ast l^{\dot{p}}\subset l^p$ related to the variable $k_0$, the above term can be
  dominated by
  \be
  \begin{split}
    \|\rho v\|_{l^{\dot{p}}}S_{p,\Om}(\vec{a},\vec{b_1},\cdots,\vec{b_m})(n_0,\vec{n}).
  \end{split}
  \ee
  The desired conclusion \eqref{lm-cvf-cd1} follows by taking $l^q$ norm of $(n_0,\vec{n}))$.

  Next, we turn to the proof of \eqref{lm-cvf-cd2}. Without loss of generality, we only give the proof for $i=1$.
  By a similar argument as in the proof of Lemma \ref{lm-cvm}, we get
   \be
\begin{split}
	&S_{p,\Om}(\vec{a},\vec{\r}\ast_2\vec{b_1},\vec{b_2},\cdots,\vec{b_m})(n_0,\vec{n})
	\\
	\leq &
	\bigg(\sum_{l\in \zd}|\r(l)v(l)^2|^{\dot{p}} |S_{p,\Om}(\vec{a},\vec{\r}\ast_2\vec{b_1},\vec{b_2},\cdots,\vec{b_m})(n_0,(n_1-l,n_2,\cdots,n_m))|^{\dot{p}}\bigg)^{1/\dot{p}}.
\end{split}
\ee
  Applying Young's inequality, we conclude that
  \be
  \begin{split}
  	\|S_{p,\Om}(\vec{a},\vec{\r}\ast_2\vec{b_1},\vec{b_2},\cdots,\vec{b_m})\|_{l^q}
  	\lesssim
  	\big\|\overrightarrow{|\r|}\big\|_{l^{\dot{p}[(q/\dot{p})\wedge 1]}_{v^2}}
  	\big\|\overrightarrow{S_{p,\Om}(\vec{a},\vec{b_1},\cdots,\vec{b_m})}\big\|_{l^{q}}.
  \end{split}	
  \ee 
\end{proof}

Using the above convolution inequalities and following the same line of the proof in Proposition \ref{pp-si}, we obtain the following 
proposition for BRWF. Then, the conclusion in Theorem \ref{thm-fsi} follows directly by this proposition.

\begin{proposition}\label{pp-sif}
	Let $p, q \in (0,\fy]$,  $\Om\in \mathscr{P}(\rr^{2(m+1)d})$.
	Let $\G=\{j: p_j>2, 0\leq j\leq m\}$.
	Suppose that $p_i, q_i\in (0,\fy)$ for $i\in \G$.
	Then the following statements are equivalent.
	\begin{enumerate}
		\item The following boundedness is valid
		\be
		R_m: W(L^{p_0},L^{q_0}_{\mu_0})(\rd)\times\cdots \times W(L^{p_m},L^{q_m}_{\mu_m})(\rd)\longrightarrow \scrF M^{p,q}_{\Om}(\rmdd).
		\ee
		\item Let $\vec{b_j}=\{b_{j}(k_j,n_j)\}_{k_j,n_j\in \zd}$ for $j\in\G$. 
		  Let $\phi$ be a smooth function which satisfies $\text{supp}\phi\subset Q_0$ and $\phi=1$ on $\frac{Q_0}{2}$.
		  For any Schwartz function sequences $f_j$ for $j\notin \G$,
		denote
		$\{b_j(k_j,n_j)\}_{k_j,n_j\in \zd}:=\{V_{\phi}f_j(k_j,n_j)\}_{k_j,n_j\in \zd}$ for $j\notin \G$, we have
		\be
		\|\overrightarrow{S_{p,\Om}}(\vec{b_0},\vec{b_1},\cdots,\vec{b_m})\|_{l^q}
		\lesssim
		\prod_{j\in \G}\big\|(\|b_j(k_j,\cdot)\|_{l^2})_{k_j}\big\|_{l^{q_j}_{\mu_j}}
		\prod_{j\notin \G}\|f_j\|_{W(L^{p_j},L^{q_j}_{\mu_j})}.
		\ee
		\item Let $\d\in (0,1/4)$, and $f_j$ be a sequence of Schwartz functions for $j\notin \G$.
		  Let
		   $f_j=\sum_{k_j\in \zd}f_{j,k_j}$ with $\text{supp}f_{j,k_j}\subset B(k_j,\d)$, for $j\in \G$.
		   Let $\phi$ be a smooth function which satisfies $\text{supp}\phi\subset Q_0$ and $\phi=1$ on $\frac{Q_0}{2}$.
		  Denote
		   $\{b_j(k_j,n_j)\}_{k_j,n_j\in \zd}:=\{V_{\phi}f_j(k_j,n_j)\}_{k_j,n_j\in \zd}$ for $j=0,1,\cdots,m$.
		   We have
		\be
		\|\overrightarrow{S_{p,\Om}}(\vec{b_0},\vec{b_1},\cdots,\vec{b_m})\|_{l^q}
		\lesssim
		\prod_{j\in \G}\big\|(\|f_{j,k_j}\|_{L^{2}})_{k_j}\big\|_{l^{q_j}_{\mu_j}}
		\prod_{j\notin \G}\|f_j\|_{W(L^{p_j},L^{q_j}_{\mu_j})}.
		\ee
		\item The following boundedness is valid
		\be
		R_m: W(L^{p_0\wedge 2},L^{q_0}_{\mu_0})(\rd)\times\cdots \times W(L^{p_m\wedge 2},L^{q_m}_{\mu_m})(\rd)\longrightarrow \scrF M^{p,q}_{\Om}(\rmdd).
		\ee
	\end{enumerate}
\end{proposition}
\begin{proof}[\textbf{Sketch of the proof}]
	As in the proof of Proposition \ref{pp-si}, we only consider the case $\G=\{0\}$. 
	Let $\vec{b}_{i}, \vec{b}_{i,N}$, $i=0,\cdots,m$, be the same meaning in the proof of Proposition \ref{pp-si}.
	Then, this proof can be done step by step as follows.
	
	\textbf{Step 1.}
	By Lemma \ref{lm-cvf} and the trick of taking expectation, we first establish the sparse estimate 
	 \be
	\begin{split}
		\|\overrightarrow{S_{p,\Om_N}}(\vec{b_{0,N}},\vec{b_{1,N}},\cdots,\vec{b_{m,N}})\|_{l^q}
		\lesssim
		\big\|(\|b_{0}(k_0,\cdot)\|_{l^2})_{k_0}\big\|_{l^{q_0}_{\mu_0}}
		\prod_{j=1}^m\|f_j\|_{W(L^{p_j},L^{q_j}_{\mu_j})},
	\end{split}
	\ee
	where
	\be
	\begin{split}
		&S_{p,\Om_N}(\vec{b_{0,N}},\vec{b_{1,N}},\cdots,\vec{b_{m,N}})(n_0,\vec{n})
		\\
		= &
		\bigg(\sum_{k_0\in \zd, \vec{k}\in \zmd}|b_{0,N}(n_0,-k_0+\sum_{j=1}^mn_j)\prod_{j=1}^mb_{j,N}(-k_j+n_0,n_j) \Om_N((k_0,\vec{k}),(n_0,\vec{n}))|^p\bigg)^{1/p}.
	\end{split}
	\ee

   \textbf{Step 2.} Using the sparse estimate and a decomposition of $\zmdd\times\zmdd$, we get the full version of estimate,
   that is, we get the conclusion in (2). At this point, we have completed the proof of $(1)\Longrightarrow (2)$.
   
   \textbf{Step 3.} 
   In (2),
   take $f_0=\sum_{k_0\in \zd}f_{0,k_0}$ with $\text{supp}f_{0,k_0}\subset B(k_0,\d)$, and let
   $b_{0}(k_0,n_0):=\widehat{f_{0,k_0}}(n_0)$ for $k_0,n_0\in \zd$.
   Then, we get the conclusion (3).
      
   \textbf{Step 4.} 
   In order to verify (4), we only need to consider the boundedness for 
   $f_j=\sum_{k_j\in \zd}f_{j,k_j}$ with $\text{supp}f_{j,k_j}\subset B(k_j,\d)$.
   Using conclusion (3),
   we obtain
   \be
   \|\overrightarrow{S_{p,\Om}}(\vec{b_{0}},\vec{b_{1}},\cdots,\vec{b_{m}})\|_{l^q}
   \lesssim
   \big\|(\|f_{0,k_0}\|_{L^{2}})_{k_0}\big\|_{l^{q_0}_{\mu_0}}
   \prod_{j=1}^m\big\|(\|f_{j,k_j}\|_{L^{p_j}})_{k_j}\big\|_{l^{q_j}_{\mu_j}}
   \ee
   with
   \be
   \begin{split}
   	\|\overrightarrow{S_{p,\Om}}(\vec{b_{0}},\vec{b_{1}},\cdots,\vec{b_{m}})\|_{l^q}
   	=
   	\|\widehat{f_{0,n_0}}(-k_0+\sum_{j=1}^mn_j)
   	\prod_{j=1}^m \widehat{f_{j,-k_j+n_0}}(n_j)\|_{l^{p,q}_{\Om}}.
   \end{split}
   \ee
	By Lemma \ref{cy-eqn}, there exists a constant $N\in \zz^+$ such that
	\be
	\begin{split}
		&\|R_m(f_0,\vec{f})\|_{\scrF M^{p,q}_{\Om}}
		\\
		\sim &
		\bigg\|V_{\phi}f_0(\frac{n_0}{N},\frac{-k_0+\sum_{j=1}^mn_j}{N})
		\prod_{j=1}^mV_{\phi}f_j(\frac{-k_j+n_0}{N},\frac{n_j}{N})\Om((\frac{k_0}{N}\frac{\vec{k}}{N}),(\frac{n_0}{N},\frac{\vec{n}}{N}))\bigg\|_{l^{p,q}(\zmdd\times\zmdd)}.
	\end{split}
	\ee
	Using the finite partition of $\mathbb{Z}^{2(m+1)d}$ mentioned in the proof of Proposition \ref{pp-si}:
	\be
	\mathbb{Z}^{2(m+1)d}=\bigcup_{(\vec{i},\vec{l})\in \La}\G_{\vec{i},\vec{l}},
	\ee
	for a smooth function $\va$ supported on $B(0,\d)$,
	we obtain the following estimate:
	\be
	\begin{split}
		&\bigg\|V_{\va}f_0(\frac{n_0}{N},\frac{-k_0+\sum_{j=1}^mn_j}{N})
		\prod_{j=1}^mV_{\va}f_j(\frac{-k_j+n_0}{N},\frac{n_j}{N})\Om((\frac{k_0}{N}\frac{\vec{k}}{N}),(\frac{n_0}{N},\frac{\vec{n}}{N}))\bigg\|_{l^{p,q}(\zmdd\times\zmdd)}.
		\\
		\sim&
		\sum_{(\vec{i},\vec{l})\in \La}\|V_{\va}f_0(n_0+\frac{l_0}{N},-k_0+\sum_{j=1}^mn_j+\frac{-i_0+\sum_{j=1}^ml_j}{N})
		\\
		&\ \ \ \ \ \ \ \ \cdot\prod_{j=1}^mV_{\va}f_j(-k_j+n_0+\frac{-i_j+l_0}{N},n_j+\frac{l_j}{N})\Om((k_0,\vec{k}),(n_0,\vec{n}))\|_{l^{p,q}(\zmdd\times\zmdd)}
		\\
		= &
		\sum_{(\vec{i},\vec{l})\in \La}
		\|\scrF(f_0\overline{M_{\frac{-i_0+\sum_{j=1}^ml_j}{N}}T_{n_0+\frac{l_0}{N}}\va})(-k_0+\sum_{j=1}^mn_j)
		\prod_{j=1}^m\scrF(f_j\overline{M_{\frac{l_j}{N}}T_{-k_j+n_0+\frac{-i_j+l_0}{N}}\va})(n_j)\|_{l^{p,q}_{\Om}(\zmdd\times\zmdd)}.
	\end{split}
	\ee
	For every $(\vec{i},\vec{l})\in \La$, in \eqref{pp-si-10}, replacing $f_0$ by
	\be
	F_0=T_{-\frac{l_0}{N}}f_0\sum_{k_0\in \zd}\overline{M_{\frac{-i_0+\sum_{j=1}^ml_j}{N}}T_{k_0}\va}
	=
	\sum_{k_0\in \zd}\big(T_{-\frac{l_0}{N}}f_0\big)\overline{M_{\frac{-i_0+\sum_{j=1}^ml_j}{N}}T_{k_0}\va}
	=\sum_{k_0\in \zd}F_{0,k_0,}
	\ee
	and replacing $f_j$ by
	\be
	F_j=T_{\frac{i_j-l_0}{N}}f_j\sum_{k_j\in \zd}\overline{M_{\frac{l_j}{N}}T_{k_j}\va}
	=
	\sum_{k_j\in \zd}\big(T_{\frac{i_j-l_0}{N}}f_j\big)\overline{M_{\frac{l_j}{N}}T_{k_j}\va}
	=\sum_{k_j\in \zd}F_{j,k_j},
	\ee
	we obtain
	\be
	\begin{split}
		&\|\scrF(f_0\overline{M_{\frac{-i_0+\sum_{j=1}^ml_j}{N}}T_{n_0+\frac{l_0}{N}}\va})(-k_0+\sum_{j=1}^mn_j)
		\prod_{j=1}^m\scrF(f_j\overline{M_{\frac{l_j}{N}}T_{-k_j+n_0+\frac{-i_j+l_0}{N}}\va})(n_j)\|_{l^{p,q}_{\Om}(\zmdd\times\zmdd)}.
		\\
		= &
		\|\widehat{F_{0,n_0}}(-k_0+\sum_{j=1}^mn_j)
		\prod_{j=1}^m \widehat{F_{j,-k_j+n_0}}(n_j)\|_{l^{p,q}_{\Om}}
		\\
		\lesssim &
		\big\|(\|F_{0,k_0}\|_{L^{2}})_{k_0}\big\|_{l^{q_0}_{\mu_0}}
		\prod_{j=1}^m\big\|(\|F_{j,k_j}\|_{L^{p_j}})_{k_j}\big\|_{l^{q_j}_{\mu_j}}
		\lesssim
		\big\|(\|f_{0,k_0}\|_{L^{2}})_{k_0}\big\|_{l^{q_0}_{\mu_0}}
		\prod_{j=1}^m\big\|(\|f_{j,k_j}\|_{L^{p_j}})_{k_j}\big\|_{l^{q_j}_{\mu_j}}.
	\end{split}
	\ee
	Recall that $\La$ is a finite subset of $\mathbb{Z}^{2(m+1)d}$.
	By a summation of the above terms with respect to all $(\vec{i},\vec{l})\in \La$, we conclude the desired conclusion
	\be
	\begin{split}
		\|R_m(f_0,\vec{f})\|_{\scrF M^{p,q}_{\Om}}
		\lesssim &
		\sum_{(\vec{i},\vec{l})\in \La}
		\big\|(\|f_{0,k_0}\|_{L^{2}})_{k_0}\big\|_{l^{q_0}_{\mu_0}}
		\prod_{j=1}^m\big\|(\|f_{j,k_j}\|_{L^{p_j}})_{k_j}\big\|_{l^{q_j}_{\mu_j}}
		\\
		\lesssim &
		\big\|(\|f_{0,k_0}\|_{L^{2}})_{k_0}\big\|_{l^{q_0}_{\mu_0}}
		\prod_{j=1}^m\big\|(\|f_{j,k_j}\|_{L^{p_j}})_{k_j}\big\|_{l^{q_j}_{\mu_j}}.
	\end{split}
	\ee
\end{proof}

\subsection{Self-improvement of embedding relations}
In this subsection, we study the embedding relation by using the self-improvement method established in Proposition \ref{pp-si}.
First, we give the self-improvement of the embedding relations between Wiener amalgam and Fourier modulation spaces.

\begin{theorem}\label{thm-siebm}
	Let $0<p_1,p_2,q_1,q_2\leq \fy$, $\Om\in \mathscr{P}(\rdd)$, $\mu\in \scrP(\rd)$.
	Then, if $p_1,q_1<\fy$, the embedding relation 
	\ben\label{thm-siebm-cd1}
	W(L^{p_1},L^{q_1}_{\mu})(\rd)\subset \scrF M^{p_2,q_2}_{\Om}(\rd)
	\een
	can be self-improved to 
	\ben\label{thm-siebm-cd2}
	W(L^{p_1\wedge 2},L^{q_1}_{\mu})(\rd)\subset \scrF M^{p_2,q_2}_{\Om}(\rd).
	\een
	On the other hand, the embedding relation 
	\ben\label{thm-siebm-cd3}
	\scrF M^{p_2,q_2}_{\Om}(\rd)\subset  W(L^{p_1},L^{q_1}_{\mu})(\rd)
	\een
	can be self-improved to 
	\ben\label{thm-siebm-cd4}
	\scrF M^{p_2,q_2}_{\Om}(\rd) \subset W(L^{p_1\vee 2},L^{q_1}_{\mu})(\rd).
	\een
\end{theorem}
\begin{proof}
	The self-improvement of \eqref{thm-siebm-cd1} follows by a similar argument as in the proof of  Proposition \ref{pp-si}.
	The self-improvement of \eqref{thm-siebm-cd3} is easier to prove than that of \eqref{thm-siebm-cd1}, but in a slight different way.
	Therefore, we give an abbreviated proof here.
	
	For a fixed truncated sequence (only finite nonzero items) $\vec{c}=\{c(k,n)\}_{k,n\in \zd}$,
	we set
	\be
	g(x)=\sum_{k\in \zd}\sum_{n\in \zd}c(k,n)e^{2\pi i n\cdot x}T_{k}\va(x)=\sum_{k\in \zd}g_{k},
	\ee
	where $\va$ is a nonzero smooth function supported on $B_{\d}$ with $\d<1/4$. 
	Recalling the fact $(M_{n}T_k\va)^{\vee}(\xi)=T_{-n}M_{k}\check{\va}(\xi)$, we write
	\be
	\begin{split}
	\check{g}(\xi)
	= &
	\sum_{k\in \zd}\sum_{n\in \zd}c(k,n)T_{-n}M_{k}\check{\va}(\xi)
	\\
	= &
	\sum_{k\in \zd}\sum_{n\in \zd}c(n,-k)T_{k}M_{n}\check{\va}(\xi)= :\sum_{k\in \zd}\sum_{n\in \zd}b(k,n)T_{k}M_{n}\check{\va}(\xi)=D_{\check{\va}}^{1,1}\vec{b},
	\end{split}
	\ee
	where we denote $b(k,n)=c(n,-k)$.
	Using the boundedness of synthesis operator (see Lemma \ref{lm, bdCD}), we obtain that
	\be
	\|g\|_{\scrF M^{p_2,q_2}_{\Om}(\rd)}=\|\check{g}\|_{M^{p_2,q_2}_{\Om}}
	=\|D_{\check{\va}}^{1,1}\vec{b}\|_{M^{p_2,q_2}_{\Om}}
	\lesssim \|\vec{b}\|_{l^{p_2,q_2}_{\Om}}.
	\ee
	If the embedding relation \eqref{thm-siebm-cd3} holds, we conclude that
	\be
	\big\|\big(\|g_k\|_{L^{p_1}}\big)_{k}\big\|_{l^{q_1}_{\mu}} \lesssim \|\vec{b}\|_{l^{p_2,q_2}_{\Om}}.
	\ee
	Using the trick of taking expectation as in the proof of Proposition \ref{pp-si}, we can conclude that
	\ben\label{thm-siebm-1}
	\big\|\big(\|c(k,\cdot)\|_{l^2}\big)_{k}\big\|_{l^{q_1}_{\mu}} \lesssim \|\vec{b}\|_{l^{p_2,q_2}_{\Om}}.
	\een
	Next, we turn to the proof of \eqref{thm-siebm-cd4} for $p_1<2$. By a reduction as in the proof of Theorem \ref{thm-M0}, we only need to verify the conclusion for $f=\sum_{j\in \zd}f_j$ with 
	$\text{supp}f_j\subset B(j,\d)$ with $\d\in (0,1/4)$.
	Take $c(k,n)=\widehat{f_k}(n)$, from \eqref{thm-siebm-1} we conclude that
	\be
	\big\|\big(\|f_k\|_{L^{2}}\big)_{k}\big\|_{l^{q_1}_{\mu}}
	\sim
		\big\|\big(\|c(k,\cdot)\|_{l^2}\big)_{k}\big\|_{l^{q_1}_{\mu}} \lesssim \|\vec{b}\|_{l^{p_2,q_2}_{\Om}}.
	\ee
	Observe that the window function $\phi$ supported on $Q_0$ satisfies $\phi=1$ on $\frac{Q_0}{2}$, so we have
	\be
	\begin{split}
	&\|\vec{b}\|_{l^{p_2,q_2}_{\Om}}
	= \|(c(n,-k))_{k,n}\|_{l^{p_2,q_2}_{\Om}}
	= 
	\|(\widehat{f_n}(-k))_{k,n}\|_{l^{p_2,q_2}_{\Om}}
	\\
	= &
	\|(V_{\phi}f(n,-k)\|_{l^{p_2,q_2}_{\Om}}
	=
	\|(V_{\check{\phi}}\check{f}(k,n)\|_{l^{p_2,q_2}_{\Om}}
	\lesssim 
	\|\scrF^{-1}f\|_{M^{p_2,q_2}_{\Om}}=\|f\|_{\scrF M^{p_2,q_2}_{\Om}},
	\end{split}
	\ee
	where we use the sampling property of $M^{p_2,q_2}_{\Om}$.
	A combination of the above two estimates yields the desired conclusion.
	
\end{proof}

Correspondingly, we give the self-improvement result for the embedding relations between Wiener amalgam and modulation spaces.
Since the proof is similar, we omit it.
\begin{theorem}\label{thm-siebf}
	Let $0<p_1,p_2,q_1,q_2<\fy$, $\Om\in \mathscr{P}(\rdd)$, $\mu\in \scrP(\rd)$.
	Then, if $p_1,q_1<\fy$, the embedding relation 
	\be
	W(L^{p_1},L^{q_1}_{\mu})(\rd)\subset M^{p_2,q_2}_{\Om}(\rd)
	\ee
	can be self-improved to 
	\be
	W(L^{p_1\wedge 2},L^{q_1}_{\mu})(\rd)\subset M^{p_2,q_2}_{\Om}(\rd).
	\ee
	On the other hand, the embedding relation 
	\be
	 M^{p_2,q_2}_{\Om}(\rd)\subset  W(L^{p_1},L^{q_1}_{\mu})(\rd)
	\ee
	can be self-improved to 
	\be
	M^{p_2,q_2}_{\Om}(\rd) \subset W(L^{p_1\vee 2},L^{q_1}_{\mu})(\rd).
	\ee
\end{theorem}

\section{The sharp exponents characterizations}
\subsection{Sharp exponents of local version of BRWM}
\begin{lemma}[Sharpness of convolution inequality, see \cite{GuoChenFanZhao2019MMJ}]\label{lm-mY}
Let $m\geq 1$ be an integer. Suppose $0<q,q_j \leq \infty$ for $j=0,1,\cdots ,m$. Let $S=\{j\in \mathbb{Z}: \ q_j\geq 1, 0\leq j\leq m\}$.
Then
\be
l^{q_0}(\zd)\ast l^{q_1}(\zd)\ast\cdots\ast l^{q_m}(\zd)  \subset l^{q}(\zd)
\ee
holds if and only if
\be
1/q\leq 1/q_j\ \ \ \ (j=0,1\cdots m)
\ee
and
\be
(|S|-1)+1/q\leq \sum_{j\in S} 1/{q_j},\ \text{for}\  |S|\geq 1.
\ee
\end{lemma}

\begin{proposition}\label{pp-M2L}
  Let $p, q, p_j\in (0,\fy]$, $j=0,1,\cdots,m$.
  Denote by
  \be
  \La:=\bigg\{j: j=0,1,\cdots,m,\  \frac{1}{p}\geq 1-\frac{1}{p_j\wedge 2}\bigg\}.
  \ee
  We have 
  \ben\label{pp-M2L-cd1}
  R_m: L^{p_0\wedge 2}(B_{\d})\times\cdots \times L^{p_m\wedge 2}(B_{\d})\longrightarrow M^{p,q}(\rmdd)
  \een
  holds if and only if
  \ben\label{pp-M2L-cd2}
  \frac{1}{q}\leq 1-\frac{1}{p_j\wedge 2},\ \ \ \ j=0,1,\cdots,m,
  \een
  and
  \ben\label{pp-M2L-cd3}
  \frac{|\La|-1}{p}+\frac{1}{q}\leq |\La|-\sum_{j\in \La} \frac{1}{p_j\wedge 2}\ \text{for}\ |\La|\geq 1.
  \een
\end{proposition}
\begin{proof}
This proof is divided into several parts.

  \textbf{The proof of $\eqref{pp-M2L-cd1}\Longrightarrow \eqref{pp-M2L-cd2}$.}
  Using Lemma \ref{lm-lbeq}, \eqref{pp-M2L-cd1} implies the following embedding relations:
  \ben\label{pp-M2L-2}
  L^{p_i\wedge 2}(B_{\d})\subset \scrF^{-1}L^q(\rd),\ \ \ i=0,1,2,\cdots,m.
  \een
  We claim that 
  \ben\label{pp-M2L-3}
  L^{p_i\wedge 2}(B_{\d})\subset \scrF^{-1}L^q(\rd)\Longleftrightarrow \frac{1}{q}\leq 1-\frac{1}{p_i\wedge 2},\ \ \ i=0,1,2,\cdots,m.
  \een

  Take $f$ to be a smooth function supported on $B_{\d}$. Denote $f_{\la}(x):=\frac{1}{\la^d}f(\frac{x}{\la})$ for $\la\in (0,1)$.
  Then, the embedding relation $L^{p_i\wedge 2}(B_{\d})\subset \scrF^{-1}L^q$ and a direct calculation tell us that
  \be
  \la^{-d/q}\sim
  \|f_{\la}\|_{\scrF^{-1}L^q}\lesssim \|f_{\la}\|_{L^{p_i\wedge 2}}\sim \la^{d(\frac{1}{p_i\wedge 2}-1)},\ \ \ \la\in (0,1),
  \ee
  which implies $-\frac{1}{q}\geq \frac{1}{p_i\wedge 2}-1$. This is equivalent to the desired conclusion
  $\frac{1}{q}\leq 1-\frac{1}{p_i\wedge 2}$.

  Next, we turn to verify the opposite direction of \eqref{pp-M2L-3}.
  Observing $\frac{1}{q}\leq 1-\frac{1}{p_i\wedge 2}\leq \frac{1}{2}$ and $\frac{1}{q'}\geq \frac{1}{p_i\wedge 2}$,
  we use the Hausdorff-Young inequality $L^{q'}(\rd)\subset \scrF^{-1}L^q(\rd)$  and the embedding relation $L^{p_i\wedge 2}(B_{\d})\subset L^{q'}(B_{\d})$
  to obtain that
  \be
  L^{p_i\wedge 2}(B_{\d})\subset L^{q'}(B_{\d})\subset \scrF^{-1}L^q(\rd).
  \ee
  This completes the proof of claim \eqref{pp-M2L-3}.

  \textbf{The proof of $\eqref{pp-M2L-cd1}\Longrightarrow \eqref{pp-M2L-cd3}$.}
  In this case, we assume $|\La|\geq 1$.
  Let $\phi$ be a smooth function supported in $Q_0$, satisfying $\phi=1$ on $\frac{Q_0}{2}$.
  Using the sampling property of STFT, we obtain that $\eqref{pp-M2L-cd1}$ implies
  \ben\label{pp-M2L-4}
\begin{split}
\|\overrightarrow{T_{p}}(\vec{b_0},\vec{b_1},\cdots,\vec{b_m})\|_{l^q}
\lesssim
\prod_{j=0}^m\|f_j\|_{L^{p_j\wedge 2}},
\end{split}
\een
where
\be
\begin{split}
	&T_{p}(\vec{b_0},\vec{b_1},\cdots,\vec{b_m})(n_0,\vec{n})
	= 
	\bigg(\sum_{k_0\in \zd, \vec{k}\in \zmd}|b_0(k_0,n_0+\sum_{j=1}^mk_j)\prod_{j=1}^mb_j(n_j+k_0,k_j)|^p \bigg)^{1/p},
\end{split}
\ee
 $\{f_j\}_{j=0}^m$ are $C_c^{\fy}(\rd)$ function sequences supported on $B_{\d}$,
and $\{b_j(k_j,n_j)\}_{k_j,n_j\in \zd}:=\{V_{\phi}f_j(k_j,n_j)\}_{k_j,n_j\in \zd}$ for $j=0,1,\cdots,m$.

Note that $\vec{b_j}(0,n_j)=V_{\phi}f_j(0,n_j)=\{\widehat{f_j}(n_j)\}_{l\in \zd}$ for $j=0,1,\cdots,m$.
We conclude that
\eqref{pp-M2L-4} implies
\ben\label{pp-M2L-5}
\bigg\|\bigg(\big(\sum_{\vec{k}\in \zmd}|\widehat{f_0}(n_0+\sum_{j=1}^mk_j)\prod_{j=1}^m\widehat{f_j}(k_j)|^p\big)^{1/p}\bigg)_{n_0}\bigg\|_{l^q}
\lesssim
\prod_{j=0}^m\|f_j\|_{L^{p_j\wedge 2}}.
\een

Let $\la\in (0,1)$.
Take $f_j(x)=h_{\la}(x)=\frac{1}{\la^d}h(\frac{x}{\la})$ for $j\in \La$, and $f_j=h$ for $j\notin \La$,
where $h$ is a smooth function satisfying that $\text{supp}h\subset B_{\d}$ and $\widehat{h}(0)=2$.
Then, there exists a constant $C$ such that
\be
\widehat{h_{\la}}(\xi)=\widehat{h}(\la\xi)\geq 1,\ \ \ |\xi|\leq (m+1)C\la^{-1}.
\ee
If $0\in \La$, for sufficiently small $\la$ and $|n_0|\leq C\la^{-1}$, we have
\be
\begin{split}
  &\bigg(\sum_{\vec{k}\in \zmd}\bigg|\widehat{f_{0}}(n_0+\sum_{j=1}^mk_j)\prod_{j=1}^m \widehat{f_{j}}(k_j)\bigg|^p\bigg)^{1/p}
  \\
  \geq &
  \bigg(\sum_{j\in \La, j\neq 0}\sum_{|k_j|\leq C\la^{-1}}\bigg|\widehat{h_{\la}}(n_0+\sum_{j\in \La, j\neq 0}k_j)\prod_{j\in \La, j\neq 0}\widehat{h_{\la}}(k_j)\bigg|^p\bigg)^{1/p}
  \gtrsim
  \la^{\frac{-d(|\La|-1)}{p}}.
\end{split}
\ee
If $0\notin \La$, without loss of generality, we assume $j_0\in \La$.
For sufficiently small $\la$ and $|n_0|\leq C\la^{-1}$, we conclude that
\be
\begin{split}
  &\bigg(\sum_{\vec{k}\in \zmd}\bigg|\widehat{f_{0}}(n_0+\sum_{j=1}^mk_j)\prod_{j=1}^m \widehat{f_{j}}(k_j)\bigg|^p\bigg)^{1/p}
  \\
  \geq &
  \bigg(\sum_{j\in \La}\sum_{k_j\in \zd}\bigg|\widehat{h}(n_0+\sum_{j\in \La}k_j)\prod_{j\in \La}\widehat{h_{\la}}(k_j)\bigg|^p\bigg)^{1/p}
  \\
  \geq &
  \bigg(\sum_{j\in \La, j\neq j_0}\sum_{|k_j|\leq C\la^{-1}}\bigg|\widehat{h}(0)\widehat{h_{\la}}(-n_0-\sum_{j\in \La,j\neq j_0}k_j)\prod_{j\in \La, j\neq j_0}\widehat{h_{\la}}(k_j)\bigg|^p\bigg)^{1/p}
  \gtrsim
  \la^{\frac{-d(|\La|-1)}{p}}.
\end{split}
\ee

With the above estimates for $|n_0|\leq C\la^{-1}$,
and by replacing all the functions $f_j$ by $h_{\la}$ for $j\in \La$, and by replacing $f_j$ by $h$ for $j\notin \La$,
we have the lower bound estimates of the  left term in \eqref{pp-M2L-5}:
\be
\begin{split}
  &\bigg\|\bigg(\big(\sum_{\vec{k}\in \zmd}|\widehat{f_0}(n_0+\sum_{j=1}^mk_j)\prod_{j=1}^m\widehat{f_j}(k_j)|^p\big)^{1/p}\bigg)_{n_0}\bigg\|_{l^q}
  \\
  \gtrsim &
  \bigg\|\bigg(\la^{\frac{-d(|\La|-1)}{p}}\bigg)_{|n_0|\leq C\la^{-1}}\bigg\|_{l^q}
  \gtrsim 
  \la^{\frac{-d(|\La|-1)}{p}}(\sum_{|n_0|\leq C\la^{-1}}1^q)^{1/q}
  \sim \la^{\frac{-d(|\La|-1)}{p}}\la^{-d/q}=\la^{-d(\frac{|\La|-1}{p}+\frac{1}{q})}.
\end{split}
\ee
Combining this with the fact $\|h_{\la}\|_{L^{p_i\wedge 2}}\sim \la^{d(\frac{1}{p_i\wedge 2}-1)}$ for $i\in \La$, we obtain
\be
\la^{-d(\frac{|\La|-1}{p}+\frac{1}{q})}
\lesssim
\|T_p(\vec{b_0},\vec{b_1},\cdots,\vec{b_m})\|_{l^q}
\lesssim
\prod_{j=0}^m\|f_{j}\|_{L^{p_j\wedge 2}}
\sim
\prod_{j\in \La}\la^{d(\frac{1}{p_j\wedge 2}-1)}\ \ \ \ (0<\la<1).
\ee
This implies that
\be
\frac{|\La|-1}{p}+\frac{1}{q}\leq |\La|-\sum_{j\in \La}\frac{1}{p_j\wedge 2}.
\ee

\textbf{The proof of $\eqref{pp-M2L-cd2},\eqref{pp-M2L-cd3}\Longrightarrow \eqref{pp-M2L-cd1}$ for $p\geq q$.}
In this case, by Lemma \ref{lm-lbeq}, we have
$\eqref{pp-M2L-cd1} \Longleftrightarrow \eqref{pp-M2L-2}$.
Then, the conclusion follows by the equivalent relations in \eqref{pp-M2L-3}.

\textbf{The proof of $\eqref{pp-M2L-cd2},\eqref{pp-M2L-cd3}\Longrightarrow \eqref{pp-M2L-cd1}$ for $p< q$.}
In this case, $p<\fy$.
By \eqref{pp-M2L-cd2}, we deduce that $(p_j\wedge 2)\geq 1$, $j=0,1,\cdots,m$.
Denote that
  \be
  \La=\bigg\{j: j=0,1,\cdots,m,\ \frac{(p_j\wedge 2)'}{p}\geq 1\bigg\},
  \ee
  \be
  \eqref{pp-M2L-cd2}\Longleftrightarrow \frac{p}{q}\leq \frac{p}{(p_j\wedge 2)'},
  \ee
and
  \be
  \eqref{pp-M2L-cd3}\Longleftrightarrow
|\La|-1+\frac{p}{q}\leq \sum_{j\in \La}\frac{p}{(p_j\wedge 2)'}.
  \ee
Then, we use \eqref{pp-M2L-cd2}, \eqref{pp-M2L-cd3} and Lemma \ref{lm-mY} to obtain the convolution inequality:
\be
l^{\frac{(p_0\wedge 2)'}{p}}\ast l^{\frac{(p_1\wedge 2)'}{p}}\ast\cdots\ast l^{\frac{(p_m\wedge 2)'}{p}}\subset l^{\frac{q}{p}}.
\ee
Using this, we deduce that
  \be
\begin{split}
&\bigg\|\bigg(\big(\sum_{\vec{k}\in \zmd}|\widehat{f_0}(n_0+\sum_{j=1}^mk_j)\prod_{j=1}^m\widehat{f_j}(k_j)|^p\big)^{1/p}\bigg)_{n_0}\bigg\|_{l^q}
\\
= &
\bigg\|\bigg(\sum_{\vec{k}\in \zmd}|\widehat{f_0}(n_0+\sum_{j=1}^mk_j)\prod_{j=1}^m\widehat{f_j}(k_j)|^p\bigg)_{n_0}\bigg\|_{l^{q/p}}^{1/p}
\lesssim 
\prod_{j=0}^m\|(|\widehat{f_j}(k)|^p)_k\|^{1/p}_{l^{\frac{(p_j\wedge 2)'}{p}}}
=
\prod_{j=0}^m\|(|\widehat{f_j}(k)|)_k\|_{l^{(p_j\wedge 2)'}}.
\end{split}
\ee
By the sampling property of Lebesgue space and the Hausdorff-Young inequality, we have
\be
\|(|\widehat{f_j}(k)|)_k\|_{l^{(p_j\wedge 2)'}}\lesssim \|f_j\|_{\scrF^{-1}L^{(p_j\wedge 2)'}}
\lesssim \|f_j\|_{L^{p_j\wedge 2}}.
\ee
The combination of the above two inequalities implies the inequality \eqref{pp-M2L-5}.
Then, the desired conclusion \eqref{pp-M2L-cd1} follows by a similar argument as in the proof of $(3)\Longrightarrow (4)$ in Proposition \ref{pp-si}.
\end{proof}

\subsection{$\star$ convolution}
Let $\vec{a}=\{a_{k_0}\}_{k_0\in \zd}$ and $\vec{B}=\{B_{\vec{k}}\}_{\vec{k}\in \zmd}$ be
two sequences defined on $\zd$ and $\zmd$ respectively, where $\vec{k}=(k_1,\cdots,k_m)$ be a vector on $\rmd$ with $k_j\in \zd$, $j=1,\cdots,m$.
The $\star$ convolution of $\vec{a}$ and $\vec{B}$ is defined by
\be
(\vec{a}\star \vec{B})(\vec{k})=\sum_{k_0\in \zd}a_{k_0}B(k_1-k_0,k_2-k_0,\cdots,k_m-k_0).
\ee
Note that for $m=1$ the $\star$ convolution recover the usual convolution,
that is, $\star=\ast $ when $m=1$.
The operation of $\star$ convolution appears naturally in the characterization of BRWM in the multilinear setting, where
we will deal with the case that $B=\otimes_{j=1}^mb_j$. For this special case, the $\star$ convolution can be written as
\be
(\vec{a}\star \otimes_{j=1}^mb_j)(\vec{k})=\sum_{k_0\in \zd}a_{k_0}\prod_{j=1}^mb_j(k_j-k_0).
\ee
Moreover, we use $l^{r_0}_{\r_0}(\zd)\star \otimes_{j=1}^m l^{r_j}_{\r_j}(\zd)\subset l^r_{\r}(\zmd)$
to denote the following inequality
\be
\|(\vec{a}\star \otimes_{j=1}^mb_j)(\vec{k})\|_{l^r_{\r}(\zmd)}\leq C \|\vec{a}\|_{l^{r_0}_{\r_0}(\zd)}\prod_{j=1}^m \|\vec{b_j}\|_{l^{r_j}_{\r_j}(\zd)},
\ee
where $\r_j$ and $\r$ are certain weight functions and $r,r_j\in (0,\fy]$.

\begin{proposition}[Sharpness of $\star$ convolution inequality]\label{lm-msc}
Let $m\geq 1$ be an integer. Suppose $0<q,q_j \leq \infty$ for $j=0,1,\cdots ,m$.
Then
\ben\label{lm-msc-cd0}
l^{q_0}(\zd)\star \otimes_{j=1}^ml^{q_j}(\zd) \subset l^{q}(\zmd)
\een
holds if and only if
\ben\label{lm-msc-cd1}
1/q\leq 1/q_j,\ \ \ \ (j=0,1\cdots m)
\een
and
\ben\label{lm-msc-cd2}
1+\frac{m}{q}\leq \sum_{j=0}^m \frac{1}{q_j}.
\een
\end{proposition}
\begin{proof}
We divide this proof into two cases.

\textbf{Case 1: $q\leq 1$.}
In this case, we have $\eqref{lm-msc-cd1}\Longrightarrow \eqref{lm-msc-cd2}$.
Thus, we only need to verify $\eqref{lm-msc-cd0}\Longleftrightarrow \eqref{lm-msc-cd1}$.
Using Lemma \ref{lm-meeq} with $p=1$, if $q\leq 1$, \eqref{lm-msc-cd0} is equivalent to
  \ben\label{lm-msc-1}
  l^{q_j}(\zd)\subset l^q(\zd),\ \ \ j=0,1,\cdots,m,
  \een
which is equivalent to \eqref{lm-msc-cd1}.
This is the desired conclusion.

\textbf{Case 2: $q>1$.} In this case, the proof is divided into two parts.

\textbf{The proof of $\eqref{lm-msc-cd0}\Longrightarrow \eqref{lm-msc-cd1},\eqref{lm-msc-cd2}$.}
Using Lemma \ref{lm-tmsc} with $p=1$, \eqref{lm-msc-cd0} implies the embedding relations \eqref{lm-msc-1},
which is equivalent to \eqref{lm-msc-cd1}.

On the other hand, take $\vec{a}=\vec{b_j}=\sum_{|i|\leq 2N}\vec{e_i}$,
where 
$\vec{e_i}:=\{e_{i,l}\}_{l\in \zd}$, $e_{i,l}=1$ for $l=i$, and vanish elsewhere. A direct calculation yields that
\ben\label{lm-msc-0}
\|\vec{a}\|_{l^{q_0}}\sim N^{d/q_0},\ \ \ \|\vec{b_j}\|_{l^{q_j}}\sim N^{d/q_j},\ \ \ j=1,2,\cdots,m.
\een

For $|k_j|\leq N$, we have
\be
\sum_{k_0\in \zd}a_{k_0}\prod_{j=1}^mb_j(k_j-k_0)
\geq
\sum_{|k_0|\leq N}\prod_{j=1}^mb_j(k_j-k_0)
=
\sum_{|k_0|\leq N}1\sim N^d.
\ee
Then, we have the estimates
\be
\|\sum_{k_0\in \zd}a_{k_0}\prod_{j=1}^mb_j(k_j-k_0)\|_{l^q}
\gtrsim
N^d\bigg(\sum_{|k_j|\leq N, 1\leq j\leq m}1\bigg)^{1/q}
\sim N^dN^{md/q}=N^{d(1+\frac{m}{q})}.
\ee
Using this and \eqref{lm-msc-0}, \eqref{lm-msc-cd0} implies that for sufficiently large $N$,
\be
N^{d(1+\frac{m}{q})}\lesssim \|\sum_{k\in \zd}a_k\prod_{j=1}^mb_j(k-k_j)\|_{l^q}
\lesssim
\|\vec{a}\|_{l^{q_0}}\prod_{j=1}^m\|\vec{b_i}\|_{l^{q_i}}
\lesssim
N^{d/q_0}\prod_{j=1}^mN^{d/q_i}
=
N^{d\sum_{j=0}^m\frac{1}{q_j}}.
\ee
It follows that the desired condition $\eqref{lm-msc-cd2}$ is valid.

\textbf{The proof of $\eqref{lm-msc-cd1},\eqref{lm-msc-cd2}\Longrightarrow\eqref{lm-msc-cd0}$.}
First, we claim that there exists $r_j\in [1,q]$ such that
\ben\label{lm-msc-2}
\frac{1}{q}\leq \frac{1}{r_j}\leq \frac{1}{q_j},\ \ \ 1+\frac{m}{q}= \sum_{j=0}^m \frac{1}{r_j}.
\een
In fact, if there exists a $q_{i}\leq 1$, we take $r_i=1$ and $r_j=q$ for $j\neq i$.
If $q_j>1$ for all $j=0,1,\cdots, m$, observe that
\be
\sum_{j=0}^m \frac{1}{q}<1+\frac{m}{q}\leq \sum_{j=0}^m \frac{1}{q_j}.
\ee
There exists a sequence $\{r_j\}_{j=0}^m$ satisfying \eqref{lm-msc-2}.
Moreover, since $q_j>1$, we have $r_j\geq q_j\geq 1$ for all $j=0,1,\cdots, m$.

In light of the claim \eqref{lm-msc-2} and the known embedding relations $l^{q_j}\subset l^{r_j}$ for all $0\leq j\leq m$, we only need
to verify that
\ben\label{lm-msc-3}
\frac{1}{q}\leq \frac{1}{r_j}\leq 1\ \text{and} \ 1+\frac{m}{q}= \sum_{j=0}^m \frac{1}{r_j}\ \  \text{imply}\ \
l^{r_0}(\zd)\star \otimes_{j=1}^ml^{r_j}(\zd) \subset l^{q}(\zmd).
\een
If there exits a $r_i=\fy$, then necessarily $q=\fy$ and $1=\sum_{j=0}^m\frac{1}{r_j}$, and the inequality is trivial to prove by H\"{o}lder's inequality. Thus, let us assume that $q<\fy$ and $r_j<\fy$ for all $j=0,1,\cdots,m$.
Without loss of generality, we also assume $\|\vec{a}\|_{l^{r_0}}=\|\vec{b_j}\|_{l^{r_j}}=1$ for all $j=1,\cdots,m$.
Observe that
\be
1+\frac{m}{q}= \sum_{j=0}^m \frac{1}{r_j}\Longleftrightarrow
1=\frac{1}{q}+\sum_{j=0}^m \big(\frac{1}{r_j}-\frac{1}{q}\big)= :\frac{1}{q}+\sum_{j=0}^m \frac{1}{\r_j}.
\ee
Next, we split the product $|a_k\prod_{j=1}^mb_j(k-k_j)|$ by
\be
|a_k|^{1-\frac{r_0}{q}}\big(|a_k|^{\frac{r_0}{q}}\prod_{j=1}^m|b_j(k-k_j)|^{\frac{r_j}{q}}\big)\prod_{j=1}^m|b_j(k-k_j)|^{1-\frac{r_j}{q}}
\ee
and apply H\"{o}lder's inequality with exponents $\r_0, q, \r_1,\cdots,\r_m$ to this product, we get
\ben
\begin{split}
|\sum_{k\in\zd}a_k\prod_{j=1}^mb_j(k-k_j)|
\leq &
\|\vec{a}\|_{l^{r_0}}^{1-\frac{r_0}{q}}
\bigg(\sum_{k\in \zd}|a_k|^{r_0}\prod_{j=1}^m|b_j(k-k_j)|^{r_j}\bigg)^{1/q}
\prod_{j=1}^m\|\vec{b_j}\|^{1-\frac{r_j}{q}}_{l^{r_j}}
\\
= &
\bigg(\sum_{k\in \zd}|a_k|^{r_0}\prod_{j=1}^m|b_j(k-k_j)|^{r_j}\bigg)^{1/q}.
\end{split}
\een
Now, applying $l^q(\zd)$ norm to the above inequality, we get
\be
\begin{split}
&\|\sum_{k\in\zd}a_k\prod_{j=1}^mb_j(k-k_j)\|_{l^q(\zd)}
\leq
\bigg(\sum_{\vec{k}\in \zmd}\sum_{k\in \zd}|a_k|^{r_0}\prod_{j=1}^m|b_j(k-k_j)|^{r_j}\bigg)^{1/q}
\\
&\quad \quad \quad \quad  =
\bigg(\sum_{k\in \zd}|a_k|^{r_0}\prod_{j=1}^m\big(\sum_{k_j\in \zd}|b_j(k-k_j)|^{r_j}\big)\bigg)^{1/q}=1.
\end{split}
\ee
We have now completed this proof.
\end{proof}

\begin{lemma}\label{lm-tmsc}
  Assume that $p,q,q_i\in (0,\fy]$, $i=0,1,\cdots,m$. The following inequality
  \ben\label{lm-tmsc-cd1}
  \tau_m(\otimes_{j=0}^ml^{q_j}(\zd))\subset l^{p,q}(\zd\times\zmd)
  \een
  holds if and only if
\begin{eqnarray}\label{lm-tmsc-cd2}
&l^{q_0/p}(\zd)\star \big(\otimes_{j=1}^m l^{q_j/p}(\zd)\big)\subset l^{q/p}(\zmd),\ \ \
&p<\fy,
\\
\label{lm-tmsc-cd3}
&l^{q_i}(\zd)\subset l^q(\zd),\ \ \ \ \ \ \ i=0,1,\cdots,m, \ \ \
&p\geq q.
\end{eqnarray}
\end{lemma}
\begin{proof}
Observe that if $p<\fy$, \eqref{lm-tmsc-cd2} is equivalent to
\be
\tau_m(\otimes_{j=0}^ml^{q_j/p}(\zd))\subset l^{1,q/p}(\zd\times\zmd).
\ee
Using this and the fact
\be
\tau_m(\otimes_{j=0}^ml^{q_j}(\zd))\subset l^{1,q}(\zd\times\zmd)
\Longleftrightarrow l^{q_0}(\zd)\star \otimes_{j=1}^ml^{q_j}(\zd) \subset l^{q}(\zmd),
\ee
we obtain the equivalent relation $\eqref{lm-tmsc-cd1}\Longleftrightarrow \eqref{lm-tmsc-cd2}$.
	
	If $p\geq q$, the equivalent relation $\eqref{lm-tmsc-cd1}\Longleftrightarrow \eqref{lm-tmsc-cd3}$ follows by Lemma \ref{lm-meeq}.
\end{proof}

\begin{proposition}\label{pp-M2G}
  Let $p, q, q_i \in (0,\fy]$, $i=0,1,\cdots,m$. Then, the following inequality
   \be
  \tau_m(\otimes_{j=0}^ml^{q_j}(\zd))\subset l^{p,q}(\zd\times\zmd)
  \ee
holds if and only if
\ben\label{pp-M2G-cd1}
\frac{1}{p}+\frac{m}{q}\leq \sum_{j=0}^m \frac{1}{q_j}
\een
and
\ben\label{pp-M2G-cd2}
1/q\leq 1/q_j,\ \ \ \ (j=0,1\cdots m).
\een
\end{proposition}
\begin{proof}
This proof is divided into two cases.

\textbf{Case 1: $p< \fy$.}
The desired conclusion follows by Lemmas \ref{lm-tmsc} and \ref{lm-msc},
and the fact that the following conditions
\be
1+\frac{mp}{q}\leq \sum_{j=0}^m\frac{p}{q_j}
\ \ \ \text{and}\ \ \
\frac{p}{q}\leq \frac{p}{q_j},\ \ \ (j=0,1,\cdots,m),
\ee
are equivalent to \eqref{pp-M2G-cd1} and \eqref{pp-M2G-cd2}, respectively.

\textbf{Case 2: $p\geq q$.}
In this case, the desired conclusion follows by Lemma \ref{lm-tmsc} and the fact that
$\eqref{pp-M2G-cd2}\Longrightarrow \eqref{pp-M2G-cd1}$.
\end{proof}

\subsection{Proof of Theorem \ref{thm-M2}}
Using Theorem \ref{thm-M1-sp}, we have the following result.
\begin{proposition}\label{pp-M2}
  Assume $p_i, q_i, p, q \in (0,\fy]$, $i=0,1,2,\cdots, m$.
We have
  \be
  R_m: W(L^{p_0},L^{q_0})(\rd)\times\cdots \times W(L^{p_m},L^{q_m})(\rd)\longrightarrow M^{p,q}(\rmdd)
  \ee
if and only if for some $\d>0$
  \be
  R_m: L^{p_0}(B_{\d})\times\cdots \times L^{p_m}(B_{\d})\longrightarrow M^{p,q}(\rmdd),
  \ee
  and
    \be
  \tau_m(\otimes_{j=0}^ml^{q_j}(\zd))\subset l^{p,q}(\zd\times\zmd).
  \ee
\end{proposition}

\begin{proof}[The proof of Theorem \ref{thm-M2}]
The sufficiency follows by Theorem \ref{thm-M1-sp}, Proposition \ref{pp-M2L}, Proposition \ref{pp-M2G} and the fact 
$W(L^{p_i}, L^{q_i})\subset W(L^{p_i\wedge 2}, L^{q_i})$.

The necessity for $p_i, q_i<\fy$ follows by Theorem \ref{thm-msi}, Proposition \ref{pp-M2}, Proposition \ref{pp-M2L} and Proposition \ref{pp-M2G}.
If there is some $p_i=\fy$ or $q_i=\fy$, by a complex interpolation between \eqref{thm-M2-cd0} and
  \be
  R_m: W(L^{2},L^{2})(\rd)\times\cdots \times W(L^{2},L^{2})(\rd)\longrightarrow M^{2,2}(\rmdd),
  \ee
we get the following boundedness result
  \be
  R_m: W(L^{\widetilde{p_0}},L^{\widetilde{q_0}})(\rd)\times\cdots \times W(L^{\widetilde{p_m}},L^{\widetilde{q_m}})(\rd)
  \longrightarrow M^{\widetilde{p},\widetilde{q}}(\rmdd),
  \ee
where
\ben\label{pf-M2-5}
\frac{1}{\widetilde{p}}=\frac{1-\th}{2}+\frac{\th}{p},\ \
\frac{1}{\widetilde{q}}=\frac{1-\th}{2}+\frac{\th}{q},\ \
\frac{1}{\widetilde{p_j}}=\frac{1-\th}{2}+\frac{\th}{p_j},\ \
\frac{1}{\widetilde{q_j}}=\frac{1-\th}{2}+\frac{\th}{q_j},\
\een
for $j=0,1,\cdots,m$ and some $\th\in (0,1)$.
Observe that $\widetilde{p_j}, \widetilde{q_j}<\fy$ for all $0\leq j\leq m$.
We get the necessary conditions as follows:
\ben\label{pf-M2-1}
  \frac{1}{\widetilde{q}}\leq 1-\frac{1}{\widetilde{p_j}\wedge 2},\ \ \ \ j=0,1,\cdots,m,
  \een
  \ben\label{pf-M2-2}
  \frac{|\widetilde{\La}|-1}{\widetilde{p}}+\frac{1}{\widetilde{q}}
  \leq |\widetilde{\La}|-\sum_{j\in \widetilde{\La}} \frac{1}{\widetilde{p_j}\wedge 2}\ \text{for}\ |\widetilde{\La}|\geq 1,
  \een
  and
\ben\label{pf-M2-3}
1/\widetilde{q}\leq 1/\widetilde{q_j},\ \ \ \ (j=0,1\cdots m),
\een
\ben\label{pf-M2-4}
\frac{1}{\widetilde{p}}+\frac{m}{\widetilde{q}}\leq \sum_{j=0}^m \frac{1}{\widetilde{q_j}},
\een
where
  \be
  \widetilde{\La}:=\bigg\{j: j=0,1,\cdots,m,\ \frac{1}{\widetilde{p}}\geq 1-\frac{1}{\widetilde{p_j}\wedge 2}\bigg\}.
  \ee
Using \eqref{pf-M2-5} and the fact
\ben\label{pf-M2-6}
\frac{1}{\widetilde{p_j}\wedge 2}=\frac{1-\th}{2}+\frac{\th}{p_j\wedge 2}, \ \ \ \ j=0,1,\cdots,m,
\een
the conditions
\eqref{thm-M2-cd1},\eqref{thm-M2-cd3} and \eqref{thm-M2-cd4}
follow by
\eqref{pf-M2-1},\eqref{pf-M2-3} and \eqref{pf-M2-4}, respectively.

On the other hand, using \eqref{pf-M2-6},
we obtain
\be
\frac{1}{\widetilde{p}}\geq 1-\frac{1}{\widetilde{p_j}\wedge 2}
\Longleftrightarrow
\frac{1-\th}{2}+\frac{\th}{p}\geq 1-(\frac{1-\th}{2}+\frac{\th}{p_j\wedge 2})
\Longleftrightarrow
\frac{1}{p}\geq 1-\frac{1}{p_j\wedge 2},
\ee
which implies $\widetilde{\La}=\La$. Then, \eqref{pf-M2-2} is equivalent to
\be
  \frac{|\La|-1}{\widetilde{p}}+\frac{1}{\widetilde{q}}
  \leq |\La|-\sum_{j\in \La} \frac{1}{\widetilde{p_j}\wedge 2}\ \text{for}\ |\La|\geq 1.
\ee
The condition \eqref{thm-M2-cd2} follows by this, \eqref{pf-M2-6} and \eqref{pf-M2-5}.
\end{proof}

\subsection{Proof of Theorem \ref{thm-F2}}
Using Theorem \ref{thm-F1-sp}, we obtain the following result.
\begin{proposition}\label{pp-F2}
	Assume $p_i, q_i, p, q \in (0,\fy]$, $i=0,1,2,\cdots,m$.
	We have
	\be
	R_m: W(L^{p_0},L^{q_0})(\rd)\times\cdots \times W(L^{p_m},L^{q_m})(\rd)\longrightarrow \scrF M^{p,q}(\rmdd)
	\ee
	if and only if
    \be
	W(L^{p_0},L^{q_0})\subset \scrF M^{p,q}
	\ee
	and
  \be
   W(L^{p_j},L^{q_j})\subset M^{p,q},\ \ \  j=1,2,\cdots,m.
  \ee
\end{proposition}
Next, we give two propositions about the embedding relations.
Using Lemma \ref{lm-eb0s} and the equivalent relation in \eqref{pp-M2L-3},
we conclude the following result.
\begin{proposition}\label{pp-ebwf}
	Suppose that $0<p, q, p_0, q_0\leq \infty$. Then the embedding relation
	\ben\label{pp-ebwf-cd0}
	W(L^{p_0\wedge 2},L^{q_0})\subset \scrF M^{p,q}
	\een
	holds if and only if
	\ben\label{pp-ebwf-cd1}
	L^{p_0\wedge 2}(B_{\d})\subset \scrF L^p,
	\een
	and
	\ben\label{pp-ebwf-cd2}
	l^{q_0}\subset l^q.
	\een
	Moreover, \eqref{pp-ebwf-cd1} is equivalent to
	\be
	\frac{1}{p}\leq 1-\frac{1}{p_0\wedge 2}.
	\ee
	The condition \eqref{pp-ebwf-cd2} is equivalent to 
	\be
	\frac{1}{q}\leq \frac{1}{q_0}.
	\ee
\end{proposition}

Using Lemma \ref{lm-ebis} and the equivalent relation in \eqref{pp-M2L-3},
we obtain the following result.
\begin{proposition}\label{pp-ebwm}
  Suppose that $0<p, q, p_i, q_i\leq \infty$. Then the embedding relation
  \ben\label{pp-ebwm-cd0}
  W(L^{p_i\wedge 2},L^{q_i})\subset M^{p,q}
  \een
  holds if and only if
  \ben\label{pp-ebwm-cd1}
  L^{p_i\wedge 2}(B_{\d})\subset \scrF^{-1}L^q,
  \een
  and
  \ben\label{pp-ebwm-cd2}
  l^{q_i}\subset l^p, l^q.
  \een
  Moreover, \eqref{pp-ebwm-cd1} is equivalent to
  \be
  \frac{1}{q}\leq 1-\frac{1}{p_i\wedge 2}.
  \ee
  The condition \eqref{pp-ebwm-cd2} is equivalent to
  \be
  \frac{1}{p},\frac{1}{q}\leq \frac{1}{q_i}.
  \ee
\end{proposition}
\begin{remark}
	The full indices range of $W(L^{p_0},L^{q_0})\subset \scrF M^{p,q}$ and $W(L^{p_i},L^{q_i})\subset M^{p,q}$ can be obtained
	by using the self-improvement property of embedding relations (see Theorems \ref{thm-siebm} and \ref{thm-siebf} ).
	In fact, we have the equivalent relations
	\be
	W(L^{p_0},L^{q_0})\subset \scrF M^{p,q} \Longleftrightarrow W(L^{p_0\wedge 2},L^{q_0})\subset \scrF M^{p,q}
	\ee
	and
	\be
	W(L^{p_i},L^{q_i})\subset M^{p,q} \subset M^{p,q} \Longleftrightarrow W(L^{p_i\wedge 2},L^{q_i}).
	\ee
\end{remark}

Now, we are in a position to give the proof of Theorem \ref{thm-F2}.

\begin{proof}[The proof of Theorem \ref{thm-F2}]
	The sufficiency follows by Proposition \ref{pp-F2}, Proposition \ref{pp-ebwm}, Proposition \ref{pp-ebwf} and the fact
	$W(L^{p_i}, L^{q_i})\subset W(L^{p_i\wedge 2}, L^{q_i})$.
	
	The necessity for $p_i, q_i<\fy$ follows by Theorem \ref{thm-F1-sp}, Proposition \ref{pp-F2}, Proposition \ref{pp-ebwm} and Proposition \ref{pp-ebwf}.
	If there is some $p_i=\fy$ or $q_i=\fy$, the desired conclusion follows by an interpolation argument as in the proof of Theorem \ref{thm-M2}.
\end{proof}

\section{Return to the boundedness of pseudodifferential operators}
As mentioned in Section 1, the boundedness of pseudodifferential operator and that of Rihaczek distribution have close connection, due to the dual relation \eqref{rPR}.
In the following, we give two propositions, showing the equivalent relations between BPWM and BRWM, and that between BPWF and BRWF.
These two propositions follows by
a dual arguments of function spaces, using a similar argument as in \cite{GuoChenFanZhao2022IMRN}. We omit the proof here.

\begin{proposition}\label{pp-eqPRM}
	Assume $1\leq p, q, p_j, q_j \leq \fy$, $j=0,1,2,\cdots,m$. Then the following statements are equivalent:
	\begin{eqnarray*}
		(i)& &
		K_{\sigma}: W(L^{p_1},L^{q_1}_{\mu_1})(\rd)\times\cdots \times W(L^{p_m},L^{q_m}_{\mu_m})(\rd)\longrightarrow W(L^{p_0},L^{q_0}_{\mu_0})(\rd)
		\\
		& &\text{is bounded for any}\
		\sigma\in M^{p,q}_{\Om}(\rmdd).
		\\
		(ii)&  &\|K_{\sigma}(f_1,\cdots,f_m)\|_{W(L^{p_0},L^{q_0}_{\mu_0})(\rd)}
		\lesssim \|\s\|_{M^{p,q}_{\Om}(\rmdd)}\prod_{j=1}^m\|f_j\|_{W(L^{p_j},L^{q_j}_{\mu_j})(\rd)}
		\\
		& &\text{for any}\
		f_j\in \calS(\rd),  \sigma\in M^{p,q}_{\Om}(\rmdd),\ j=1,2,\cdots,m.
		\\
		(iii)&   &R_m: W(L^{p_0'},L^{q_0'}_{\mu_0^{-1}})(\rd)\times W(L^{p_1},L^{q_1}_{\mu_1})(\rd)\times\cdots \times W(L^{p_m},L^{q_m}_{\mu_m})(\rd)
		\longrightarrow M^{p',q'}_{\Om^{-1}}(\rmdd).
	\end{eqnarray*}
\end{proposition}

\begin{proposition}\label{pp-eqPRF}
	Assume $1\leq p, q, p_j, q_j \leq \fy$,  $j=0,1,2,\cdots,m$. Then the following statements are equivalent:
	\begin{eqnarray*}
		(i)& &
		K_{\sigma}: W(L^{p_1},L^{q_1}_{\mu_1})(\rd)\times\cdots \times W(L^{p_m},L^{q_m}_{\mu_m})(\rd)\longrightarrow W(L^{p_0},L^{q_0}_{\mu_0})(\rd)
		\\
		& &\text{is bounded for any}\
		\sigma\in \scrF M^{p,q}_{\Om}(\rmdd);
		\\
		(ii)&  &\|K_{\sigma}(f_1,\cdots,f_m)\|_{W(L^{p_0},L^{q_0}_{\mu_0})(\rd)}
		\lesssim \|\s\|_{\scrF M^{p,q}_{\Om}(\rmdd)}\prod_{j=1}^m\|f_j\|_{W(L^{p_j},L^{q_j}_{\mu_j})(\rd)}
		\\
		& &\text{for any}\
		f_j\in \calS(\rd), \sigma\in \scrF M^{p,q}_{\Om}(\rmdd),\ j=1,2,\cdots,m,
		\\
		(iii)&   &R_m: W(L^{p_0'},L^{q_0'}_{\mu_0^{-1}})(\rd)\times W(L^{p_1},L^{q_1}_{\mu_1})(\rd)\times\cdots \times W(L^{p_m},L^{q_m}_{\mu_m})(\rd)
		\longrightarrow \scrF M^{p',q'}_{\Om^{-1}}(\rmdd).
	\end{eqnarray*}
\end{proposition}

Using the above two propositions, all the boundedness results of Rihaczek distribution $R_m$ can be automatically transformed into the boundedness
of the $m$-linear pseudodifferential operator. 
Here, we do not intend to focus on stating the boundedness results of pseudodifferential operator which can be concluded directly.
We only point out that as the direct corollaries of Theorems \ref{thm-M2} and \ref{thm-F2}, the characterization for the boundedness
$K_{\sigma}: W(L^{p_1},L^{q_1})(\rd)\times\cdots \times W(L^{p_m},L^{q_m})(\rd)\longrightarrow W(L^{p_0},L^{q_0})(\rd)$
essentially extends the main results in \cite{CorderoNicola2010IMRNI}.
Here, we state the 1-linear version of BPWM as follows. 
\begin{theorem}\label{thm-BPWM1}
	Let $1\leq p,q,p_i,q_i\leq \fy$, $i=1,2$. 
	Then the boundedness
	\ben\label{thm-BPWM1-cd0}
	K_{\sigma}: W(L^{p_1},L^{q_1})(\rd)\longrightarrow W(L^{p_2},L^{q_2})(\rd)
	\een
	holds for all symbols $\sigma\in M^{p,q}(\rdd)$,
	if and only if
	\ben\label{thm-BPWM1-cd1}
	q\leq p_1\wedge 2, p_2'\wedge 2, q_1', q_2,
	\een
	and 
	\ben\label{thm-BPWM1-cd2}
	\frac{1}{p}\geq \frac{1}{q'}+\bigg(\frac{1}{p_1\wedge 2}-\frac{1}{p_2\vee 2}\bigg)\vee \bigg(\frac{1}{q_2}-\frac{1}{q_1}\bigg).
	\een
\end{theorem}
\begin{proof}
	Using Proposition \ref{pp-eqPRM} and Theorem \ref{thm-M2}, the boundedness \eqref{thm-BPWM1-cd0} holds if and only if \eqref{thm-BPWM1-cd1},
	$\frac{1}{p}\geq \frac{1}{q'}+\big(\frac{1}{q_2}-\frac{1}{q_1}\big)$ and 
	\ben\label{thm-BPWM1-1}
	\bigg(\big(\frac{1}{p_2'\wedge 2}-\frac{1}{p}\big)\vee 0\bigg)+\bigg(\big(\frac{1}{p_1\wedge 2}-\frac{1}{p}\big)\vee 0\bigg)+\frac{1}{p}-\frac{1}{q}\leq 0,\ \ \ \ \text{if}\  p\geq p_2'\wedge 2\ \text{or}\ p\geq p_1\wedge 2.
	\een
	Observe that under the condition $q\leq p_1\wedge 2, p_2'\wedge 2$,  \eqref{thm-BPWM1-1} is equivalent to
	\be
	\frac{1}{p_2'\wedge 2}-\frac{1}{p}+\frac{1}{p_1\wedge 2}-\frac{1}{p}+\frac{1}{p}-\frac{1}{q}\leq 0,\ \ \ \ \text{if}\  p\geq p_2'\wedge 2\ \text{and}\ p\geq p_1\wedge 2,
	\ee
	which is equivalent to 
	\be
	\frac{1}{p}\geq \frac{1}{q'}+\bigg(\frac{1}{p_1\wedge 2}-\frac{1}{p_2\vee 2}\bigg).
	\ee
	We have now completed this proof.
\end{proof}

The 1-linear version of BPWF is as follows.

\begin{theorem}\label{thm-BPWF1}
	Let $1\leq p,q,p_i,q_i\leq \fy$, $i=1,2$. 
	Then the boundedness
	\ben\label{thm-BPWF1-cd0}
	K_{\sigma}: W(L^{p_1},L^{q_1})(\rd)\longrightarrow W(L^{p_2},L^{q_2})(\rd)
	\een
	holds for all symbols $\sigma\in \scrF M^{p,q}(\rdd)$,
	if and only if
	\ben\label{thm-BPWF1-cd1}
	q\leq p_1\wedge 2, q_1', q_2,
	\een
	and 
	\ben\label{thm-BPWF1-cd2}
	p\leq p_2'\wedge 2, q_1'.
	\een
\end{theorem}
\begin{proof}
	Using Proposition \ref{pp-eqPRF} and Theorem \ref{thm-F2}, the boundedness \eqref{thm-BPWF1-cd0} holds if and only if 
	\be
	\frac{1}{p'}\leq 1-\frac{1}{p_2'\wedge 2},\ \ \ \frac{1}{q'}\leq \frac{1}{q_2'}
	\ee
	and
	\be
	\frac{1}{q'}\leq 1-\frac{1}{p_1\wedge 2},\ \ \ \frac{1}{p'},\frac{1}{q'}\leq  \frac{1}{q_1}.
	\ee
	Then the desired conclusion follows by a direct calculation.
\end{proof}

Next, we focus on the boundedness of pseudodifferential operator with symbols belonging to the Sj\"{o}strand's class,
from Bessel potential Wiener amalgam space $J_{s}W(L^{p_1},L^{q_1})(\rd)$
into another Wiener amalgam space $W(L^{p_2},L^{q_2})(\rd)$. Here, $J_s$ means the Bessel potential operator of order $s\in\bbR$, that is
\be
J_s=(I-\Delta)^{-s/2}.
\ee
The function space $J_{s}W(L^{p_1},L^{q_1})(\rd)$ consists of all $f\in \calS'(\rd)$ such that the norm
\be
\|f\|_{J_{s}W(L^{p_1},L^{q_1})(\rd)}
=
\|J_{-s}f\|_{W(L^{p_1},L^{q_1})(\rd)}=\|(I-\Delta)^{s/2}f\|_{W(L^{p_1},L^{q_1})(\rd)}
\ee
is finite. Observe that when $p_1=q_1=p\in (1,\fy)$, the Bessel potential Wiener amalgam space $J_{s}W(L^{p_1},L^{q_1})(\rd)=J_s L^p(\rd)$ recover the 
classical Sobolev space $L^{p}_s(\rd)$.

Although the Bessel potential Wiener amalgam space seems to beyond the scope of our main theorems in Section 1, 
there exists some equivalent relations that allow us to translate the Bessel potential problem into the BRWM we have fully studied.
See also \cite[Proposition 4.1]{Cunanan2017JoFAaA} for a similar argument.

\begin{lemma}\label{lm-eqsm1}
	Let $1\leq p_i, q_i \leq \fy$, $i=1,2$, $s\in \rr$.
	Denote $\widetilde{\sigma}(x,\xi)=\lan \xi\ran^{-s}\sigma(x,\xi)$.
	Then the following statements are equivalent:
		\begin{eqnarray*}
		(i)& &
		K_{\sigma}: J_{s}W(L^{p_1},L^{q_1})(\rd)\longrightarrow W(L^{p_2},L^{q_2})(\rd)\ 
        \text{is bounded for any}\
		\sigma\in M^{\fy,1}(\rdd);
		\\
		(ii)&  &
		K_{\widetilde{\sigma}}: W(L^{p_1},L^{q_1})(\rd)\longrightarrow W(L^{p_2},L^{q_2})(\rd)
		\text{is bounded for any}\
		\widetilde{\sigma}\in M^{\fy,1}_{v_{0,s}\otimes 1}(\rdd);
		\\
		(iii)&   &R: W(L^{p_2'},L^{q_2'})(\rd)\times W(L^{p_1},L^{q_1})(\rd)
		\longrightarrow M^{1,\fy}_{v_{0,-s}\otimes 1}(\rmdd).
	\end{eqnarray*}
\end{lemma}
\begin{proof}
	The equivalent relation $(i)\Longleftrightarrow (ii)$ follows by the equivalent relation between 
	\be
	K_{\sigma}: J_{s}W(L^{p_1},L^{q_1})(\rd)\longrightarrow W(L^{p_2},L^{q_2})(\rd)
	\ee
	and 
	\be
	K_{\widetilde{\sigma}}: W(L^{p_1},L^{q_1})(\rd)\longrightarrow W(L^{p_2},L^{q_2})(\rd),
	\ee
	and the fact that the multiplication operator mapping $\sigma(x,\xi)$ to  $\lan \xi\ran^{-s}\sigma(x,\xi)$,
	is an isomorphism from $M^{\fy,1}(\rdd)$ into  $M^{\fy,1}_{v_{0,s}\otimes 1}(\rdd)$.
	The equivalent relation $(ii)\Longleftrightarrow (iii)$ follows by Proposition \ref{pp-eqPRM}. 
\end{proof}

Using Theorem \ref{thm-M1-sp}, we have following equivalent relation.
\begin{lemma}\label{lm-eqsm2}
Let $1\leq p_i, q_i \leq \fy$, $i=1,2$, $s\in \rr$. The boundedness 
\be
R: W(L^{p_1},L^{q_1})(\rd)\times W(L^{p_2},L^{q_2})(\rd)
\longrightarrow  M^{1,\fy}_{v_{0,s}\otimes 1}(\rdd)
\ee
holds if and only if
\be
R: L^{p_1}(B_{\d})\times L^{p_2}(B_{\d})\longrightarrow M^{1,\fy}_{v_{0,s}\otimes 1}(\rdd),
\ee
for some $\d>0$,
and
\be
\tau_1\big(l^{q_1}(\zd)\otimes l^{q_2}(\zd)\big)\subset l^{1,\fy}(\zd\times\zd).
\ee
\end{lemma}

\begin{lemma}[see {\cite[Corollary 1.6]{GuoWuZhao2017JoFA}}]\label{lm-smeb}
Let $p\in [1,\fy]$, $q\in [1,2]$. Then the inequality
$\|\widehat{f}\|_{L^q_s}\lesssim \|f\|_{L^p}$ holds for all $f$ supported on $B(0,R)$
if and only if $s\leq d(1-1/(p\wedge 2)-1/q)$, with strict inequality when 
$1/q>1/(p\wedge 2)\ (p\neq 1)$ or $q\neq \fy\ (p=1)$.
\end{lemma}

In order to deal with the local boundedness, we establish the following result.

\begin{lemma}\label{lm-sm}
	Let $1\leq p_i\leq \fy$, $i=1,2$, $s\in \rr$. 
	The following statements are equivalent 
	\begin{eqnarray*}
		(i)& &
		R: L^{p_1\wedge 2}(B_{\d})\times L^{p_2\wedge 2}(B_{\d})\longrightarrow M^{1,\fy}_{v_{0,s}\otimes 1}(\rdd) 
		\ \text{holds for some }\ \d>0.
		\\
		(ii)&  &
		\|\widehat{g}\cdot \widehat{f}\|_{L^1_s}\lesssim \|g\|_{L^{p_1\wedge 2}}\|f\|_{L^{p_2\wedge 2}}
		\ \ \text{holds for all}\ \ g, f\in \calS(\rd)\ \ \text{supported on}\  B_{\d}.
		\\
		(iii)&   &
		s\leq d(1-\frac{1}{p_1\wedge 2}-\frac{1}{p_2\wedge 2})\ \text{with strict inequality when}\ p_1=1\ \text{or}\ p_2=1.
	\end{eqnarray*}
\end{lemma}
\begin{proof}
We first deal with the equivalent relation $(i)\Longleftrightarrow (ii)$.
Let $\Phi=R(\phi,\phi)$, where
 $\phi$ is a smooth function supported on $B_{4\d}$ with $\phi=1$ on $B_{2\d}$. 
By a direct calculation, we conclude that
\ben\label{lm-sm-1}
\begin{split}
\|R(g,f)\|_{M^{1,\fy}_{v_{0,s}\otimes 1}(\rdd)}
= &
\sup_{\z_1,\z_2}\|V_{\Phi}R(g,f)(z_1,z_2,\z_1,\z_2)\lan z_2\ran^{s}\|_{L^1(\rdd)}
\\
= &
\sup_{\z_1,\z_2}\|V_{\phi}g(z_1,\z_1+z_2)V_{\phi}f(\z_2+z_1,z_2)\lan z_2\ran^{s}\|_{L^1(\rdd)}
\\
\gtrsim &
\|V_{\phi}g(z_1,z_2)V_{\phi}f(z_1,z_2)\lan z_2\ran^{s}\|_{L^1(\rdd)}.
\end{split}
\een
Observe that
\be
V_{\phi}g(z_1,z_2)=\widehat{g}(z_2),\ \ V_{\phi}f(z_1,z_2)=\widehat{f}(z_2),\ \  z_1\in B_{\d}.
\ee
Then the last term of \eqref{lm-sm-1} can be dominated from below by
\be
\|\widehat{g}(z_2)\cdot \widehat{f}(z_2)\lan z_2\ran^{s}\|_{L^1(\rd)}.
\ee
This implies the relation $(i)\Longrightarrow (ii)$. 

On the other hand, if (ii) holds,
for any smooth functions $g,f$ supported on $B_{\d}$, we have
\be
\begin{split}
&
\sup_{\z_1,\z_2}\|V_{\phi}g(z_1,\z_1+z_2)V_{\phi}f(\z_2+z_1,z_2)\lan z_2\ran^{s}\|_{L^1(\rdd)}
\\
= &
\sup_{\z_1,\z_2}\|\scrF(g\overline{M_{\z_1}T_{z_1}\phi})(z_2)
\scrF(f\overline{T_{\z_2+z_1}\phi})(z_2)
\lan z_2\ran^{s}\chi_{B_{5\d}}(z_1)\|_{L^1(\rdd)}
\\
\lesssim &
\sup_{\z_1,\z_2,z_1}
\|\scrF(g\overline{M_{\z_1}T_{z_1}\phi})(z_2)
\scrF(f\overline{T_{\z_2+z_1}\phi})(z_2)
\lan z_2\ran^{s}\|_{L^1(\rd)}
\\
\lesssim &
 \|g\overline{M_{\z_1}T_{z_1}\phi}\|_{L^{p_1\wedge 2}}\|f\overline{T_{\z_2+z_1}\phi}\|_{L^{p_2\wedge 2}}
 \lesssim
 \|g\|_{L^{p_1\wedge 2}}\|f\|_{L^{p_2\wedge 2}}.
\end{split}
\ee

Next, we turn to the equivalent relation $(ii)\Longleftrightarrow (iii)$.
Take $h$ to be a smooth function supported on $B_{\d}$ with $\widehat{h}(\xi)\geq 1$ for $\xi\in B(0,1)$, 
and $\widehat{h}\geq 0$. Denote $f_{\la}(x)=g_{\la}(x)=\frac{1}{\la^d}h(\frac{x}{\la})$, $\la\in (0,1)$.
We conclude that
\be
\begin{split}
\|\widehat{g_{\la}}(\xi)\cdot \widehat{f_{\la}}(\xi)\lan \xi\ran^{s}\|_{L^1(\rd)}
\gtrsim &
\|\chi_{B(0,\la^{-1})}(\xi)\lan \xi\ran^{s}\|_{L^1(\rd)}
\gtrsim 
(\la^{-1})^{s+d}.
\end{split}
\ee
If (ii) holds, we have
\be
\begin{split}
(\la^{-1})^{s+d}
\lesssim 
\|\widehat{g_{\la}}(\xi)\cdot \widehat{f_{\la}}(\xi)\lan \xi\ran^{s}\|_{L^1(\rd)}
\lesssim
\|g_{\la}\|_{L^{p_1\wedge 2}}\|f_{\la}\|_{L^{p_2\wedge 2}}
\sim (\la^{-1})^{d(2-\frac{1}{p_1\wedge 2}-\frac{1}{p_2\wedge 2})}.
\end{split}
\ee
Letting $\la\rightarrow 0^+$, we conclude 
\be
s\leq d(1-\frac{1}{p_1\wedge 2}-\frac{1}{p_2\wedge 2}).
\ee
If $p_1=1$ and (ii) holds, for the smooth function $f$ supported on $B_{\d}$,
 we obtain
\be
\int_{\rd}|\widehat{h}(\la\xi)\widehat{f}(\xi)|\lan\xi\ran^{s}d\xi
\lesssim
\|h_{\la}\|_{L^1}\|f\|_{L^{p_2\wedge 2}}\sim \|f\|_{L^{p_2\wedge 2}}.
\ee
Letting $\la\rightarrow 0^+$, we conclude that
\be
\int_{\rd}|\widehat{f}(\xi)|\lan\xi\ran^{s}d\xi
\lesssim
\|f\|_{L^{p_2\wedge 2}}.
\ee
Using Lemma \ref{lm-smeb}, we obtain $s<-\frac{d}{p_2\wedge 2}$. 
By a similar argument, we obtain $s<-\frac{d}{p_1\wedge 2}$ for the case $p_2=1$.
This completes the proof of $(ii)\Longrightarrow (iii)$.

Finally, we verify that (iii) implies (ii).
Take
\be
s= d(1-\frac{1}{p_1\wedge 2}-\frac{1}{p_2\wedge 2})-2\epsilon,
\ee
where $\epsilon>0$ is a small positive constant for $p_1=1$ or $p_2=1$, and vanishes for other cases.
Set
\be
s_1=d(\frac{1}{2}-\frac{1}{p_1\wedge 2})-\epsilon,\ \ \ s_1=d(\frac{1}{2}-\frac{1}{p_2\wedge 2})-\epsilon.
\ee
Using Lemma \ref{lm-smeb}, we have the embedding relations
\be
\|\widehat{g}\|_{L^2_{s_1}}\lesssim \|g\|_{L^{p_1\wedge 2}},\ \ \ 
\|\widehat{f}\|_{L^2_{s_2}}\lesssim \|f\|_{L^{p_2\wedge 2}}
\ee
for smooth functions $f$ and $g$ supported on $B_{\d}$.
From this and the H\"{o}lder inequality, we conclude that
\be
\|\widehat{g}\cdot \widehat{f}\|_{L^1_s}
\lesssim
\|\widehat{g}\|_{L^2_{s_1}}
\|\widehat{f}\|_{L^2_{s_2}}
\lesssim
\|g\|_{L^{p_1\wedge 2}}
\|f\|_{L^{p_2\wedge 2}},
\ee
which completes the proof of $(iii)\Longrightarrow (ii)$.
\end{proof}

\begin{theorem}\label{thm-BPSM}
Let $1\leq p_i, q_i \leq \fy$, $i=1,2$, $s\in \rr$. Then the boundedness
\ben\label{thm-BPSM-cd0}
K_{\sigma}: J_{s}W(L^{p_1},L^{q_1})(\rd)\longrightarrow W(L^{p_2},L^{q_2})(\rd)
\een
holds for all symbols $\sigma\in M^{\fy,1}(\rdd)$,
if and only if
\ben\label{thm-BPSM-cd1}
s\geq d(\frac{1}{p_1\wedge 2}-\frac{1}{p_2\vee 2})\ \text{with strict inequality when}\ p_1=1\ \text{or}\ p_2=\fy,
\een
and 
\ben\label{thm-BPSM-cd2}
1/q_2\leq 1/q_1.
\een
\end{theorem}
\begin{proof}
Using Lemma \ref{lm-eqsm1}, the statement \eqref{thm-BPSM-cd0}
is equivalent to
\ben\label{thm-BPSM-1}
R_1: W(L^{p_2'},L^{q_2'})(\rd)\times W(L^{p_1},L^{q_1})(\rd)
		\longrightarrow M^{1,\fy}_{v_{0,-s}\otimes 1}(\rdd).
		\een
Observe that
\be
s\geq d(\frac{1}{p_1\wedge 2}-\frac{1}{p_2\vee 2})
\Longleftrightarrow
-s\leq d(1-\frac{1}{p_2'\wedge 2}-\frac{1}{p_1\wedge 2}).
\ee
We divide this proof into two parts.

\textbf{``Only if'' part.}
By a complex interpolation between \eqref{thm-BPSM-1} and the boundedness
\be
R_1: W(L^{2},L^{2})(\rd)\times W(L^2,L^2)(\rd)
		\longrightarrow M^{1,\fy}(\rdd),
\ee
we get the following boundedness result
\ben\label{thm-BPSM-2}
R_1: W(L^{\widetilde{p_2}'},L^{\widetilde{q_2}'})(\rd)\times W(L^{\widetilde{p_1}},L^{\widetilde{q_1}})(\rd)
		\longrightarrow M^{1,\fy}_{v_{0,-\widetilde{s}}\otimes 1}(\rdd),
\een
where
\be
\frac{1}{\widetilde{p_i}}=\frac{1-\theta}{2}+\frac{\theta}{p_i},\ 
\frac{1}{\widetilde{q_i}}=\frac{1-\theta}{2}+\frac{\theta}{q_i},\ 
\widetilde{s}=\theta s,\ \ \ i=1,2.
\ee
Applying Theorem \ref{thm-M1-sp} on the boundedness \eqref{thm-BPSM-2}, and using Lemma \ref{lm-eqsm2} and Lemma \ref{lm-sm}, we obtain that
\be
-\widetilde{s}\leq d(1-\frac{1}{\widetilde{p_2}'\wedge 2}-\frac{1}{\widetilde{p_1}\wedge 2}),
\ee
which is equivalent to
\be
-\widetilde{s}\leq d(\frac{1}{\widetilde{p_2}\vee 2}-\frac{1}{\widetilde{p_1}\wedge 2})
\Longleftrightarrow
s\geq d(\frac{1}{p_1\wedge 2}-\frac{1}{p_2\vee 2}).
\ee

Next, we deal with the critical case $p_1=1$ or $p_2'=1$.
The cases $p_1=1,\ p_2'<\fy$ and $p_1<\fy,\ p_2'=1$ can be verified 
by using Theorem \ref{thm-M1-sp}, Lemma \ref{lm-eqsm2} and Lemma \ref{lm-sm}.

If $p_1=1$, $p_2'=\fy$, by a similar argument as in the proof of Lemma \ref{lm-sm}, we obtain
\be
\int_{\rd}|\widehat{f}(\xi)|\lan\xi\ran^{-s}d\xi
\lesssim
\|f\|_{L^{\fy}}\ \ \text{for any smooth function}\ f\ \ \text{supported on}\ B_{\d}.
\ee
From this and Lemma \ref{lm-smeb}, we get $s>\frac{d}{2}=d(\frac{1}{p_1\wedge 2}-\frac{1}{p_2\vee 2})$.
If $p_1=\fy$, $p_2'=1$, we can also conclude $s>\frac{d}{2}=d(\frac{1}{p_1\wedge 2}-\frac{1}{p_2\vee 2})$ by a similar argument.

We have now verified the necessity of condition \eqref{thm-BPSM-cd1}.
Using \eqref{thm-BPSM-1}, Lemma \ref{lm-eqsm2} and
the fact that
\ben\label{thm-BPSM-3}
\tau_1(l^{q_2'}(\zd)\otimes l^{q_1}(\zd))\subset l^{1,\fy}(\zd\times \zd)\Longleftrightarrow 1\leq \frac{1}{q_2'}+\frac{1}{q_1},
\een
from Proposition \ref{pp-M2G}, we obtain $1\leq \frac{1}{q_2'}+\frac{1}{q_1}$, which is
equivalent to \eqref{thm-BPSM-cd2}.

\bigskip

\textbf{``If'' part.}
If the conditions \eqref{thm-BPSM-cd1} and \eqref{thm-BPSM-cd2} hold, the boundedness \eqref{thm-BPSM-1} follows by
Lemma \ref{lm-eqsm2},
Lemma \ref{lm-sm}, \eqref{thm-BPSM-3}  and the facts that
\be
W(L^{p_2'},L^{q_2'})(\rd)\subset W(L^{p_2'\wedge 2},L^{q_2'})(\rd),\ \ \ 
W(L^{p_1},L^{q_1})(\rd)\subset W(L^{p_1\wedge},L^{q_1})(\rd).
\ee
\end{proof}

As a direct corollary, we give an essential extension of the main result in \cite{Cunanan2017JoFAaA}.
\begin{corollary}\label{cy-BPSM}
Let $1\leq p_i\leq \fy$, $i=1,2$, $s\in \rr$. Then the boundedness
\be
K_{\sigma}: J_{s}L^{p_1}(\rd)\longrightarrow L^{p_2}(\rd)
\ee
holds for all symbols $\sigma\in M^{\fy,1}(\rdd)$,
if and only if
\be
s\geq d(\frac{1}{p_1\wedge 2}-\frac{1}{p_2\vee 2})\ \text{with strict inequality when}\ p_1=1\ \text{or}\ p_2=\fy,
\ee
and 
\be
1/p_2 \leq 1/p_1.
\ee
\end{corollary}

\subsection*{Acknowledgements} Supported by the National Natural Science Foundation of China [12371100] and Fujian Provincial Natural Science Foundation of China [2024J01727,2024J011196,2022J011241].

\bibliographystyle{abbrv}

\end{document}